\documentclass[a4paper, 11pt]{article}

\usepackage {amsmath,amsthm}
\usepackage{amsfonts}
\usepackage{amssymb}
\usepackage{bbm}
\usepackage{color}
\usepackage{yfonts}
\usepackage{mathtools}
\usepackage{natbib}
\usepackage{enumerate}
\usepackage{hyperref}

\newcommand{\supp}{\mathrm{supp}}

\newcommand{\x}{\times}
\newcommand{\p}{\partial}

\renewcommand{\d}{\mathrm{d}}
\newcommand{\e}{\mathrm{e}}

\newcommand{\irm}{\mathrm{i}}

\newcommand{\Lap}{\mathrm{\Delta}}

\newcommand{\eqnb}{\begin{equation}}
\newcommand{\eqne}{\end{equation}}
\newcommand{\loc}{\mathrm{loc}}
\newcommand{\re}[1]{(\ref{#1})}
\newcommand{\comment}[1]{}
\newcounter{thm}
\numberwithin{thm}{section}
\numberwithin{equation}{section}

\newcommand{\lewy}{\left\lbrace}
\newcommand{\prawy}{\right\rbrace}
\newcommand{\RR}{\mathbb{R}}
\newcommand{\QQ}{\mathbb{Q}}

\newcommand{\NN}{\mathbb{N}}

\newtheorem{theorem}[thm]{Theorem}%[section]
\newtheorem{lemma}[thm]{Lemma}%[section]
\newtheorem{fact}[thm]{Fact}%[section]
\newtheorem{remark}[thm]{Remark}%[section]
\newtheorem{corollary}[thm]{Corollary}%[section]
\newtheorem{definition}[thm]{Definition}

\renewcommand \thesection{\arabic{section}}%########################

\begin{document}
\title{Leray's fundamental work on the Navier--Stokes equations: a modern review of \emph{``Sur le mouvement d'un liquide visqueux emplissant l'espace''}}
\author{Wojciech S. O\.za\'nski, Benjamin C. Pooley}
\maketitle
\begin{abstract}
This article offers a modern perspective which exposes the many contributions of Leray in his celebrated work on the Navier--Stokes equations  from 1934. Although the importance of his work is widely acknowledged, the precise contents of his paper are perhaps less well known. The purpose of this article is to fill this gap. We follow Leray's results in detail: we prove local existence of strong solutions starting from divergence-free initial data that is either smooth, or belongs to $H^1$, $L^2\cap L^p$ (with $p\in(3,\infty]$), as well as lower bounds on the norms $\| \nabla u (t) \|_2$, $\| u(t) \|_p$ ($p\in(3,\infty]$) as $t$ approaches a putative blow-up time. We show global existence of a weak solution and weak-strong uniqueness. We present Leray's characterisation of the set of singular times for the weak solution, from which we deduce that its upper box-counting dimension is at most $\tfrac{1}{2}$. Throughout the text we provide additional details and clarifications for the modern reader and we expand on all ideas left implicit in the original work, some of which we have not found in the literature. We use some modern mathematical tools to bypass some technical details in Leray's work, and thus expose the elegance of his approach. 
\end{abstract}
%\tableofcontents
%#######################################################
\section{Introduction}
The Navier--Stokes equations, 
\[
\p_t u+(u\cdot\nabla)u-\nu\Lap u +\nabla p =0,
\]
\[
\nabla\cdot u=0,
\]
where $u$ denotes the velocity of a fluid, $p$ denotes  the scalar pressure and $\nu>0$ denotes viscosity of the fluid, 
comprise the fundamental model for the flow of an incompressible viscous fluid. They are named in recognition of Claude-Louis \cite{Navier_1822} and George \cite{Stokes_1845} who first formulated them, and they form the basis for many models in engineering and mathematical fluid mechanics. The equations have been studied extensively and a number of excellent textbooks on the subject are available, see for instance \cite{ConstE_NS}, \cite{ConstFoias}, \cite{ L-R_Book},  Robinson, Rodrigo \& Sadowski (2016), \cite{ Sohr_NSE_Book} and \cite{ Temam_NSEnNumerics_Book}. However, the fundamental issue of the well-posedness of the equations in three dimensions remains unsolved.

In this article we focus solely on the work of Jean \cite{Leray_1934}, which to this day remains of fundamental importance in the study of the Navier--Stokes equations. Leray was the first to study the Navier--Stokes equations in the context of \emph{weak solutions}. It is remarkable that such solutions are defined using a distributional form of the equations while the distribution theory was only formally introduced later by \cite{Schwartz_1950}. 

Many of the ideas in the modern treatment of these equations and a number of other systems originate from  Leray's (1934) paper (which we shall often refer to simply as ``Leray's work'' or ``Leray's paper''). The importance of that work is witnessed by the fact that it is one of the most cited works in mathematical fluid mechanics.

Leray studied the Navier--Stokes equations on the whole space ($\RR^3$). Unlike later authors who have largely adopted Faedo-Galerkin techniques (see \cite{Hopf_1951} and \cite{Kiselev_Ladyzhenskaya} for early examples), a characteristic of earlier works, including Leray's, is the use of explicit kernels. In this instance, Leray applied the Oseen kernel (herein denoted by $\mathcal{T}$), as derived by \cite{Oseen_1911}, to obtain solutions of the \emph{inhomogeneous Stokes equations}:
\[
\p_t u-\nu\Delta u +\nabla p =F,\qquad \nabla\cdot u=0,
\]
for a given forcing $F$. At the time these were also known as the equations for \emph{infinitely slow motion}. 

Oseen had previously applied this kernel iteratively to prove local well-posedness of the Navier--Stokes equations in $C^2(\RR^3)$ (with bounded velocity, decay conditions on $\omega=\nabla\x u$, and polynomial growth estimtes on $\nabla\omega$ and $\nabla u$; see Section 3.8 of \cite{L-R_21st_century} for a more complete discussion of Oseen's contributions). Leray applied a more elegant iteration scheme (a Picard iteration) to prove existence, uniqueness, and regularity results for local-in-time strong solutions for initial data $u_0\in H^1\cap L^\infty \cap C^1$ (Leray used the term \emph{regular solutions}). We will see that in fact, $u_0\in L^2\cap L^\infty$ suffices to construct strong solutions when his arguments are rewritten using a distributional form of equations.

Leray then derives lower bounds on various norms of the strong solution $u(t)$ as $t$ approaches the maximal time of existence $T_0$ if $T_0$ is finite, which indicate the rate of blow-up of a strong solution if such a blow-up occurs. He leaves open the issue of existence of blow-ups.

Next, Leray considers a generalised notion of solution of the Navier--Stokes equations, \emph{weak solutions}. For this he studies the so-called \emph{regularised equations}, which are obtained by replacing the nonlinear term $(u\cdot \nabla )u$ by $(J_\varepsilon u \cdot \nabla )u$ (where $J_\varepsilon$ is the standard mollification operator). He shows that the regularised equations admit local well-posedness results similar to those for strong solutions of the Navier--Stokes equations, with an additional global-in-time estimate on the $L^\infty$ norm of the velocity. This extra property results in global-in-time existence and uniqueness of solutions $u_\varepsilon$ for each $\varepsilon >0$. By a careful compactness argument, he constructs a sequence of solutions $\{ u_{\varepsilon_n} \}$ converging to a global-in-time weak solution to the Navier--Stokes equations. Such solutions, which he termed \emph{turbulent solutions}, can be thought of as weak continuations of the strong solution beyond the blow-up time, a revolutionary idea at the time. He then shows that these weak solutions are strong locally-in-time except on a certain compact set of singular times with Lebesgue measure zero. To this end he uses a certain \emph{weak-strong uniqueness} property. 

We now briefly highlight a few important developments that proceeded from Leray's work.
Eberhard \cite{Hopf_1951} performed a study of the Navier--Stokes equations on the bounded domain $\Omega \subset \RR^3$, and proved global-in-time existence of weak solutions. Then \cite{ladyzhenskaya1959} proved existence and uniqueness of global-in-time strong solutions to the Navier--Stokes equations in bounded two-dimensional domains (\cite{Leray_1933} dealt with well-posedness in $\RR^2$ in his thesis but was less successful in studying the case of bounded domains \cite{Leray_1934_essai}). \cite{fujita_kato1964} used fractional powers of operators and the theory of semigroups to construct local-in-time strong solutions of the (three-dimensional) Navier--Stokes equations on bounded domains (\cite{kato1984} used similar methods in the case of unbounded domains). As for the smoothness of weak solutions, it follows from the work of \cite{serrin1962},  \cite{Prodi_1959}, and \cite{Ladyzhenskaya_1967}, that if a weak solution $u$ belongs to $ L^r ((a,b);L^s )$ with $2/r+3/s\leq1$, $s>3$ then $u$ is a strong solution on the time interval $(a,b)$; the critical case $r=\infty,\ s=3$ was proved by \cite{Escauriaza_S_S_2003}. In addition, \cite{beale_kato_majda1984} showed that if $\mathrm{curl}\, u \in L^1 ((a,b);L^\infty )$ then $u$ is a strong solution on the time interval $(a,b]$.

 \cite{scheffer1977} was the first to study the size of the singular set in both time and in space. Subsequently, Caffarelli, Kohn \& Nirenberg (1982)\nocite{CKN} proved that the one-dimensional Hausdorff measure of the singular set is zero. We refer the reader to the textbooks above for a wider description of the contributions to the theory of the Navier--Stokes equations in the last 80 years.

 It is remarkable that despite many significant contributions, it is still not known whether the (unique) local-in-time strong solutions to the Navier--Stokes equations develop blow-ups or whether the global-in-time weak solutions are unique. This remains one of the most important open problems in mathematics, at the turn of the millennium, it was announced as one of seven Millennium Problems, see \cite{Fefferman_Clay}.

A number of concepts and methods that found early use in Leray's work are now ubiquitous in the analysis of PDEs. These include: weak compactness of bounded sequences in $L^2$, the concept of weak derivatives (called \emph{quasi-derivatives} by Leray), the mollification operation and the fact that a weakly convergent sequence converges strongly if and only if the norm of the limit is the limit of the norms. He made an extensive use of the space of $L^2$ functions with weak derivatives in $L^2$ two years before the celebrated work of \cite{sobolev}. This space would later be called Sobolev space $H^1$. Furthermore, Leray was the first to introduce the compactness method of solving partial differential equations (see the proof of Theorem \ref{existence_weak}); in fact this, together with his work with Juliusz Schauder opened a new branch in mathematics, the use of topological methods in PDEs. 

The terms \emph{Leray weak solution} (or \emph{Leray-Hopf weak solution}) of the Navier--Stokes equations (see Definition \ref{def_weak_sol}), \emph{Leray regularisation} (see \eqref{regularised_NSE}) and \emph{Leray projection} (for the projection of $L^2$ onto the space of weakly divergence-free vector fields) %see comments following Theorem \ref{classical_sol_if_X_smooth_theorem})
 have become part of the mathematical lexicon in recognition of his seminal paper on the Navier--Stokes equations. 
We refer the interested reader to \cite{Lax_about_Leray} for a broader description of Leray's work in the field of PDEs.

This article arose from series of lectures presented by the authors for a fluid mechanics reading group organised at the University of Warwick by James Robinson and Jos\'e Rodrigo, and its purpose is to offer a modern exposition of Leray's work. We update the notation and we simplify some technical details by applying some modern methods; in particular we use the Fourier transform (see Theorem \ref{classical_sol_if_X_smooth_theorem}) and the distributional forms of the partial differential equations appearing in Leray's work. It is perhaps remarkable that these updates do not detract from the originality of Leray's arguments; rather they make them even more elegant.

We have also endeavoured to give a rigorous account of all non-trivial results that are left implicit in the original work, some of which we were not able to find in the literature. These include Leray's derivations of the blow-up rate of the norm $\| u(t) \|_p$ (with $p>3$) of a strong solution $u$ as $t$ approaches the putative blow-up time (see Corollary \ref{blowup_rates_corollary}), and a result on local existence of strong solutions for initial data $u_0 \in L^2 \cap L^p$ (with $p\in (3,\infty ]$) that is weakly divergence free (see Corollary \ref{exist_semireg_Linfty_or_Lp}). In order to make the exposition self-contained we have also added appendices on relevant facts from the theory of the heat equation, integral inqualities, the Volterra equation and other topics.

For simplicity of notation we focus only on the case $\nu =1$. The corresponding results for any $\nu >0$ can be recovered using the following rescaling argument: if $u$, $p$ is a solution of the Navier--Stokes equations with $\nu= 1$, then $\widetilde{u}(x,t)\coloneqq u(x \nu,t \nu)$, $\widetilde{p} \coloneqq p (x\nu , t\nu )$ is a solution for given $\nu>0$.

The structure of the article is as follows. In the remainder of this section we describe some notation, we recall some preliminary results, and we introduce the Oseen kernel $\mathcal{T}$, which will be the main tool for solving the Stokes equations.

In Section \ref{stokes_eq_section} we study the Stokes equations. We first show that if the forcing $F$ is sufficiently smooth (see \eqref{regularity_of_X_requirement}) the equations can be solved classically using the representation formulae
\[\begin{split}
u(t)&\coloneqq  \Phi (t)\ast u_0  +\int_0^t  \mathcal{T}(t-s) \ast F(s)\, \d s ,  \\
p(t)&\coloneqq  -(-\Delta) ^{-1} ( \mathrm{div} \, F(t) ) ,
\end{split} \]
see \eqref{repr_formula}, \eqref{repr_formula_p}, where $\Phi (t)$ denotes the heat kernel (Theorem \ref{classical_sol_if_X_smooth_theorem}). We also show some further properties of the representation formula for $u$ in the case of less regular forcing $F$ (Lemma \ref{prop_of_up_lemma}). We then focus on a special form of the forcing 
\[F=-(Y\cdot \nabla )Y ,\] 
see \eqref{form_of_X}, which is reminiscent of the nonlinear term in the Navier--Stokes equations. This special form of $F$ gives rise to the modified representation formulae,
\[ \begin{split}
u(t) &\coloneqq    \Phi (t)\ast u_0  +\int_0^t  \nabla  \mathcal{T}(t-s) \ast \left[ Y(s) Y (s) \right] \,  \d s ,\\
p(t) & \coloneqq  \p_k \p_i (-\Delta )^{-1} (Y_i (t) Y_k(t)) ,
\end{split}\]
see \eqref{repr_formula1}, \eqref{repr_formula1_p}. After deducing some properties of this modified representation formula (Lemma \ref{additional_prop_up_lemma}) we show that it gives a unique solution (in some wide sense) to the Stokes equations with the forcing $F$ of the form above,  in the sense of distributions (Theorem \ref{uniqueness_stokes} and Theorem \ref{sol_distributions_anyX_theorem}).

In Section \ref{sec_strong} we study strong solutions of the Navier--Stokes equations. After defining strong solutions on an open time interval $(0,T)$ (Definition \ref{def_strong_open}) we use the theory of the Stokes equations developed in Section \ref{stokes_eq_section} to deduce the smoothness of such solutions (Corollary \ref{smoothness_reg_sol_corollary}), as well as other interesting properties, such as the energy equality
\[
\| u(t_2) \|^2 + 2 \int_{t_1}^{t_2} \| \nabla u(s) \|^2 \, \d s = \| u(t_1) \|^2, 
\]
 (Theorem \ref{EE_open_int_thm}) and the comparison of strong solutions 
\[
\| (u-v)(t_2)\|^2 \leq  \| (u-v)(t_1) \|^2 \mathrm{e}^{\frac{1}{2} \int_{t_1}^{t_2}  \| u(s) \|^2_\infty \,\d s} , \quad t_1<t_2
\] 
 (Lemma \ref{comparison_two_sols_lem}). We then define strong solutions on a half-closed time interval $[0,T)$ (Definition \ref{def_reg_sol_halfclosed}) and show local-in-time existence and uniqueness of such solutions with weakly divergence-free initial data $u_0\in L^2 \cap L^\infty$ (Theorem \ref{local_exist_strong_thm}). Next, we discuss the issue of the maximal time of existence $T_0$ of strong solutions (Lemma \ref{bounds_on_infty_norm_lemma} and Corollary \ref{existence_time_corollary}), from which we deduce the rates of blow-up of $u(t)$ in various norms as $t$ approaches $T_0$ (if $T_0$ is finite):
 \[
 \| u(t) \|_\infty \geq \frac{C}{\sqrt{T_0-t}},\quad \|\nabla u(t) \| \geq \frac{C}{(T_0-t)^{1/4}},
 \]
 and
 \[\| u(t) \|_p \geq \frac{C^{(1-3/p)/2}(1-3/p)}{(T_0-t)^{(1-3/p)/2}}
 \]
 (Corollary \ref{blowup_rates_corollary}). This study of strong solutions is concluded with an observation that less regular initial data $u_0$ also gives rise to a unique strong solution on the time interval $(0,T)$ for some $T>0$. This motivates the definition of \emph{semi-strong solutions} (Definition \ref{def_semi_reg}); we show that if $u_0 \in L^2$ is weakly divergence free, and either 
 \[ \nabla u_0 \in L^2 \quad \text{ or } \quad u_0 \in L^p  \quad \text{(with }p>3\text{),}
 \]
then there exists a unique local-in-time semi-strong solution starting from initial data $u_0$ (Theorem \ref{semi_reg_sols_existence_thm} and Corollary \ref{exist_semireg_Linfty_or_Lp}).

In Section \ref{sec_weak} we study weak solutions of the Navier--Stokes equations. To this end, for each $\varepsilon >0$ we consider regularised equations, where a mollification operator is applied in the nonlinear term,
\[
\p_t u - \Delta u + ((J_\varepsilon u) \cdot \nabla ) u + \nabla p = 0
\]
(Definition \ref{def_of_sol_regularised_eqs}). We show that for each $\varepsilon >0 $ the regularised equations can be analysed in a similar way to strong solutions of the Navier--Stokes equations in Section \ref{sec_strong}, the difference being that the maximal time of existence of the solution $u_\varepsilon$ is infinite (Theorem \ref{thm5.1}). In order to take the limit $\varepsilon \to 0^+$ we first show that the kinetic energy of $u_\varepsilon (t)$ outside a ball can be estimated independently of $\varepsilon $,
\[
 \int_{|x|>R_2} |u_\varepsilon (t)|^2 \,\d x \leq  \int_{|x|>R_1} |u_0 |^2 \,\d x + \frac{C(u_0,t)}{R_2-R_1},\qquad R_2>R_2>0,
\]
(Lemma \ref{tail_estimate}). Thanks to this separation of energy result we can let $\varepsilon \to 0^+$ (along a carefully chosen subsequence) to obtain a global-in-time weak solution of the Navier--Stokes equations (Theorem \ref{existence_weak}). We then show the so-called weak-strong uniqueness result (Lemma \ref{weak_strong_uniqueness_lem}) and we deduce that the weak solution admits a particular structure, namely that it is (locally) a strong solution at times
\[
t\in  \bigcup_{i} (a_i , b_i)
\]
where the intervals $(a_i,b_i)\subset (0,\infty )$ are pairwise disjoint (Theorem \ref{thmEpochs}). Finally we show that the complement of their union
\[
\Sigma \coloneqq (0,\infty ) \setminus \bigcup_{i} (a_i , b_i),
\]
(the set of putative singular times) is bounded, its box-counting dimension is bounded above by $1/2$ and that the weak solution admits certain decay for sufficiently large times (Theorem \ref{thmSupInfo} and Corollary \ref{corBoxDim}).

Unless specified otherwise, each proof follows \cite{Leray_1934}, possibly with minor modifications. We also comment on Leray's methodology throughout the text, in footnotes and in the ``Notes'' at the end of each section. Equation numbers marked in italics correspond to expressions in Leray's paper.

%#######################################################
\subsection{Preliminaries}\label{sec_prelims}
The letter $C$ denotes a numerical constant, whose value may change at each occurrence. Occasionally we write $C'$ (or $C''$) to denote a constant that has the same value wherever it appears within a given section. Also, $C_m$ (and $c_m$) denotes a numerical constant for each $m$. Throughout the article (unless specified otherwise) we consider function spaces on $\RR^3$, for example $L^p \coloneqq L^p (\RR^3)$, $H^m \coloneqq H^m (\RR^3 )$, $C_0^\infty \coloneqq C_0^\infty (\RR^3 )$. We also define $\int \coloneqq \int_{\RR^3 }$, $\| \cdot \|_{p} \coloneqq \| \cdot \|_{L^p}$ and we reserve the notation $\| \cdot \|$ for the $L^2$ norm. We say that a velocity field $v$ is \emph{weakly divergence free} if $\mathrm{div} \, v =0$ in the sense of distributions, that is
\eqnb\label{weakly_div_free_def}
\int v \cdot \nabla g =0 
\eqne
for all $g\in C_0^\infty$, and we set
\eqnb\label{def_of_HV}
\begin{split}
H &\coloneqq \{ f\in L^2 \, : \,  f \text{ is weakly divergence free} \}.\\
V &\coloneqq \{ f\in H^1 \colon \mathrm{div}\, f =0 \}. 
\end{split} 
\eqne
We understand $\QQ^+ $ as the nonnegative rational numbers and we define $\p_j \coloneqq \p /\p x_j$ and $\nabla^m \coloneqq D^m$, where we understand all derivatives in the weak sense. We use the convention of summing over repeated indices. For example, we write $v_j \p_j u_i$ to denote the vector $(v\cdot \nabla ) u$ (here $i=1,2,3$). For an interval $I$ define
\[ \begin{split}
\mathcal{H}^{1/2}(I) \coloneqq &\{f  : \RR^3 \times I  \, :\,  \exists C(t) \text{, a continuous function on }I \text{, such that } \\
&|f(x,t)-f(y,t)|\leq C(t)|x-y|^{1/2} \mbox{ for all } t\in I,\, x,y\in \RR^3 \},
\end{split}\]
that is $\mathcal{H}^{1/2}(I)$ is the space of functions such that $\| f(t) \|_{C^{0,1/2 }}\leq C(t)$ for some continuous $C(t)$ (where $\| \cdot \|_{C^{0,1/2 }}$ denotes $1/2$-H\"older seminorm). Note that $\mathcal{H}^{1/2}(I)$ is defined in the same way for vector-valued functions.

We recall the \emph{integral Minkowski inequality}
\begin{equation}\label{minkowski_integral_version}
\| f(t) \|_{p} \leq \int_0^t \| g(s) \|_{p}\, \d s .
\end{equation}
whenever $p\in [1,\infty ]$ and $f(x,t)$ is of the form $f(x,t)= \int_0^t g(x,s) \d s$.
More generally, for integrable and nonnegative $\xi $, $\eta$, if
\[
F(x,t)= \int_\RR \xi (s) \int \eta (y) f(x,y,t,s)\,  \d y\,  \d s
\]
then
\begin{equation}\label{minkowski_integral_version_more_general}
\| F( t) \|_{p} \leq \int_\RR \xi (s) \int \eta (y ) \| f(\cdot , y , t, s) \|_{p}\,  \d y\,  \d s .
\end{equation}
Now let $p,q,r>0$ satisfy $1/q=1/p+1/r-1$ and let $f\in L^p$, $g\in L^r$. Then
\begin{equation}\label{young_for_convolutions}
\| f \ast g \|_{q} \leq \| f \|_{p} \| g \|_{r}.
\end{equation}
This is \emph{Young's inequality for convolutions} (see e.g. \cite{stein_weiss_introduction}, p. 178, for the proof). Here ``$\ast$'' denotes the \emph{convolution}, that is 
\[ (f\ast g)(x) \coloneqq \int f(x-y) g(y)\, \d y. \]
If $f$ and $g$ are also functions of time $t$ we omit $x$ and simply write $u(t) = f(t) \ast g(t)$. We apply this notation in the statement of the following extension of Young's inequality to the case of space-time convolutions, whose proof we give in Appendix \ref{young_spacetime_continuity_section} (see Lemma \ref{young_spacetime_continuity_appendix}).
\begin{lemma}\label{young_spacetime_continuity}
If $p,q,r \geq 1$ satisfy 
\[\frac{1}{q}=\frac{1}{p}+\frac{1}{r}-1,\]
$A\in L^1_{\loc}([0,T);L^p)$ and $B\in C((0,T);L^r)$ with $\| B(t) \|_r$ bounded as $t \to 0^+$ then $u$ defined by
\[
u(t)\coloneqq \int_0^t  A(t-s) \ast B(s) \, \d s
\]
belongs to $C([0,T);L^q)$ and
\[
\| u (t) \|_q \leq \int_0^t \| A(t-s) \|_p \| B(s) \|_r \,\d s.
\]
\end{lemma}
Let $J_\varepsilon$ denote the standard \emph{mollification} operator, that is
\[ J_\varepsilon v  \coloneqq   \eta_{\varepsilon } \ast v ,\]
where $\eta_\varepsilon$ is a standard mollifier, e.g. $\eta_\varepsilon (x) \coloneqq  \varepsilon^{-3} \eta (x/\varepsilon )$, where $\eta (x) \coloneqq C \exp ((|x|^2-1)^{-1})$ for $|x|< 1$ and $\eta (x)\coloneqq  0$ for $|x| \geq 1$ with the constant $C>0$ chosen such that $\int \eta =1$.
\begin{lemma}[properties of mollification]\label{prop_molli}
The mollification operator $J_\varepsilon$ (on $\RR^3$) enjoys the following properties:
\begin{enumerate}[\rm(i)]
\item $\| J_\varepsilon v \|_{p} \leq \|  v \|_{p} $ for all $p\in [1,\infty ]$, $\varepsilon >0$,
\item $\p_{k} J_\varepsilon v = J_\varepsilon (\p_{k} v)$ for every $k=1,2,3$ whenever $\nabla v \in L^1_{\mathrm{\loc}}$,
\item $\| J_\varepsilon v \|_{{\infty }} \leq C \varepsilon^{-3/2} \| v\|$,
\item if $v\in L^2$ then $J_\varepsilon v \in H^m $ for all $m$ with $\| J_\varepsilon v \|_{H^m} \leq C_m \varepsilon^{-m} \| v\|$,
\item if $v\in L^1_{\loc}$ then $J_\varepsilon v \to v$ almost everywhere as $\varepsilon \to 0$,
\item if $v\in L^p$, where $p\in [1,\infty )$, then $J_\varepsilon v \to  v$ in $L^p$ as $\varepsilon \to 0$.
\end{enumerate}
\end{lemma}
The proofs of the above properties are elementary (and can be found in Appendix C in \cite{Evans_PDEBook}, Section 3.5.2 in \cite{MajdaBertozzi} or Appendix A.3 in \cite{JCR_R_S_NSE_book}).

 For $f\in L^2$ define %TODO: L^2 might not be enough decay
\begin{equation}\label{def_of_inverse_laplace}
(-\Delta )^{-1} f (x) \coloneqq \int \frac{f(y)}{4\pi |x-y|}\, \d y .
\end{equation}
The symbol $(-\Delta )^{-1}$ relates to the fact that $g\coloneqq (-\Delta )^{-1}f$ satisfies the Poisson equation $-\Delta g =f$ in $\RR^3$ in the sense of distributions (which follows by an application of Fubini's theorem). Since we will often estimate terms similar to the right-hand side of \eqref{def_of_inverse_laplace}, we formulate the following lemma (which is Leray's {\it(1.14)}).
\begin{lemma}\label{temp_singular_int_bound} If $f\in H^1$ then for any $y\in \RR^3$
\begin{equation}\label{1.14}
\int\frac{|f(x)|^2}{|x-y|^2} \,\d x \leq 4\|\nabla f\|^2.
\end{equation}
\end{lemma}
\begin{proof} It is enough to prove the claim when $y=0$ and when $f$ is a scalar function. If $f\in C^\infty_0$ then integration by parts, the Cauchy--Schwarz inequality and Young's inequality give
\[
\begin{split}
\int \frac{(f(x))^2}{|x|^2} \,\d x &= - \int  \frac{x}{|x|^2} \cdot \nabla (f(x))^2  \,\d x= -2 \int \frac{x}{|x|^2} \cdot  \nabla f(x)   f(x) \, \d x\\
&\leq  2\| \nabla f \|  \sqrt{ \int \frac{(f(x))^2}{|x|^2} \,\d x } \leq 2 \| \nabla f \|^2 +\frac{1}{2} \int \frac{(f(x))^2}{|x|^2} \,\d x .
\end{split}
\]
For $f\in H^1 $ the claim follows from the density of $C_0^\infty $ functions in $H^1 $ and Fatou's lemma.
\end{proof}
\begin{lemma}[The Plancherel Lemma]\label{CZ_lemma}
The operator
\[
f\mapsto \p_i \p_k (-\Delta )^{-1} f 
\]
is a bounded operator from $L^2$ to $L^2$ for every $i,k$. Consequently an application of Fubini's theorem gives that $\p_i \p_k (-\Delta )^{-1}$ is a bounded operator from $H^m$ to $H^m$ for every $m\geq 0$.
\end{lemma}
\begin{proof}
This follows by considering the Fourier transform and using the Plancherel property:
\[
\| \p_i \p_k (-\Delta )^{-1} f \|^2 = \int \left| \frac{\xi_i\xi_k}{|\xi |^2} \widehat{f} (\xi ) \right|^2 \d \xi \leq  \int \left|  \widehat{f} (\xi ) \right|^2 \d \xi = \| f \|^2,  
\]
where $\widehat{f}$ denotes the Fourier transform of $f$.
\end{proof}
Consider the \emph{heat equation} in $\RR^3 \times (0,T)$:
\[
v_t - \Delta v =0
\]
with initial condition $v(0)=v_0$ (understood in the sense of $L^2$ limit as $t\to 0^+$) for some $v_0 \in L^2$.
Then a classical solution $v$ of the heat equation is given by the convolution
\[
v(t) =  \Phi (t) \ast v_0,
\]
where 
\eqnb\label{def_of_heat_kernel}
\Phi (x,t):= \frac{1}{(4\pi t)^{3/2}} \mathrm{e}^{-|x|^2/4t}\eqne is the \emph{heat kernel}. In what follows, we will rely on some well-known properties of the heat equation and the heat kernel, which we discuss in Appendix \ref{heat_section_appendix}.

Finally, we will often use an integral version of an elementary fact from the theory of ordinary differential equations: if $f, \phi \colon [0,T) \to \RR^+$ are $C^1$ functions such that
\[
\begin{cases}
f' \leq g\,  f^k\,+a,\\
\phi' \geq g \, \phi^k + b
\end{cases} \text{ on } [0,T) \qquad \text{ with }\quad  f(0) < \phi (0),
\]
where $k>0$ and $g,a,b$ are continuous functions such that $g>0$ and $a\leq b$, then $f< \phi$ on $[0,T)$. In particular, we will use the following result, which corresponds to the case $k=2$.
\begin{lemma}[Integral inequalities]\label{lem_integral_ineq}
Suppose $g>0$ is a continuous function on $(0,T)$ that is locally integrable $[0,T )$, that functions $f,\phi: (0,T)\to \RR^+$ satisfy
\[
\begin{split}
f(t) &\leq  \int_0^t g(t-s) f(s)^2 \,\d s + a(t) ,  \\
\phi (t) &\geq  \int_0^t g(t-s) \phi (s)^2 \,\d s + b(t) 
\end{split}\]
for all $t\in (0,T)$, where $a$, $b$ are continuous functions satisfying $a\leq b$, $\phi$ is continuous, and that $f^2$ and $\phi^2$ are integrable near $0$.
Then $f\leq \phi $ on $ (0,T)$.
\end{lemma}
Note that no assumption on the continuity of $f$ is needed. We prove this lemma, along with a few related results, in Appendix \ref{integral_ineq} (see Lemma \ref{lem_integral_ineq_app}).
\subsection{The Oseen kernel $\mathcal{T}$}\label{oseen_kernel_section}
The Oseen kernel is the main tool in solving the Stokes equations in $\RR^3$ (which we discuss in the next section; see also the comment following Theorem \ref{classical_sol_if_X_smooth_theorem}). It is a $3\times 3$ matrix-valued function $\mathcal{T}=\left[ \mathcal{T}_{ij}  \right]$ defined by
\begin{equation}\label{def_Tij}
\mathcal{T}_{ij} (x,t) \coloneqq  \delta_{ij} \Phi (x,t) + \p_{i} \p_{j} P (x,t),\qquad x\in \RR^3 , t>0,
\end{equation}
where $\delta_{ij}$ denotes the Kronecker delta and
\begin{equation}\label{def_of_p}
P(x,t) \coloneqq \frac{1}{4\pi^{3/2}t^{1/2} |x|} \int_0^{|x|} \mathrm{e} ^{-\xi^2/4t}\,  \d \xi.
\end{equation}
It was first introduced by \cite{Oseen_1911} (see pages 3, 19 and 41 therein)%\footnote{It was also acknowledged by Leray in the footnote on p. 211.}. 

Note that $P(\cdot , t)$ is a smooth function for each $t>0$. Indeed, for fixed $t$ the function $\tilde P: \RR \to \RR$ defined by $\tilde P (0) := 1$, 
\[\tilde P (s) := \frac{1}{4\pi^{3/2} s} \int_0^s \frac{\mathrm{e} ^{-r^2/4t}}{t^{1/2}} \, \d r\]
is even and smooth (one can verify that $\tilde P$, $\frac{\d \tilde P}{\d s}$ are continuous and that $\frac{\d^2 \tilde P}{\d s^2}=  \e^{-s^2/4t}/{(4\pi t)^{3/2}}$, a smooth function). Therefore, since $P(x,t)=\tilde P(|x|)$, $P(\cdot , t)$ is smooth as well (for each $t>0$). A direct computation shows that
\begin{equation}\label{laplace_relation}
-\Delta P = \Phi .
\end{equation}
This yields an equivalent definition of $\mathcal{T}_{ij}$:
\begin{equation}\label{eqOseen1}
\begin{array}{lll}
\mathcal{T}_{1,1}=-(\p_2^2+\p_3^2)P,\quad &\mathcal{T}_{1,2}=\p_1\p_2 P,\quad  &\mathcal{T}_{1,3}=\p_1\p_3 P,\\
\mathcal{T}_{2,1}=\p_1\p_2P,&\mathcal{T}_{2,2}=-(\p_1^2+\p_3^2) P, &\mathcal{T}_{2,3}=\p_2\p_3 P,\\
\mathcal{T}_{3,1}=\p_1\p_3P,&\mathcal{T}_{3,2}=\p_2\p_3 P, &\mathcal{T}_{3,3}=-(\p_1^2+\p_2^2) P.
\end{array}
\end{equation}
Since the derivatives $\nabla^m P$ satisfy the pointwise estimate
\[
| \nabla^m P(x,t) | \leq \frac{C_m}{(|x|^2+t)^{(m+1)/2}}, \qquad m\geq 0
\]
(see Theorem \ref{decay_p_theorem}), we obtain the following pointwise estimates on the Oseen kernel
\begin{equation}\label{eqOseenKerEst2}
|\nabla^m \mathcal{T}(x,t)|\leq  \frac{C_m}{(|x|^2+ t)^{(m+3)/2}} \qquad \text{ for } x \in \RR^3, t>0 , m\geq 0.
\end{equation}
Using these bounds we can easily deduce the integral estimates
\begin{equation}\label{integral_est_T}
\| \mathcal{T} (t) \| \leq C\, t^{-3/4}\quad \text{ and }\quad \| \nabla \mathcal{T} (t) \|_{1} \leq C\, t^{-1/2}
\end{equation}
for $t>0$, where we used the facts $\int(|x|^2+t)^{-3}\,  \d x = C/t^{3/2}$ and $\int(|x|^2+t)^{-2}\,  \d x = C/t^{1/2}$. % Recall that $\| \cdot \|$ denotes the $L^2$ norm on $\RR^3$.
 Similarly, by \eqref{eqOseenKerEst2} and an application of the Dominated Convergence Theorem we obtain
\begin{equation}\label{fcn_spaces_for_T}
\mathcal{T}\in  C((0,\infty ); L^2 ) \quad \text{ and } \quad \nabla \mathcal{T} \in C((0,\infty ); L^1 ).
\end{equation}
Finally $\mathcal{T}$ enjoys a certain integral continuity property, which we will use later to show H\"older continuity of the solution to Stokes equations (see Lemma \ref{additional_prop_up_lemma} (ii)).
\begin{lemma}[$1/2$-H\"older continuity of $\nabla \mathcal{T}$ in an $L^1$ sense]\label{holder_cts_nablaT_lemma}
For $x,y \in \RR^3$, $t>0$
\[
\int | \nabla \mathcal{T}(x-z,t)-\nabla \mathcal{T}(y-z,t)|\,  \d z \leq C |x-y |^{1/2} {t^{-3/4}} .
\]
\end{lemma}
Leray mentions this inequality on page 213, and he frequently uses it in his arguments (in {\it(2.14)}, {\it(2.18)}, the inequality following {\it(3.3)}, and the first inequality on page 219) to show H\"older continuity of functions given by representation formulae involving $\nabla \mathcal{T}$. We provide a proof for the sake of completeness.
\begin{proof}
Let $R\coloneqq |x-y|$ and 
\[
\Omega\coloneqq B(x,2R) \cup B(y,2R).
\]
Since $\Omega \subset B(x,3R)$ we can use \eqref{eqOseenKerEst2} to write 
\begin{equation}\label{T_integral_cts_temp1}
\begin{split}
\MoveEqLeft\int_\Omega | \nabla \mathcal{T} (x-z,t) |\,  \d z\\
&\leq C \int_{B(x,3R)} \frac{\d z}{(|x-z|^2+t)^2} =C\int_{0}^{3R}\frac{r^2}{(r^2+t)^2}\, \d r \\
&\leq C\int_{0}^{3R}\frac{r^2+t}{(r^2+t)^2}\, \d r  = \frac{C}{t^{1/2}} \tan^{-1}\left(\frac{3R}{t^{1/2}}\right)\leq C \frac{R^{1/2} }{t^{3/4}},
\end{split}
\end{equation}
since $\tan^{-1} \alpha \leq \alpha^{1/2}$ for $\alpha>0$. Analogously 
\begin{equation}\label{T_integral_cts_temp2}
\int_\Omega | \nabla \mathcal{T} (y-z,t) |\,  \d z \leq CR^{1/2} t^{-3/4}.
\end{equation}
As for $z\in \RR^3 \setminus \Omega$ note that $|z-x|\geq 2R $. Hence for any point $\xi $ from the line segment $[x,y]$
\[
|z-x| \leq |z-\xi | + |\xi -x | \leq |z-\xi | + R \leq |z-\xi | + |z -x |/2, 
\]
and so $|z-\xi |\geq |z-x|/2$. Thus, using the Mean Value Theorem and the bound on $\nabla^2 \mathcal{T}$ (see \eqref{eqOseenKerEst2}) we obtain
\begin{align*}
\MoveEqLeft[0]
\int_{\RR^3\setminus \Omega} | \nabla \mathcal{T}(x-z,t)-\nabla \mathcal{T}(y-z,t)|\,  \d z = R\int_{\RR^3\setminus \Omega } |\nabla^2 \mathcal{T} (\xi(z) -z,t) |\,  \d z \\
&\leq C R \int_{\RR^3 \setminus \Omega } \frac{\d z}{(|\xi (z) - z |^2 +t)^{5/2}} \leq C R \int_{\RR^3 \setminus B(x,R) } \frac{\d z}{(|x - z |^2 +2t)^{5/2}} \\
&= C R \int_{R}^\infty \frac{r^{2}}{(r^2+2t)^{5/2}}\,\d r  \leq C R^{1/2} \int_{R}^\infty \frac{r^{5/2}}{(r^2+2t)^{5/2}}\,\d r& \\
 &&\mathllap{\leq C R^{1/2} t^{-3/4} \int_{R/\sqrt{t}}^\infty \frac{\rho^{5/2}}{(\rho^2+1)^{5/2}}\,\d \rho \leq C R^{1/2} t^{-3/4},\hspace{5pt}}
\end{align*}
where $\xi(z) \in [x,y]$ for each $z$. This together with \eqref{T_integral_cts_temp1} and \eqref{T_integral_cts_temp2} proves the claim. 
\end{proof}

%/////////////////////////////////////////////////////

\section{The Stokes equations}\label{stokes_eq_section}
In this section we consider the Stokes equations,
\begin{eqnarray}
\p_t  u -\Delta u +\nabla p&=& F , \label{eqLinSys1} \\
\mathrm{div}\, u&=&0\qquad \text{ in } \RR^3 \times (0,T),   \label{eqLinSys1_incomp}
\end{eqnarray}
where $T>0$ and $F(x,t)$ is a vector-valued forcing. Leray calls these equations the \emph{infinitely slow motion}. The Stokes equations model a drift-diffusion flow of a incompressible velocity field $u$. Here $p$ denotes the pressure function. One can think of the appearance of the pressure function as providing the extra freedom necessary to  impose the incompressibility constraint  \eqref{eqLinSys1_incomp} for an arbitrary $F$, see the comment after Theorem \ref{classical_sol_if_X_smooth_theorem}. As usual, we denote the initial condition for \eqref{eqLinSys1}, \eqref{eqLinSys1_incomp} by $u_0\in H$, which is understood in the sense of the $L^2$ limit, that is $\| u(t)-u_0 \| \to 0$ as $t \to 0^+$.

In his paper, Leray includes an essentially complete analysis of the Stokes initial value problem in $\RR^3$. The results for this problem are fundamental in the analysis of the Navier--Stokes equations that follows; while the arguments are at times somewhat technical, we therefore present them in full, but with some details in Appendix \ref{appendix_stokes_eq_X}. 

The Stokes equations with the general form of the forcing $F$ can be solved using the representation formulae\footnote{These formulae are stated by Leray in {\it(2.2)} and {\it(2.9)}.},
\begin{eqnarray}
u(t)&= &u_1 (t)+u_2 (t) \coloneqq \Phi (t)\ast u_0  +\int_0^t  \mathcal{T}(t-s) \ast F(s)\, \d s , \label{repr_formula} \\
p(t)&\coloneqq & -(-\Delta) ^{-1} ( \mathrm{div} \, X(t) ) ,\label{repr_formula_p}
\end{eqnarray}
see Theorem \ref{classical_sol_if_X_smooth_theorem} below (in which we focus only on the case of regular $F$). (See \eqref{def_of_inverse_laplace} for the definition of $(-\Delta) ^{-1}$, and recall that $\Phi (t) $ denotes the heat kernel \eqref{def_of_heat_kernel}.) Here the convolution of the matrix function $\mathcal{T}(t-s)$ and a vector function $F(s)$ is understood as a matrix-vector operation, that is
\[
u_{2,i} (x,t) = \int_0^t \int \mathcal{T}_{ij} (x-y, t-s) F_j (y,s) \, \d y \,  \d s.
\]
In this section we study the representation formulae \eqref{repr_formula}, \eqref{repr_formula_p} and certain modified representation formulae (that is \eqref{repr_formula1}, \eqref{repr_formula1_p}) in the case when $F$ is of the special form $F=-(Y\cdot \nabla )Y$ for some vector field $Y$. Of the two cases
\[
F\text{ of  general form} \,\,\quad  \text{ and } \quad \,\, F=-(Y\cdot \nabla )Y \text{ for some }Y
\]
Leray considers\footnote{This corresponds to Sections 11--13.} mainly the former case;  studying the  formula \eqref{repr_formula} given appropriate regularity of $F$, and only mentioning briefly the latter case\footnote{See Lemma 8 in his work.}. Here we treat the two cases separately. We treat the former case briefly, and we focus more on the latter case. An advantage of this approach is that it makes our results for each of the two cases directly applicable in the analysis of the Navier--Stokes equations. Moreover, in this slight refinement of Leray's approach, we construct the solution using the representation formulae, rather than deducing the representation as a property of the solution. As a result, we obtain a simple existence and uniqueness theorem for the Stokes equations (Theorem \ref{sol_distributions_anyX_theorem}).

\subsection{A general forcing $F$}\label{section_forcing_X}
Consider a forcing $F\in C([0,T);L^2)$ and let $u$, $p$ be given by the representation formulae \eqref{repr_formula}, \eqref{repr_formula_p} above.
\begin{lemma}\label{prop_of_up_lemma}
If $F\in C([0,T);L^2)$ then the function $u$ defined above satisfies
\begin{enumerate}[\rm(i)]
\item $u \in C((0,T);L^\infty )$ with\footnote{Leray does not state this bound (we state it as a tool for proving (iii)).},
\begin{equation}\label{linfty_bound}
\| u(t) \|_{\infty } \leq C\int_0^t \frac{\| F(s) \|}{(t-s)^{3/4}} \, \d s +  C \| u_0 \| t^{-3/4}.
\end{equation}
\item $\nabla u \in C((0,T);L^2)$ with\footnote{This is Leray's {\it(2.8)}, {\it(2.12)} and {\it(2.19)}.}
\begin{equation}\label{l2_grad_bound}
\| \nabla u(t) \| \leq C \int_0^t \frac{\| F(s) \|}{(t-s)^{1/2}}\,  \d s+  C\| u_0 \| t^{-1/2}  .
\end{equation}
More generally, if $F, \ldots , \nabla^m F\in C([0,T);L^2)$ then 
\[\nabla^{m+1} u \in C((0,T);L^2)\]
 with
\[
\| \nabla^{m+1} u (t) \| \leq C \int_0^t \frac{\| \nabla^m F(s) \|}{(t-s)^{1/2}}\,  \d s+  C_{m+1}\| u_0 \| t^{-(m+1)/2}.
\]
\item $u\in C([0,T); L^2 )$ with\footnote{This is Section 13 in \cite{Leray_1934}.}
\begin{equation}\label{l2_bound}
\| u (t) \| \leq \int_0^t \| F(s) \|\,  \d s + \| u_0 \| \qquad \text{ for all }t\in (0,T).
\end{equation}
Moreover $u$ satisfies the energy dissipation equality
\begin{equation}\label{en_dissipation}
\| u(t) \|^2 - \| u_0 \|^2 + 2 \int_0^t \| \nabla u(s) \|^2\,  \d s = 2 \int_0^t \int u \cdot F\,\, \d x \, \d s 
\end{equation}
for all $t\in (0,T)$. 
\end{enumerate}
\end{lemma}
The properties (i), (ii) of the lemma follow from Lemma \ref{young_spacetime_continuity}, the integral bounds on the Oseen kernel (see \eqref{integral_est_T}) and the corresponding property of the heat kernel (see (ii), (iii) in Appendix \ref{heat_section_appendix}). As for (iii), assuming first that $F$ is smooth, %(as in \eqref{regularity_of_X_requirement} below) 
the functions $u$, $p$ constitute a classical solution to the Stokes equations \eqref{eqLinSys1}, \eqref{eqLinSys1_incomp} (which is proved in the following theorem). The required estimates are straightforward for classical solutions. If $F$ is not smooth, one obtains (iii) by a density argument. See Appendix \ref{app_prop_iii} for the detailed proof of (iii).

\begin{theorem}[Classical solution for smooth forcing $F$]\label{classical_sol_if_X_smooth_theorem}$\mbox{}$\\
Suppose that for some $R>0$ 
\begin{equation}\label{regularity_of_X_requirement}
F\in C^\infty (\RR^3 \times [0,T); \RR^3 )\quad \text{ and } \quad \supp \, F(t) \subset B(0,R)\text{ for } t\in [0,T) .
\end{equation}
Then the pair of functions $u$, $p$ given above is a classical solution of the Stokes equations \eqref{eqLinSys1}, \eqref{eqLinSys1_incomp} with $u(0)=u_0$. Moreover $u\in C([0,T); L^2 )$.
\end{theorem}
In fact the functions $u$, $p$ constitute a unique solution in a much wider class, namely the class of distributional solutions $u$, $p$ such that $u\in C([0,T);L^2)$ and $p\in L^1_{\loc} (\RR^n \times [0,T))$, see Theorem \ref{uniqueness_stokes} in Section \ref{distr_sol_SE_section}.

Theorem \ref{classical_sol_if_X_smooth_theorem} follows by showing that $u_1$, $u_2$ satisfy the equations\footnote{The study of $u_1$ and $u_2$ corresponds to Leray's Sections 11 and 12, respectively.}
\begin{equation}\label{the_decomposition_u1_u2}
\left\{
\begin{array}{rl}
\p_t u_1 -\Delta u_1 &=0,\\
\mathrm{div}\, u_1 &=0,\\
 u_1 (0)&=u_0,
\end{array}  
\right.\hspace{2cm} 
\left\{
\begin{array}{rl}
\p_t u_2 -\Delta u_2 +\nabla p  &=F,\\
\mathrm{div}\, u_2 &=0,\\
 u_2 (0)&=0,
\end{array}  
\right.
\end{equation}
The part of the claim for $u_1$ follows directly from the analysis of the heat equation, see Appendix \ref{heat_section_appendix}. As for $u_2$, using the Fourier transform one can see that it is enough to prove the claim in Fourier space. It therefore suffices to use the Fourier transform of the Oseen kernel,
\begin{equation}\label{eq_LerProj_Fourier}
\mathcal{F} [\mathcal{T}(t)]=\left(I- \frac{\xi\otimes\xi}{|\xi|^2} \right) \mathrm{e}^{-4\pi^2 t|\xi|^2}    \quad t>0,
\end{equation}
obtained from \eqref{def_Tij} and \eqref{laplace_relation}, where $I$ denotes the identity matrix and $\xi \otimes \xi $ denotes the $3\times 3$ matrix with components $\xi_i \xi_j$. An interested reader is referred to Appendix \ref{app_stokes_pf_of_thm} for the detailed proof.

At this point it is interesting to note that the Stokes equations are in fact a nonhomogeneous heat equation for $u$ under the incompressibility constraint $\mathrm{div} \, u=0$. Since $\Delta p = \mathrm{div}\, F$ we see that $p$ appearing in the Stokes equations acts as a modification of the forcing $F$ to make it divergence free (that is $\mathrm{div}\, (F-\nabla p )=0$). Since any solution of a nonhomogeneous heat equation with divergence-free forcing and initial data remains divergence free for positive times, we see that the role of $p$ in the Stokes equations is to guarantee that $u(t)$ remains divergence free for $t>0$. One can also think of it as the projection of $X(t)$ onto the space of weakly divergence-free vector fields (which is often called the \emph{Leray projection}).
 
Moreover, from \eqref{repr_formula_p} we see that the Fourier transform of the modified forcing $F-\nabla p$ is
\begin{equation*}
\left( I - \frac{ \xi\otimes \xi}{|\xi |^2} \right) \widehat{F}(\xi,t).
\end{equation*}
 Thus we see from \re{eq_LerProj_Fourier} that the Oseen kernel is precisely the modification of the heat kernel that accounts for this modification of the forcing. This is particularly clear from a calculation in the Fourier space in Appendix \ref{app_stokes_pf_of_thm}.

\subsection{A forcing of the form $F=-(Y\cdot \nabla )Y$}\label{distr_sol_SE_section}
Here we assume that $F$ is of a particular form, namely
\begin{equation}\label{form_of_X}
F = -(Y \cdot \nabla ) Y
\end{equation}
(in components $F_k = Y_i \p_i Y_k$) for some weakly divergence-free $Y \in C((0,T),L^\infty )$ such that $ \| Y(t) \|_\infty $ remains bounded as $t \to 0^+$. Note that since the derivatives $\p_i Y_k$ are not well-defined we understand \eqref{form_of_X} in a formal sense and we will consider the Stokes equations \eqref{eqLinSys1}, \eqref{eqLinSys1_incomp} in the sense of distributions. More precisely, we say that $u,p$ is a {\it distributional solution} of \eqref{eqLinSys1}, \eqref{eqLinSys1_incomp} with $F$ of the form \eqref{form_of_X} if $u(t)$ is weakly divergence free for $t\in (0,T)$ and
\begin{equation}\label{eqLinSys_distributions}
\int u_0 \cdot \phi (0) \,\d x+\int_0^T  \int \left( u \cdot (\phi _t + \Delta \phi  )+ p\, \mathrm{div} \, \phi  \right)  =  \int_0^T \int Y \cdot (Y\cdot \nabla ) \phi  
\end{equation}
for $\phi \in C_0^\infty (\RR^3 \times [0,T);\RR^3)$. We will also consider a modified form of the representation formulae \eqref{repr_formula}, \eqref{repr_formula_p} that accounts for this special form of the forcing,
\begin{equation}\label{repr_formula1}
u(t) =  u_{1} (t) + u_{2} (t) \coloneqq  \Phi (t)\ast u_0  +\int_0^t  \nabla  \mathcal{T}(t-s) \ast \left[ Y(s) Y (s) \right] \,  \d s ,
\end{equation}
\begin{equation}\label{repr_formula1_p}
p(t)\coloneqq  \p_k \p_i (-\Delta )^{-1} (Y_i (t) Y_k(t)) ,
\end{equation}
where we write
\begin{equation}\label{notation_oseen_ker_conv}
\left(  \nabla  \mathcal{T}(t-s) \ast \left[ Y(s) Y (s) \right] \right)_j (x) \coloneqq \int \p_i \mathcal{T}_{jk} (x-y,t-s) Y_i (y,s) Y_k (y,s) \, \d y .
\end{equation}
Clearly, such $u$, $p$ are well-defined since no derivatives fall on $Y$ in these modified representation formulae. If $Y$ is regular (in the sense of \eqref{regularity_of_X_requirement}) then the above definition of $u$, $p$ is equivalent to \eqref{repr_formula}, \eqref{repr_formula_p}, and so Theorem \ref{classical_sol_if_X_smooth_theorem} implies that such $u$, $p$ constitute a classical solution of the Stokes equations (and hence also a distributional solution). In this section we show that $u$, $p$ constructed above constitute the unique distributional solution in a wide class if we have $Y\in C([0,T);L^2)$ in addition to the assumptions on $Y$ mentioned in \eqref{form_of_X}, see Theorem \ref{sol_distributions_anyX_theorem} below.\footnote{This corresponds to Leray's Lemma 8, in which he states that the representation formula \eqref{repr_formula1} is a property of the solution.}

For this purpose we derive several properties of such $u$, $p$. In view of Lemma \ref{prop_of_up_lemma}, we now prove refined bounds on $\| u(t) \|_\infty$ and $\| u (t) \|$, and show that $u\in \mathcal{H}^{1/2} ((0,T))$ and $\nabla u \in C((0,T);L^\infty )$.
\begin{lemma}[Properties of $u$, $p$ given by (\ref{repr_formula1}-\ref{repr_formula1_p})]\label{additional_prop_up_lemma}
Let $u_0\in L^2$ and $u,p$ given by \re{repr_formula1} and \re{repr_formula1_p} for some $Y$ with the properties described following \eqref{form_of_X}. 
\begin{enumerate}[{\rm(i)}]
\item If $u_0$ is bounded then $u\in C((0,T),L^\infty )$ with\footnote{Leray does not state this bound explicitly, but he uses it in later sections during the study of the Navier--Stokes equations (for instance in {\it(3.5)}, at the bottom of p. 222, and at the top of p. 232).} 
\[
\| u(t) \|_\infty \leq C \int_0^t \frac{\| Y(s) \|_\infty^2}{\sqrt{t-s}} \,\d s + \| u_0 \|_\infty .
\]
Moreover $u\in C([0,T),L^\infty )$ if $u_0 \in L^\infty$ is uniformly continuous.
\item $u\in \mathcal{H}^{1/2} ((0,T))$ and the corresponding H\"older constant $C_0(t)$ satisfies\footnote{Leray shows a similar property of $\nabla u$ in the case of $F$ of general form (which he obtains in {\it(2.18)} and as a consequence of {\it(2.7)} and {\it(2.8)}). We translate this result to the case of $F$ of the form $F=-(Y\cdot \nabla ) Y$.}
\[
C_0(t) \leq c_0 \int_0^t \frac{\| Y(s) \|_{\infty}^2}{(t-s)^{3/4}} \,\d s + c_0 \frac{\| u_0 \|}{t}
\]
for some $c_0>0$.

More generally, if $Y$, $\nabla Y $, ... , $\nabla^m Y\in C((0,T),L^\infty )$ with the respective $L^\infty$ norms bounded as $t \to 0^+$ then $\nabla^m u \in \mathcal{H}^{1/2} ((0,T))$ and the corresponding constant $C_m(t)$ satisfies
\[
C_m (t) \leq c_m \sum_{\alpha+\beta = m} \int_0^t \frac{\| \nabla^\alpha Y(s) \|_{\infty }\| \nabla^\beta Y(s) \|_{\infty }}{(t-s)^{3/4}} \,\d s + c_m \frac{\| u_0 \|}{t^{(m+2)/2}} .
\]
\item If additionally $Y\in C([0,T ); L^2)$ then $p\in C([0,T);L^2)$ and $u\in C([0,T); L^2)$ with\footnote{Leray does not state this property, but he uses it in showing existence of strong solutions to the Navier--Stokes equations (in the inequality he states at the bottom of page 223). We will apply it in a similar way (see Theorem \ref{local_exist_strong_thm}) and also it in the existence and uniqueness theorem for the Stokes equations (Theorem \ref{sol_distributions_anyX_theorem}).}
\begin{equation}\label{l2_continuity_X_special_form}
\| u (t) \| \leq C \int_0^t \frac{\| Y(s) \|_{\infty } \| Y(s) \|}{\sqrt{t-s}} \,\d s + \| u_0 \| .
\end{equation}
Moreover if $T'<T$ and $\{ Y^{(n)} \}$ is a sequence such that\ $ Y^{(n)} \to Y$ in $C([0,T'];L^2)$ and  $\max_{t\in [0,T']} \| Y^{(n)} (t) \|_\infty \leq  \max_{t\in [0,T']} \| Y (t)\|_\infty  $ then
\[
u^{(n)} \to u \quad \text{ and }\quad  p^{(n)} \to p \quad \text{ in } \quad C([0,T'];L^2).
\]
\item If additionally $Y\in  \mathcal{H}^{1/2} ((0,T))$ with the corresponding constant $C_0(t)$ bounded as $t \to 0^+$ then $\nabla u \in C((0,T);L^\infty )$ with\footnote{This corresponds to Leray's property of $\nabla u$ (which he obtains in {\it(2.7)} and {\it(2.20)}).}
\begin{equation}\label{nabla_of_u_Linfty_bound}
\| \nabla u(t) \|_{\infty} \leq C\int_0^t \frac{\| Y(s) \|_{\infty } C_0 (s) }{(t-s)^{3/4}} \,\d s +  C \frac{  \| u_0 \|}{t^{5/4}}  .
\end{equation}
More generally if for every multi-index $\alpha$ with $|\alpha | \leq m-1$, $D^\alpha Y \in \mathcal{H}^{1/2}((0,T))\cap  C((0,T);L^\infty )$ with the corresponding H\"older constant $C_\alpha (t)$ and $\| D^\alpha Y (t) \|_\infty$ bounded as $t \to 0^+$ then $\nabla^{m} u \in C((0,T);L^\infty )$ with
\begin{multline*}
\| \nabla^{m} u (t) \|_{\infty } \leq C_m \int_0^t \frac{\sum_{\alpha+\beta =m-1} \| \nabla^\alpha Y(s) \|_{\infty }C_\beta (s)}{(t-s)^{3/4}} \,\d s\\
 +  C_{m}t^{-\frac{m}{2}-\frac{3}{4}}\|u_0\|
\end{multline*}
\end{enumerate}
\end{lemma}
\begin{proof}
Since $Y\in C((0,T);L^\infty )$ with $\| Y(t) \|_\infty $ bounded as $t \to 0^+$ the same is true of $Y_i Y_k $ for each pair $i,k$ and so claim (i) follows from Lemma \ref{young_spacetime_continuity} and from the properties of the heat kernel (see Appendix \ref{heat_section_appendix}; note also that Lemma \ref{Linfty_conv_heat_eq} verifies the comment in (i)). In a similar way one obtains (iii), where the claim for $p$ follows directly from the Plancherel Lemma (Lemma \ref{CZ_lemma}) and the limiting property follows from \eqref{l2_continuity_X_special_form} and the Plancherel Lemma. Property (ii) is a consequence of the H\"older continuity of the heat kernel (see (v) in Appendix \ref{heat_section_appendix}) and the H\"older continuity of $\nabla \mathcal{T}$ in $L^1$ (see Lemma \ref{holder_cts_nablaT_lemma}). Indeed we have
\begin{multline*}
|u_2(x,t)-u_2(y,t)|\\\leq \int_0^t \int \left| \nabla \mathcal{T} (x-z,t-s) - \nabla \mathcal{T}(y-z,t-s) \right| \,\d z \| Y(s) \|_{\infty}^2 \,\d s \\\ 
\leq  C |x-y|^{1/2} \int_0^t \frac{\| Y(s) \|_{\infty}^2}{(t-s)^{3/4}} \,\d s .
\end{multline*}
As for property (iv) note that the bound on $\nabla u_1$ in \eqref{nabla_of_u_Linfty_bound} (that is $\| \nabla u_1 (t) \|_\infty \leq C \| u_0 \| t^{-5/4}$) and the continuity $\nabla  u_1 \in C((0,T);L^\infty  )$ follow from properties of the heat kernel (see (iv) in Appendix \ref{heat_section_appendix}). As for $u_2$, the bound on $\nabla u_2$ in \eqref{nabla_of_u_Linfty_bound} can be shown using the following trick. Recalling that $Y$ is weakly divergence free, we obtain %$\p_{li}  \mathcal{T}_{kj} Y_i = \p_i (\p_l \mathcal{T}_{kj} Y_i)$ for $j,k,l\in \{ 1,2,3\}$, and so the estimate \eqref{eqOseenKerEst2} lets us integrate this equality in space to obtain
\[
 \int \p_{li}  \mathcal{T}_{kj} (x-y,t) Y_i(y,s) \,\d y=0
\]
for all $j,k,l$ and $s,t \in (0,T)$, $x\in\RR^3$, where the integral exists due to \eqref{eqOseenKerEst2}. Hence, for each $j,l$ and $t\in (0,T)$
\begin{multline}\label{temp_prop_up}
\p_l u_{2,j} (x,t) = - \int_0^{t} \int \p_{li} \mathcal{T}_{jk} (x-y,t-s)  Y_i (y,s) Y_k (y,s)\,  \d y \, \d s \\
= - \int_{0}^{t} \int \p_{li} \mathcal{T}_{jk} (x-y,t-s)  Y_i (y,s)  \left[ Y_k (y,s) - Y_k(x,s) \right]\, \d y \, \d s,
\end{multline}
and so using the bound $|\nabla^2 \mathcal{T} (x,t)|\leq C(|x|^2+t)^{-5/2}$ (see \eqref{eqOseenKerEst2}) and the assumption $Y\in \mathcal{H}^{1/2} ((0,T))$ we obtain  
\[
\begin{split}
|\nabla u_2 (x,t) | &\leq C \int_0^{t} \int \frac{C_0 (s)\| Y (s) \|_{\infty }  |x-y|^{1/2}}{(|x-y|^2+(t-s))^{5/2}}  \,\d y \, \d s\\
&=  C \int_{0}^{t} \frac{\| Y(s) \|_{\infty } C_0 (s)}{(t-s)^{3/4}} \,\d s,
\end{split} \]
where we used the fact $\int |y|^{1/2}/(|y|^2 +t)^{5/2} \,\d y = C t^{-3/4}$. Thus \eqref{nabla_of_u_Linfty_bound} follows. One can also employ this trick to show the continuity $\nabla u_2 \in C((0,T);L^\infty )$, see Appendix \ref{app_stokes_continuity} for the details. 

Finally, claims (ii) and (iv) for higher derivatives $\nabla^m u$ follow in a similar way. Indeed, the claims corresponding to $u_1$ follow from the properties of the heat kernel (see (iv) in Appendix \ref{heat_section_appendix}) and, as for $u_2$, we write any $D^\gamma u_{2,j}$ with $|\gamma |=m$ as the sum of the integrals 
\[
\int_0^t \int \p_l \p_i \mathcal{T}_{jk} (x-y,t-s) Y_{\alpha ,i} (y,s)  Y_{\beta ,k} (y,s)\, \d y\, \d s
\]
where $Y_\alpha$, $Y_\beta$ denote appropriate derivatives of $Y$ of orders $\alpha$, $\beta$, respectively, where $\alpha+\beta =m-1$, $l,j=1,2,3$, and we repeat the reasoning above.
\end{proof}

\begin{corollary}\label{additional_prop_up_corollary}
The results of the above lemma extend to the case $F_i=-\p_k (Y_i Z_k)$ for some weakly divergence-free $Y,Z\in C((0,T),L^\infty )$ with the $L^\infty$ norms bounded as $t \to 0^+$. In particular, if such $Y,Z$ satisfy also $Y,Z \in C([0,T),L^2 )$ and $v$ is given by
\eqnb\label{repr_formula1_two_vectors}
v(t) \coloneqq  \Phi (t)\ast u_0  +\int_0^t \nabla \mathcal{T} (t-s) \ast \left[ Y(s) Z (s) \right]  \, \d s 
\eqne
then $v\in C((0,T),L^\infty ) \cap C([0,T),L^2 )$ with
\[
\begin{split}
\| v(t) \|_\infty &\leq C \int_0^t \frac{\| Y(s) \|_\infty \| Z(s) \|_\infty }{\sqrt{t-s}} \,\d s + \| u_0 \|_\infty ,\\
\| v (t) \| &\leq C \int_0^t \frac{\| Y(s) \|_{\infty } \| Z(s) \|}{\sqrt{t-s}} \,\d s + \| u_0 \| .
\end{split}
\]
\end{corollary}
The existence and uniqueness theorem for distributional solutions to the Stokes equations is based on the following uniqueness result.
\begin{theorem}[Uniqueness of distributional solutions to the Stokes equations\footnote{This is a version of the argument from Section 14 of \cite{Leray_1934}.}]\label{uniqueness_stokes}
If $u,p$ are such that $u\in C([0,T); L^2)$ is weakly divergence-free, $p\in L^1_{\loc} (\RR^3 \times [0,T))$, and 
\begin{equation}\label{weak_homogeneous_solution}
\int_0^T \int \left( (\phi_t + \Delta \phi ) \cdot u + p \, \mathrm{div}\, \phi \right) \,\d x \, \d t =0
\end{equation}
for all $\phi \in C_0^\infty (\RR^3 \times [0,T) ; \RR^3)$, then $u\equiv 0$.
\end{theorem}
It follows that $\int_0^T \int p \, \mathrm{div}\, \phi \,\d x \, \d t =0$, and so integration by parts and the Fundamental Lemma of Calculus of Variations give $\nabla p \equiv 0$; that is $p$ is a function of $t$ only. Since both the Stokes equations and the Navier--Stokes equations are invariant under addition to the pressure function any function of time, we will identify two solutions $u_1,p_1$ and $u_2,p_2$ of the Stokes equations (or the Navier--Stokes equations) if $u_1=u_2$ and $p_1$ differs from $p_2$ by a function of time.
\begin{proof}[Proof (sketch; see Appendix \ref{app_stokes_uniqueness} for details).] The proof of the theorem is based on considering the regularisations of $u$, $p$,
\[
v (x,t) \coloneqq \int_0^t (J_\varepsilon u)(x,s)\, \d s, \qquad {q} (x,t) \coloneqq \int_0^t (J_\varepsilon p)(x,s)\, \d s,\quad \varepsilon >0.
\]
Such a regularisation is still a solution of \eqref{weak_homogeneous_solution} and one can show that $\Delta q =0$ in a distributional sense. Thus $\Delta v$ satisfies the homogeneous heat equation in a distributional sense and the uniqueness of the solution to the heat equation gives $\Delta v=0$. An application of Liouville's theorem and the assumption $\| u(t) \|<\infty $ for all $t$ then gives $v\equiv 0$ for all $\varepsilon >0$, and consequently $u\equiv 0$.
\end{proof}
We are now ready to prove the existence and uniqueness of distributional solutions, the central result of the study of the Stokes equations.
\begin{theorem}[Distributional solution for $F$ of the form \eqref{form_of_X}]\label{sol_distributions_anyX_theorem}
Let $Y \in C([0,T),L^2)\cap C((0,T),L^\infty )$ be weakly divergence free such that $\| Y(t) \|_\infty$ remains bounded as $t \to 0^+$. Then $u,p$ given by \eqref{repr_formula1}, \eqref{repr_formula1_p} comprise a distributional solution of \eqref{eqLinSys1}, \eqref{eqLinSys1_incomp} with initial data $u_0$ and $F=-(Y\cdot\nabla)Y$. Moreover this solution is unique in the class $u\in C([0,T),L^2)$, $p\in L^1_{\loc} (\RR^3 \times [0,T))$.
\end{theorem}

\begin{proof}
Uniqueness follows from the theorem above. The fact that $u\in C([0,T);L^2)$ and the $L^2$ continuity at $t=0$, $\| u(t) - u_0 \| \to 0$ as $t\to 0$, is clear from Lemma \ref{additional_prop_up_lemma}, (iii).  Fix $\phi\in C_0^\infty (\RR^3 \times [0,T),\RR^3)$ and let $T'\in (0,T)$ be such that $\phi =0$ for $t\geq T'$. Let $\{ Y^{(n)} \}$ be a sequence of functions $Y^{(n)} \in C^\infty (\RR^3 \times [0,T))$ such that $\mathrm{supp}\, Y^{(n)} (t) \subset B(0,R_n)$ for some $R_n >0$,
\[
\| Y- Y^{(n)} \|_{C([0,T'],L^2)} \to 0 \quad \text{ as } n \to \infty 
\]
and $\max_{t\in[0,T']} \| Y^{(n)} (t)\|_\infty \leq \max_{t\in[0,T']} \| Y (t)\|_\infty $. Note that the above convergence means that also 
\[
\| Y_iY_k- Y^{(n)}_i Y^{(n)}_k \|_{C([0,T'],L^2)} \to 0 \quad \text{ as } n \to \infty 
\]
for all $i,k$.
The existence of such $Y^{(n)}$'s is guaranteed by Lemma \ref{approx_of_X_lemma}. Let $(u_n,p_n)$ be given by \eqref{repr_formula}, \eqref{repr_formula_p} with $F$ replaced by $F^{(n)}$, where $F_k^{(n)}\coloneqq -\p_i (Y^{(n)}_i Y^{(n)}_k)$. By Theorem \ref{classical_sol_if_X_smooth_theorem} $(u_n,p_n)$ satisfies the equations \eqref{eqLinSys1}, \eqref{eqLinSys1_incomp} with $F$ replaced by $F^{(n)}$ in the classical sense, and so also in the sense of distributions, that is $u_n$ is weakly divergence free and
\begin{equation}\label{eqLinSys_distributions1}
\int u_0 \cdot \phi \, \d x + \int_0^T  \int \left( u_n \cdot (\phi_t + \Delta \phi )+ p_n \, \mathrm{div}\, \phi \right)  =  \int_0^T \int Y^{(n)} \cdot (Y^{(n)} \cdot \nabla ) \phi .
\end{equation}
By Lemma \ref{additional_prop_up_lemma}, (iii), we have 
\[
\| u_n - u \|_{C([0,T'],L^2)} \to 0 ,\quad \| p_n -  p \|_{C([0,T'],L^2)} \to 0 \quad \text{ as } n \to \infty
\]
and so we can take the limit $n\to \infty $ to obtain that $u$ is weakly divergence free and, from \eqref{eqLinSys_distributions1},
\[
\int u_0 \cdot \phi \, \d x +\int_0^T  \int \left( u \cdot (\phi_t + \Delta \phi )+ p \, \mathrm{div}\, \phi \right)  =  \int_0^T \int Y \cdot (Y \cdot \nabla ) \phi ,
\]
that is $u,p$ is indeed the distributional solution.
\end{proof}
\begin{corollary}\label{cor_can_have_yz}
The conclusion of Theorem \ref{sol_distributions_anyX_theorem} also holds if $F$ is of the form $F_i=-\p_i (Y_i Z_k)$, where $Y,Z \in C([0,T),L^2)\cap C((0,T),L^\infty )$ are weakly divergence free with $\| Y(t) \|_\infty$, $\| Z(t) \|_\infty$ bounded as $t\to 0^+$, and the representation formula \eqref{repr_formula1} is replaced by \eqref{repr_formula1_two_vectors}.
\end{corollary}
\subsection*{Notes}
As remarked in the beginning of the section, we focused on the forcing of the form $F=-(Y\cdot \nabla )Y$ more directly than Leray. In particular Lemma \ref{additional_prop_up_lemma} is not stated by explicitly by Leray. Thanks to the use of the distributional form of the Stokes equations \eqref{eqLinSys_distributions} and the limiting property of the representation formulae \eqref{repr_formula1}, \eqref{repr_formula1_p} (that is Lemma \ref{additional_prop_up_lemma} (iii)) the main results of the section can be encapsulated in Theorem \ref{sol_distributions_anyX_theorem}.

In the next section we will follow Leray in applying the results for the Stokes equations to study the Navier--Stokes equations. In particular we will employ the other properties of the modified representation formulae \eqref{repr_formula1}, \eqref{repr_formula1_p}, that is Lemma \ref{additional_prop_up_lemma} (i),(ii),(iv). These properties were presented by Leray either implicitely during the study of the Navier--Stokes equations or by showing a related result for a general of the forcing $F$, as we pointed out in the footnotes in Lemma \ref{additional_prop_up_lemma}.
%#######################################################
\section{Strong solutions of the Navier--Stokes equations}\label{sec_strong}
We now consider the Navier--Stokes equations
\begin{equation}\label{eqNSE}
\p_t u - \Delta u +\nabla p = -(u \cdot\nabla)u,
\end{equation}
\begin{equation}\label{eqNSE_incomp}
\mathrm{div}\,  u=0
\end{equation}
in $\RR^3\times (0,T)$. We will consider a weak form of these equations,
\begin{equation}\label{eqNSE_distr}
\int_0^T  \int (u \cdot (f_t + \Lap f) + p \, \mathrm{div}\, f) = \int_0^T \int u \cdot (u \cdot \nabla) f
\end{equation}
for $f\in C_0^\infty (\RR^3 \times (0,T))$, where $u(t)$ is weakly divergence free. We first define solutions on the open time interval $(0,T)$ (see below) and we study their properties in Section \ref{smoothness_section}. In Section \ref{sec_local_exist_uni_strong} we equip the problem \eqref{eqNSE}-\eqref{eqNSE_incomp} with initial data, and for this reason we extend the definition of strong solutions to the half-closed time interval $[0,T)$ (Definition \ref{def_reg_sol_halfclosed}). We then show existence and uniqueness of local-in-time strong solutions (Theorem \ref{local_exist_strong_thm}). In Section \ref{char_of_sing_section} we study the maximal time of existence for strong solutions and the rate of blow-up of strong solutions if the maximal time is finite. In Section \ref{sec_semi_reg} we study local existence and uniqueness of strong solutions with less regular initial data.
\begin{definition}\label{def_strong_open} A function $u$ is a \emph{strong solution of the Navier--Stokes equations on the time interval $(0,T)$} if it satisfies the weak form of the equations with some $p\in L^1_{\loc} (\RR^3 \times (0,T))$ and if
\[
u\in C((0,T);L^2) \cap C((0,T);L^\infty ).
\]
\end{definition}
This is how Leray defines a strong solution, except that he requires the continuity of all terms appearing in the Navier--Stokes equations \eqref{eqNSE} (see p. 217). Here we make use of the weak form of equations and thus we avoid specifying any conditions on derivatives of $u$. However, smoothness of strong solutions (Corollary \ref{smoothness_reg_sol_corollary}) implies that the two definitions are equivalent to each other.
Note also that if $u$ is a strong solution of the Navier--Stokes equations then \eqref{eqNSE_distr} is equivalent to requiring that
\begin{multline}\label{eqNSE_distr_equiv}
\int u(t_1)\cdot f(t_1) - \int u(t_2) \cdot f(t_2) + \int_{t_1}^{t_2}  \int (u \cdot (f_t + \Lap f) + p \, \mathrm{div}\, f) \\= \int_{t_1}^{t_2} \int u \cdot (u \cdot \nabla) f
\end{multline}
for all $f\in C_0^\infty (\RR^3 \times (0,T))$ and $t_1,t_2\in (0,T)$ with $t_1<t_2$. While the ``$\Leftarrow$'' part this equivalence is trivial, the ``$\Rightarrow$'' part is not immediate but can be obtained by a simple cut-off procedure, which we now explain. 

For $h>0$ let $F_h (x,s)\coloneqq f(x,s) \theta_h (s)$, where $\theta_h \in C^\infty (\RR )$ is a nonincreasing function such that $\theta_h (s) = 1 $ for $s\leq t_2$, $\theta_h (s) =0 $ for $s\geq t_2+h$. Using $F_h$ as a test function in \eqref{eqNSE_distr} we obtain
\begin{multline*}
\int_0^{t_2+h}  \int \left( u \cdot (f_t + \Delta f ) + p\, \mathrm{div} \, f \right)\theta_h  + \int_{t_2}^{t_2+h} \int u \cdot f \theta_h'\\ =  \int_0^{t_2+h} \int u \cdot (u\cdot \nabla ) f \theta_h  .
\end{multline*}
Since $u,f \in C((0,T);L^2)$ the function $s\mapsto \int u(s) \cdot f(s) $ is continuous. Thus, since $\theta_h$ is nonincreasing and $\int_\RR \theta_h' =-1$ we obtain
\[
\int_{t_2}^{t_2+h} \int u \cdot f\, \theta_h' \to -\int u(t_2) \cdot f(t_2)  \quad \text{ as } h\to 0^+ .
\]
Thus taking the limit $h\to 0^+$ in the last equation (via the Dominated Convergence Theorem) gives 
\[
-\int u(t_2) \cdot f (t_2) + \int_0^{t_2}  \int \left( u \cdot (f_t + \Delta f ) + p\, \mathrm{div} \, f \right)    =  \int_0^{t_2} \int u \cdot (u\cdot \nabla ) f  .
\]
Applying a similar cut-off procedure at time $t_1$ gives \eqref{eqNSE_distr_equiv}.

From \eqref{eqNSE_distr_equiv} and the theorem about the existence and uniqueness of distributional solutions to Stokes equations (Theorem \ref{sol_distributions_anyX_theorem}) we see that a strong solution of the Navier--Stokes equations admits representation formulae\footnote{These are {\it(3.2)} and {\it(3.3)} in \cite{Leray_1934}.},
\begin{equation}\label{sol_NS_form}
\begin{split}
u(t_2)&=  \Phi (t_2-t_1) \ast u (t_1)+\int_{t_1}^{t_2} \int \nabla \mathcal{T} (t_2-t_1-s) \ast \left[ u (s)  u (s) \right]  \, \d s ,\\
p(t_2)&=  \p_i \p_k (-\Delta )^{-1} (u_i (t_2) u_k (t_2) )   .
\end{split} 
\end{equation}
for all $t_1,t_2 \in (0,T)$ with $t_1<t_2$. Recall we employ the notation
\[
\left(  \nabla  \mathcal{T}(t-s) \ast \left[ Y(s) Z (s) \right] \right)_j (x) \coloneqq \int \p_i \mathcal{T}_{jk} (x-y,t-s) Y_i (y,s) Z_k (y,s) \, \d y .
\]
Note this representation formula also determines uniquely the pressure function $p$.
\subsection{Properties of strong solutions}\label{smoothness_section}
In this section we study the properties of strong solutions of the Navier--Stokes equations on the open time interval $(0,T)$. We will show that if $u$ is a strong solution and $p$ is the corresponding pressure function then $u$ and $p$ are smooth and $u$ satisfies an energy equality, along with some other useful results.
The theorem below as well as the following corollary show that $u$, $p$ are smooth.
\begin{theorem}\label{thmStrongSolnsSmooth}\footnote{Theorem \ref{thmStrongSolnsSmooth} and Corollary \ref{smoothness_reg_sol_corollary} correspond to pp. 218-219 in \cite{Leray_1934}.} If $u$ is a strong solution of the Navier--Stokes equations on $(0,T)$ then
\[\nabla^m u \in C((0,T);L^2) \cap C((0,T);L^\infty)\quad \text{ for all } m \geq 0 .\]
\end{theorem}
The proof of the theorem (as well as the corollary that follows) is a simplification\footnote{This simplification is due to our choice to organise the properties of the representation formulae \eqref{repr_formula}, \eqref{repr_formula_p} and \eqref{repr_formula1}, \eqref{repr_formula1_p} into Lemmas \ref{prop_of_up_lemma} and \ref{additional_prop_up_lemma}, see Notes at the end of the section} of Leray's arguments (which he presents on pages 218-219).
\begin{proof}
The proof proceeds by a double use of induction. First, we show that
\[
\nabla^m u \in C((0,T);L^\infty ) \cap \mathcal{H}^{1/2} ((0,T))\quad \text{ for } m\geq 0.
\]
Here the base case follows from the definition of a strong solution and from Lemma \ref{additional_prop_up_lemma}, (ii), and the induction step follows from the same lemma, properties (iv) and (ii).

Second, we show that 
\[
\nabla^m u \in C((0,T);L^2) ,\,\, \nabla^m \left[ (u\cdot \nabla ) u \right] \in C((0,T);L^2)\quad \text{ for } m\geq 0.
\]
Here the base case follows from the definition of a strong solution and by deducing that $(u\cdot \nabla )u \in C((0,T);L^2)$ by using H\"older's inequality,
\begin{multline*}
\| (u(t)\cdot \nabla )u(t) - (u(s)\cdot \nabla )u(s) \|\\ \leq \| u(t) \| \, \| \nabla u(t) - \nabla u(s) \|_\infty + \| \nabla u(s) \|_\infty \| u(t) - u(s) \|.
\end{multline*}
The induction step follows from Lemma \ref{prop_of_up_lemma} (ii) and from a similar use of H\"older's inequality.
\end{proof}
Note that in fact for each $s\in (0,T)$ we have bounded the norms $\| \nabla^m u (s) \|$, $\| \nabla^m u(s) \|_\infty$, $m\geq 0$ using only the norms $\| u (t) \|$, $\| u(t) \|_\infty$, $t\in (0,T)$.
\begin{corollary}[Smoothness of strong solutions]\label{smoothness_reg_sol_corollary} If $u$ is a strong solution of the Navier--Stokes equations on $(0,T)$ and $p$ is the corresponding pressure then
\[
\p_t^k \nabla^m u, \p_t^k \nabla^m p \in C((0,T);L^2) \cap C((0,T);L^\infty )\quad \text{ for all } m,k\geq 0.
\]
In particular $u,p \in C^\infty (\RR^3 \times (0,T) )$ and $u,p$ constitute a classical solution of the Navier--Stokes equations on $\RR^3 \times (0,T)$.
\end{corollary}
\begin{proof}
From the representation of $p$, \eqref{sol_NS_form}, and the Plancherel Lemma (Lemma \ref{CZ_lemma}), we obtain that $\nabla^m p\in C((0,T);L^2)$ for all $m$. Moreover, using the Cauchy-Schwarz inequality and Lemma \ref{temp_singular_int_bound} we obtain for any $x\in \RR^3$
\begin{align*}
|p(x,t)-p(x,s)|& \leq \int \frac{|\p_i u_j (y,t) (\p_j u_i (y,t)-\p_j u_i(y,s) )|}{4\pi |x-y|} \, \d y\\ 
&\quad+ \int \frac{|\p_j u_i (y,s) (\p_i u_j (y,t)-\p_i u_j(y,s) )|}{4\pi |x-y|} \, \d y
\MoveEqLeft[0]
\\&&\mathllap{\leq C \| \nabla u (t) - \nabla u (s)\| \left( \sqrt{\int \frac{|\nabla u (y,t)|^2}{ |x-y|^2} \, \d y } +\sqrt{\int \frac{|\nabla u (y,s)|^2}{ |x-y|^2} \, \d y }  \right) }
\\&&\mathllap{\leq C  \| \nabla u (t) - \nabla u (s)\| \left( \| \nabla^2 u(t)\| +\| \nabla^2 u(s) \|  \right).}
\end{align*}
Thus $p\in C((0,T);L^\infty)$ and performing a similar calculation for each of the spatial derivatives of $p$ shows that $\nabla^m p \in C((0,T);L^\infty )$ for all $m$. Thus the distributional form of the Navier--Stokes equations \eqref{eqNSE_distr} gives $\nabla^m u_t \in C((0,T);L^2) \cap C((0,T);L^\infty )$ for all $m$.
 
 The regularity of higher derivatives in time follows by induction: regularity of $\p_t^k u$ follows from the regularity of $ u, \p_{t} u , \ldots , \p_t^{k-1} u $ and $\p_t^{k-1} p$, and by taking $(k-1)$-th time derivative (in a weak sense) of the Navier--Stokes equations, and the regularity of $\p_t^k p$ follows by taking $k$ time derivatives of the representation formula of $p$, \eqref{sol_NS_form}. 
\end{proof}
\begin{theorem}[Energy equality for strong solutions\footnote{This is {\it(3.4)} in \cite{Leray_1934}.}]\label{EE_open_int_thm}
A strong solution $u$ of the Navier--Stokes equations on $(0,T)$ satisfies
\begin{equation}\label{EE_strong_open_interval}
\| u(t_2 ) \|^2 + 2\int_{t_1}^{t_2} \| \nabla u (s) \|^2 \,\d s = \| u(t_1) \|^2
\end{equation}
for all $t_1, t_2 \in (0,T)$.
\end{theorem}
\begin{proof}
Since Theorem \ref{thmStrongSolnsSmooth} gives in particular $(u\cdot \nabla ) u \in C((0,T);L^2)$, Lemma \ref{prop_of_up_lemma}, (iii)  gives
\[
\| u(t_2) \|^2 - \| u(t_1)  \|^2 + 2 \int_{t_1}^{t_2} \| \nabla u(s) \|^2 \,\d s = 2\int_{t_1}^{t_2} \int u \cdot (u\cdot \nabla ) u\,\, \d x \,\d s .
\]
The theorem follows by noting that the right-hand side vanishes: integration by parts and the incompressibility constraint, $\p_k u_k=0$, give
\begin{equation}\label{temp_cancellation_EE}
\int u\cdot (u\cdot \nabla ) u = \int u_i \, u_k \,\p_k u_i = - \int \p_k u_i \, u_k \,  u_i ,
\end{equation}
that is $ \int u \cdot (u\cdot \nabla ) u =0$.
\end{proof}
We now show that we can control the separation of two strong solutions.
\begin{lemma}[Comparison of two strong solutions\footnote{This is Section 18 of \cite{Leray_1934}.}]\label{comparison_two_sols_lem}
Suppose that $u$, $v$ are strong solutions to the Navier--Stokes equations on $(0,T)$ and let
\[
w\coloneqq u-v .
\]
Then 
\begin{equation}\label{comparison_two_sols_inequality}
\| w(t_2)\|^2 \leq  \| w(t_1) \|^2 \mathrm{e}^{\frac{1}{2} \int_{t_1}^{t_2}  \| u(s) \|^2_\infty \,\d s}
\end{equation}
for $t_1, t_2\in (0,T)$ with $t_1<t_2$. 
\end{lemma}
In particular, if $u$, $v$ coincide at time $t_1$ then they continue to coincide for the later times. We will extend this uniqueness property to account for the initial data in Section \ref{sec_local_exist_uni_strong} (Lemma \ref{uniqueness_and_EE_strong_sols_lem}).
\begin{proof}
Since both $u$ and $v$ satisfy the Navier--Stokes equations pointwise, subtracting them gives
\[
\p_t w  - \Delta w+\nabla q = -(u \cdot\nabla) u + (v\cdot \nabla) v = - (v \cdot \nabla ) w - (w\cdot \nabla ) u.
\]
As in \eqref{temp_cancellation_EE} we have $\int w\cdot (v\cdot \nabla ) w =0$ and hence multiplying the above equality by $w$, integrating by parts in spatial variables, using the incompressibility constraint, $\p_k w_k =0$, we obtain for $t\in (0,T)$
\[
\begin{split}
\frac{1}{2}\frac{\d}{\d t}\|w(t)\|^2+ \|\nabla w(t)\|^2&=-\int w \cdot \left( (v\cdot\nabla)w + (w\cdot\nabla)u \right) \\
&= \int w_i \, w_k \, \p_k u_i = - \int  \p_k w_i \, w_k \,  u_i \\
&\leq    \|\nabla w(t)\|\|w(t)\|\|u(t)\|_{\infty}\\
&\leq \| \nabla w(t) \|^2 + \frac{1}{4} \|w(t)\|^2 \|u(t)\|_{\infty}^2,
\end{split}
\]
where we also used the Cauchy-Schwarz and Young inequalities (and we omitted the argument ``$t$'' under the integrals). Hence
\[
\frac{\d}{\d t}\|w(t)\|^2 \leq \frac{1}{2} \|w(t)\|^2 \|u(t)\|_{\infty}^2,
\]
and the claim follows by applying Gronwall's inequality.
\end{proof}
Finally, the remark after the proof of Theorem \ref{thmStrongSolnsSmooth} suggests the following convergence property of a family of strong solutions of the Navier--Stokes equations.
\begin{lemma}[Convergence lemma\footnote{This is Lemma 9 in \cite{Leray_1934}.}]\label{lem9_from_Leray}
Suppose $\{ u_\varepsilon \}_{\varepsilon >0}$ is a family of strong solutions of the Navier-Stokes equations such that 
\[\| u_\varepsilon (t) \|_{\infty}, \| u_\varepsilon (t) \| \leq f(t)\qquad \text{ for } t\in (0,T), 
\]
where $f$ is a continuous function on $(0,T)$. Then there exists a sequence $ \varepsilon_k \to 0^+$ and a function $u$, such that $u_{\varepsilon_k} \to u ,\, \nabla u_{\varepsilon_k} \to \nabla u$  uniformly on compact sets in  $\RR^3 \times (0,T)$ as $\varepsilon_k \to 0^+$.

Moreover $u$ is a strong solution of the Navier-Stokes equations in $(0,T)$ and satisfies $\| u (t) \|_{\infty}, \| u (t) \| \leq f(t) $ for all $t\in (0,T)$. 
\end{lemma}
\begin{proof}
Let $p_\varepsilon (t) $ denote the pressure function corresponding to $u_\varepsilon$ (which is determined uniquely by the representation formula \eqref{sol_NS_form} with $u$ replaced by $u_\varepsilon$). Fix a bounded domain $\Omega \subset \RR^3$ and $\delta >0$. We see from Corollary \ref{smoothness_reg_sol_corollary} that for a multi-index $\alpha $ and $m$ such that $m, |\alpha |\leq 3 $ 
\[
|\p_{t^m} D^\alpha u_\varepsilon |,| \p_{t^m} D^\alpha p_\varepsilon | \leq C_{\Omega , \delta} \quad \text{ on } \Omega_\delta \times (\delta /2 , T- \delta /2)  , 
\]
where $\Omega_\delta \coloneqq \Omega + B(0,\delta )$. Thus an application of the Arzel\`a--Ascoli theorem and a simple diagonalization argument produces a sequence $\{ \varepsilon_k \}$ such that all derivatives $\p_{t^m} D^\alpha u_{\varepsilon_k} $, $\p_{t^m} D^\alpha p_{\varepsilon_k}$ with $m,|\alpha |\leq 2$ converge to the respective derivatives of $u$ and $p$ uniformly on $\Omega \times (\delta , T-\delta )$, for some functions $u$, $p$. In particular  the Navier--Stokes equations
\[
\begin{split} \p_t u_{\varepsilon_k} - \Delta u_{\varepsilon_k} +\nabla p_{\varepsilon_k} &= -(u_{\varepsilon_k} \cdot\nabla)u_{\varepsilon_k},\\
\mathrm{div}\, u_{\varepsilon_k}&=0
\end{split}
\]
converge uniformly on $\Omega \times (\delta , T-\delta )$ to Navier--Stokes equations for $u$ (in the sense that all terms converge). Now consider a sequence of bounded sets $\Omega_n\nearrow \RR^3$, a sequence $\delta_n \to 0^+$ and apply another diagonal argument to obtain a subsequence $\{ \varepsilon_k \}$ (which we relabel) such that for $m$, $\alpha$ with $m,|\alpha |\leq 2$
\[
\p_{t^m} D^\alpha u_{\varepsilon_k}  \to \p_{t^m} D^\alpha u,\quad \p_{t^m} D^\alpha p_{\varepsilon_k} \to  \p_{t^m} D^\alpha p
\]
uniformly on compact sets in $\RR^3 \times (0,T)$. In particular $u$, $p$ satisfy the Navier--Stokes equations pointwise, and thus also in the sense of distributions \eqref{eqNSE_distr}. That $\| u(t) \|_\infty \leq f(t)$ holds for all $t$ is clear, and the inequality $\| u(t) \|\leq f(t)$ follows by an application of Fatou's lemma. According to Definition \ref{def_strong_open} it remains to verify that $u\in C((0,T);L^2)\cap C((0,T);L^\infty )$. For this let $I$ be a closed interval in $(0,T)$ and note that there exists $M>0$ such that $\| \p_t u_{\varepsilon_k} (t) \|_\infty , \| \p_t u_{\varepsilon_k } (t) \| \leq M$ for $t\in I$, $k\geq 0$ (see Corollary \ref{smoothness_reg_sol_corollary}). Thus the mean value theorem gives 
\[
\| u_{\varepsilon_k } (t) - u_{\varepsilon_k } (s) \|_\infty , \, \| u_{\varepsilon_k }(t) - u_{\varepsilon_k }(s) \|  \leq M |t-s| \quad \text{ for } s,t \in I,\,k\geq 0.
\] 
Thus taking the limit in $k$ (and applying Fatou's lemma) gives the required continuity.
\end{proof}

\subsection{Local existence and uniqueness of strong solutions}\label{sec_local_exist_uni_strong}
In this section we study the Navier--Stokes initial value problem, that is we consider the equations \eqref{eqNSE}, \eqref{eqNSE_incomp} with initial data. For this reason we extend the definition of strong solutions (Definition \ref{def_strong_open}) to the half-closed time interval $[0,T)$.
\begin{definition}\label{def_reg_sol_halfclosed}
A function $u$ is a \emph{strong solution of the Navier--Stokes equations on $[0,T)$} if for some $p\in L_{\loc}^1 (\RR^3 \times [0,T))$ 
\begin{equation}\label{eqNSE_distr_incl_zero}
\int u(0) \cdot f(0)  +\int_0^T  \int (u \cdot (f_t + \Lap f) + p \, \mathrm{div}\, f) = \int_0^T \int u \cdot (u \cdot \nabla) f
\end{equation}
for all $f\in C_0^\infty (\RR^3 \times [0,T);\RR^3)$, $u(t)$ is weakly divergence free for $t\in (0,T)$, and 
\eqnb\label{temp_regularity_in_def}
 u\in C([0,T); L^2)\cap C((0,T);L^\infty )
 \eqne
with $\| u (t) \|_\infty$ bounded as $t\to 0^+$.
\end{definition}
The regularity \eqref{temp_regularity_in_def} is a part of Leray's definition of solutions on the time interval $[0,T)$, but he also requires $\nabla u \in C([0,T);L^2)$ and that $u$ and its spacial derivatives are continuous at $t=0$ (see pp. 220-221 of his paper). It is remarkable that by use of the weak formulation \eqref{eqNSE_distr_incl_zero}, these additional assumptions are not necessary for showing local well-posedness (see Theorem \ref{local_exist_strong_thm} below). Moreover, the above definition is much less restrictive than the commonly used definition of strong solutions, which usually requires
\[
u \in L^\infty_{\mathrm{loc}} ([0,T);H^1 ) \cap L^2_{\mathrm{loc}} ([0,T);H^2 ),
\]
and so consequently $u\in C([0,T);H^1)$ see, for example, Definition 6.1 and the following discussion in \cite{JCR_R_S_NSE_book}.
In particular, Definition \ref{def_reg_sol_halfclosed} makes no assumption on the regularity of $\nabla u(t)$ for times $t$ near $0$.   

Since Definition \ref{def_reg_sol_halfclosed} is an extension of the definition of the strong solution on the open time interval $(0,T)$ (Definition \ref{def_strong_open}), we see that $u$, $p$ admit the representation formulae \eqref{sol_NS_form}. Moreover, now the representation formula also holds for $t_1=0$, 
\begin{equation}\label{repr_half_closed}
u(t)= \Phi (t)\ast u (0) +\int_{0}^{t}  \nabla \mathcal{T}(t-s )\ast \left[ u(s)\, u(s) \right]  \,\d s
\end{equation}
for $t\in (0,T)$, a consequence of the definition above and Theorem \ref{sol_distributions_anyX_theorem}.

We also see that \eqref{eqNSE_distr_incl_zero} is equivalent to
\begin{equation}\label{eqNSE_distr_incl_zero_more_flex}
\begin{split}
\int u(0) \cdot f(0)  &-\int u(t) \cdot f(t)   \\
&+\int_0^t  \int (u \cdot (f_t + \Lap f) + p \, \mathrm{div}\, f) = \int_0^t \int u \cdot (u \cdot \nabla) f
\end{split}
\end{equation}
being satisfied for all $t \in (0,T)$, $f\in C_0^\infty (\RR^3 \times [0,T);\RR^3)$ (see the discussion following Definition \ref{def_strong_open}). A consequence of this fact is that a strong solution on the time interval $[0,T)$ is also a strong solution on the time interval $[\tau , T)$ for any $\tau \in (0,T)$.

Given the definition of strong solutions on the half-closed time interval $[0,T)$, we immediately obtain the energy equality \eqref{EE_strong_open_interval} with $t_1=0$ and the uniqueness of strong solutions.
\begin{lemma}[Uniqueness and the energy equality for local strong solution]\label{uniqueness_and_EE_strong_sols_lem}
A strong solution $u$ to the Navier--Stokes equations on $[0,T)$ satisfies the \emph{energy equality}
\begin{equation}\label{EE_half_closed}
\| u(t ) \|^2 + 2\int_{0}^{t} \| \nabla u (s) \|^2 \,\d s = \| u_0 \|^2
\end{equation}
for $t\in [0,T )$.

Moreover, if $v$ is another strong solution to the Navier--Stokes equations on $[0,T)$ with $v(0)=u(0)$ then $u\equiv v$. 
\end{lemma}
\begin{proof} The claim follows by taking the limit $t_1 \to 0^+$ in \eqref{EE_strong_open_interval} and \eqref{comparison_two_sols_inequality} and applying the Monotone Convergence Theorem.
\end{proof}
From the lemma we obtain the semigroup property: if $\tau \in (0,T_1)$ and $u$, $\widetilde{u}$ are strong solutions on the intervals $[0,T_1)$, $[\tau , T_2)$ respectively with $\widetilde{u} (\tau )= u(\tau )$, then $\widetilde{u}=u$ on the time interval $ [\tau , \min \{T_1,T_2\} )$.
We now state the central theorem regarding strong solutions to the Navier--Stokes equations.
\begin{theorem}[Local existence of strong solutions]\label{local_exist_strong_thm}
If $u_0 \in H  \cap L^\infty $ (that is $u_0 \in L^2 \cap L^\infty $ is weakly divergence free, see \eqref{def_of_HV}) then there exists a unique strong solution $u$ of the Navier--Stokes equations on $[0,T)$ with $u(0)=u_0$, where $T>C/\| u_0 \|_{\infty }^2$.
\end{theorem}
This theorem is proved by Leray in Section 19 in the case when $u_0\in C^1\cap H^1 \cap L^\infty$. Here, we use the distributional form of equations to relax this regularity requirement to only $u_0 \in H\cap L^\infty$, and to demonstrate how Leray's original proof can be simplified while exposing his main ideas.
\begin{proof}
Uniqueness is guaranteed by Lemma \ref{uniqueness_and_EE_strong_sols_lem}. As for existence, we consider the following iterative definition of $u^{(n)}$:
\[
u^{(0)} (t) \coloneqq   \Phi(t) \ast u_0 
 \]
and
\[
u^{(n+1)}(t) \coloneqq\int_0^t \nabla \mathcal{T}(t-s) \ast \left[ u^{(n)}(s)\, u^{(n)} (s) \right] \,\d s + u^{(0)}(t),
\]
using the notation from \eqref{notation_oseen_ker_conv}. From properties of the heat kernel (see (iii) in Appendix \ref{heat_section_appendix}) we have 
\[
u^{(0)} \in C([0,\infty ); L^2)\cap C((0,\infty ); L^\infty ),
\]
$\|u^{(0)}(t)\|_{\infty}\leq \|u_0\|_{\infty}$ and $\|u^{(0)}(t)\|\leq \|u_0\|$ for $t\geq 0$. Moreover, using induction we deduce from Lemma \ref{additional_prop_up_lemma}, (i) and (iii) (applied with $Y\coloneqq u^{(n)}$) that for all $n\geq 0$, $t\geq 0$,
\[
u^{(n+1)} \in C([0,\infty );L^2)\cap C((0,\infty ) ; L^\infty ),
\] 
\begin{equation}\label{eqSmoothUIterationEst1}
\|u^{(n+1)}(t)\|_{\infty}\leq {C'} \int_0^t\frac{\|u^{(n)}(s)\|_{\infty}^2}{\sqrt{t-s}}\,\d s+\|u_0\|_{\infty},
\end{equation}
and
\begin{equation}\label{eqSmoothUIterationEst2}
\|u^{(n+1)}(t)\|\leq C' \int_0^t\frac{\|u^{(n)}(s)\|_{\infty} \|u^{(n)}(s)\|  }{\sqrt{t-s}}\,\d s+\|u_0\|,
\end{equation}
Theorem \ref{sol_distributions_anyX_theorem} guarantees that if
\[
p^{(n+1)} (t) \coloneqq  \p_k \p_i  (-\Delta )^{-1} \left( u^{(n)}_i (t)\, u_k^{(n)}(t) \right) 
\]
(recall \eqref{def_of_inverse_laplace} for the notation) then for each $n\geq 0$ the pair $u^{(n+1)},p^{(n+1)}$ is a distributional solution of the problem
\[
\p_t  u^{(n+1)} -\Lap u^{(n+1)} +\nabla p^{(n+1)} = (u^{(n)}\cdot \nabla ) u^{(n)},\qquad \mathrm{div}\, u^{(n+1)} =0 
\]
with initial condition $u^{(n+1)} (0) = u_0$,
that is $u^{(n+1)}$ is weakly divergence free and
\begin{multline}\label{temp_existence_pf}
\int u_0 \cdot f(0) \,\d x+\int_0^\infty  \int \left( u^{(n+1)} \cdot (f_t + \Delta f )+ p^{(n+1)}\, \mathrm{div} \, f \right)\\  =  \int_0^\infty \int u^{(n)} \cdot (u^{(n)}\cdot \nabla ) f 
\end{multline}
for all $f\in C_0^\infty (\RR^3 \times [0,\infty )$, $n\geq 0$. In order to take the limit in $n$ we will find a uniform bound on $\| u^{(n+1)} \|_\infty $ on the finite time interval $[0,T ]$, where
\[
T \coloneqq \frac{1}{32  (1+C')^4 \| u_0 \|_\infty^2 }.
\]
For such choice of $T >0$ the constant function $\phi(t)\coloneqq (1+{C'})\|u_0\|_{\infty}$ satisfies the integral inequality
\begin{equation}\label{eqSmoothSolnExistTime}
\phi(t)\geq C'\int_0^t\frac{\phi^2(s)}{\sqrt{t-s}}\,\d s+\|u_0\|_{\infty}
\end{equation}
for all $t\in [0,8T )$. A use of induction and \eqref{eqSmoothUIterationEst1} thus gives 
\begin{equation}\label{temp3_existence_pf}
\|u^{(n)}(t)\|_{\infty}\leq \phi(t).
\end{equation}
for all such $t$'s and all $n\geq 0$. Now noting that for all $j,k=1,2,3$, $n\geq 1$
\[
u_k^{(n)}u_j^{(n)} - u_k^{(n-1)}u_j^{(n-1)} =  u_k^{(n)}\left( u_j^{(n)}- u_j^{(n-1)}\right)+ u_j^{(n-1)} \left(u_k^{(n)} -  u_k^{(n-1)}\right)
\]
we use Corollary \ref{additional_prop_up_corollary} twice (first with $Y\coloneqq  u^{(n)} $, $Z\coloneqq  u^{(n)}-u^{(n-1)} $ and then with $Y\coloneqq  u^{(n-1)} $, $Z\coloneqq  u^{(n)}-u^{(n-1)} $) to obtain
\begin{subequations}\label{the_two_inequalities}
\begin{multline}
\|u^{(n+1)}(t)-u^{(n)}(t)\|_{\infty} \\\leq C' \int_0^t \frac{\left( \| u^{(n)} (s) \|_\infty + \| u^{(n-1)} (s) \|_\infty \right) \|u^{(n)} (s) - u^{(n-1)} (s) \|_\infty }{\sqrt{t-s}} \,\d s,
\end{multline}
\begin{multline}
\|u^{(n+1)}(t)-u^{(n)}(t)\| \\\leq C' \int_0^t \frac{\left( \| u^{(n)} (s) \|_\infty + \| u^{(n-1)} (s) \|_\infty \right) \|u^{(n)} (s) - u^{(n-1)} (s) \| }{\sqrt{t-s}} \,\d s .
\end{multline}
\end{subequations}
Applying \eqref{temp3_existence_pf} to the second of the above inequalities gives for $t\in [0,T]$
\begin{align*}
%\begin{split}
\MoveEqLeft
\|u^{(n+1)}(t)-u^{(n)}(t)\|\\ &\leq 2C'(1+C') \| u_0 \|_\infty \| u^{(n)}-u^{(n-1)} \|_{C([0,T],L^2 )}\int_0^t \frac{1}{\sqrt{t-s}} \,\d s \\
&= 4C'(1+C') \| u_0 \|_\infty \sqrt{t}  \| u^{(n)}-u^{(n-1)} \|_{C([0,T ],L^2 )} \\
&&\mathllap{\leq \lambda  \| u^{(n)}-u^{(n-1)} \|_{C([0,T ],L^2 )},}
%\end{split}
\end{align*}
where $\lambda \coloneqq 4C'(1+C') \| u_0 \|_\infty \sqrt{T }$. Hence
\[
\|u^{(n+1)}-u^{(n)}\|_{C([0,T ],L^2 )} \leq \lambda  \| u^{(n)}-u^{(n-1)} \|_{C([0,T ],L^2 )}
\]
Because the definition of $T $ implies $\lambda \in (0,1)$ we see that $\{ u^{(n)} \}$ is a Cauchy sequence in $C([0,T ]; L^2 )$ and so 
\[
u^{(n)} \to u     \quad \text{ in } C([0,T ];L^2 )
\] 
for some $u\in C([0,T ];L^2 )$ such that $u(0)=u_0$ and $u(t)$ is weakly divergence free for each $t\in [0,T]$ (since $u^{(n)}(0)=u_0$ and $u^{(n)}(t)$ is weakly divergence free for each $n$). Similarly, applying \eqref{temp3_existence_pf} to the first inequality in \eqref{the_two_inequalities} gives
\[
\|u^{(n+1)}-u^{(n)}\|_{C([0,T ];L^\infty )} \leq \lambda  \| u^{(n)}-u^{(n-1)} \|_{C([0,T ];L^\infty )} .
\]
Although this does not imply that $\{ u^{(n)} \}$ is Cauchy in $C([0,T ]; L^\infty  )$ (recall each $u^{(n)}$ need not belong to this space; it is continuous into $L^\infty$ only on the open time interval $(0,\infty)$), it does follow that $\{ u^{(n)} \}$ is a Cauchy sequence in $C([\delta ,T ]; L^\infty  )$ for any $\delta \in (0,T)$, and therefore 
\begin{equation}\label{conv_linfty_delta_to_T}
u^{(n)} \to u     \quad \text{ in } C([\delta ,T ];L^\infty )
\end{equation}
for each $\delta$. Note that the limit function is $u$ since $L^2$ convergence implies convergence almost everywhere on a subsequence.
Therefore 
\[
u \in C([0,T ];L^2 ) \cap C((0,T ];L^\infty )
\]
and \eqref{temp3_existence_pf} gives $\| u(t) \|_\infty \leq (1+C') \| u_0 \|_\infty$ for all $t\in [0,T]$. Letting
\[
p (t) \coloneqq \p_i \p_k (-\Delta )^{-1} \left( u_i (t) \, u_k (t) \right) 
\]
we see that the Plancherel Lemma (Lemma \ref{CZ_lemma}) gives that
\[
p^{(n)} \to p \quad \text{ in } C([0,T ];L^2 ) .
\]
Therefore, taking the limit $n\to \infty $ in \eqref{temp_existence_pf} we obtain
\[
\int u_0 \cdot f \, \d x +\int_0^T \int \left( u \cdot (f_t + \Delta f )+ p\, \mathrm{div} \, f \right)  =  \int_0^T \int u \cdot (u\cdot \nabla ) f 
\]
for all $f\in C_0^\infty (\RR^3 \times [0,T ))$. Thus $u$ is a strong solution of the Navier--Stokes equations on $[0,T )$.
\end{proof}
Note that if $u_0$ is more regular then the continuity as $t\to 0^+$ of the corresponding strong solution $u$ on $[0,T)$ may be stronger. This is, in essence, the issue of continuity as $t\to 0^+$ of the solution of the heat equation (cf. Lemma \ref{Linfty_conv_heat_eq}). If $u_0$ is uniformly continuous then the representation formula \eqref{repr_half_closed} gives that $u(t)\to u_0 $ in $L^\infty$ as $t\to 0^+$ (cf.\ the proof of Lemma \ref{additional_prop_up_lemma} (i)). Furthermore, for such $u_0$ the proof above simplifies since each $u^{(n)}$ belong to $C([0,\infty); L^\infty )$ (cf. the same lemma) and so they converge in $C([0,T];L^\infty )$ rather than in $C([\delta ,T];L^\infty )$ for all $\delta$'s as in \eqref{conv_linfty_delta_to_T}. 
\subsection{Characterisation of singularities}\label{char_of_sing_section}
Here we investigate the maximal time of existence of strong solutions and derive the rates of blow-up of $u$ in various norms at an approach to a (putative) blow-up time.

Let $u_0 \in  V \cap L^\infty$ (that is $u_0\in H^1 \cap L^\infty$ is divergence free, see \eqref{def_of_HV}), let $u$ be the strong solution of the Navier--Stokes equations starting from $u_0$ and let $T_0$ be its maximal time of existence, that is $u$ cannot be extended to a solution on $[0,T')$ for any $T'>T_0$.
Note that Theorem \ref{local_exist_strong_thm} gives that $T_0 \geq C/\| u_0 \|_\infty^2$ and, if $T_0$ is finite,
\[\| u(t) \|_\infty \quad \text{ blows up as }\quad t\to T_0^- , \]
as otherwise we could extend $u$ beyond $T_0$ and hence obtain a contradiction.

In this section we will apply the theory of integral inequalities (see Lemma \ref{lem_integral_ineq}) to bound the $L^\infty$ norm of a strong solution $u$ on some time interval starting from $0$ and thus obtain lower bounds on $T_0$ as well as lower bounds on $\| u(t) \| _\infty$, $\| \nabla u ( t)\|$ and $\| u(t) \|_p$ with $p>3$ when $t\to T_0^-$ (if $T_0$ is finite).

Since \eqref{repr_half_closed} holds for $t\in (0,T_0)$, Lemma \ref{additional_prop_up_lemma} (i) gives that for all $t\in [0,T_0)$, $p>3$
\begin{equation}\label{char_irreg_fund_rel_infty}
\| u(t) \|_\infty \leq C' \int_0^t \frac{ \| u(s) \|_\infty^2 }{\sqrt{t-s}}\,\d s + \min \left( \| u_0 \|_\infty , C' \frac{\| \nabla u_0 \|}{t^{1/4}}, C'\frac{\| u_0 \|_p }{t^{3/2p}} \right).
\end{equation}
Here, the minimum on the right-hand side is obtained by applying Young's inequality for convolutions \eqref{young_for_convolutions} to $u_1(t) = \Phi(t) \ast u_0$ (with exponents $(1,\infty )$, $(6/5,6)$ and $(p/(p-1),p)$ respectively), the fact that $\| \Phi (t) \|_p \leq C/t^{-3(p-1)/2p}$, and the Gagliardo-Nirenberg-Sobolev\footnote{Note that this inequality were not available in the 1930's. Instead, Leray used Lemma \ref{temp_singular_int_bound} to obtain
\[
\left|\int \Phi(x-y,t)u_0(y)\ \d y\right|\leq 2\left|\int|\Phi(x-y,t)|^2|x-y|^2\right|^{1/2}\left\|\nabla u_0\right\| \leq C \left\|\nabla u_0\right\| t^{-1/4}.
\]} inequality $\| u_0 \|_6 \leq C \| \nabla u_0 \|$. Similarly, the convolution under the time integral in the representation formula \eqref{repr_half_closed} can be bounded in terms of $  \|\nabla \mathcal{T}(t-s) \|_\infty \| u(s) \|^2 $ (rather than by $ \| \nabla \mathcal{T}(t-s) \|_1 \| u(s) \|_\infty^2$ as in Lemma \ref{additional_prop_up_lemma}, (i)), and this gives for $t\in [0,T_0)$,
\begin{equation}\label{char_irreg_fund_rel_infty1}
\| u(t) \|_\infty \leq C'' \int_0^t  \min \left(  \frac{ \| u(s) \|_\infty^2 }{\sqrt{t-s}} ,  \frac{ \| u_0 \|^2 }{(t-s)^2}\right) \,\d s +  \| u_0 \|_\infty ,
\end{equation}
where we also used the facts $\| \nabla \mathcal{T} (t) \|_\infty \leq C t^{-2}$ (see \eqref{eqOseenKerEst2}) and $\| u (t) \| \leq \| u_0 \|$ (see the energy equality \eqref{EE_half_closed}). The two inequalities above\footnote{These comprise {\it(3.5)} in \cite{Leray_1934}.} allow us to obtain bounds on $\| u(t) \|_\infty$ in terms of various norms of the initial data.
\begin{lemma}\label{bounds_on_infty_norm_lemma}\footnote{This lemma and the two following corollaries correspond to Sections 21 and 22 in \cite{Leray_1934}.}If $u$ is the strong solution with initial data $u_0\in V \cap L^\infty$ then
\begin{enumerate}[\rm(i)]
\item$\| u(t) \|_\infty \leq C\| u_0 \|_\infty$ for $t\leq C / \| u_0 \|^2_\infty $,
\item $\| u(t) \|_\infty \leq C \| \nabla u_0 \| t^{-1/4} $ for $t\leq C / \| \nabla u_0 \|^4 $,
\item $\| u(t) \|_\infty \leq C\| u_0 \|_p t^{-3/2p}$ for $t\leq (C (1-3/p) / \| u_0 \|_p )^{2p/(p-3)}$, $p>3$.
\end{enumerate}
Moreover, there exists $\varepsilon >0$ such that $\| u(t) \|_\infty \leq C \| u_0 \|_\infty$ for all $t\geq 0$  if $\| u_0 \|^2 \, \| u_0 \|_\infty < \varepsilon$.
\end{lemma}
\begin{proof} As in the proof of Theorem \ref{local_exist_strong_thm}, we note that the constant function $\phi(t) = C\| u_0 \|_\infty $ satisfies the integral inequality \eqref{eqSmoothSolnExistTime} for $t\in [0,C / \| u_0 \|^2_\infty ]$ and (i) follows from this, \eqref{char_irreg_fund_rel_infty} and from the theory of integral inequalities (see Lemma \ref{lem_integral_ineq}). Similarly, a direct calculation shows that the function $\phi (t) = C \| \nabla u_0 \| t^{-1/4}$ satisfies the integral inequality
\[
\phi (t) \geq C' \int_0^t \frac{ \phi(s)^2 }{\sqrt{t-s}} \,\d s + C' \| \nabla u_0 \| t^{-1/4}
\]
for $t\in (0, C/ \| \nabla u_0 \|^{4} ]$, and so (ii) follows. One can also check that the function $\phi (t) = C \|  u_0 \|_p t^{-3/2p}$ satisfies the integral inequality
\[
\phi (t) \geq C' \int_0^t \frac{ \phi(s)^2 }{\sqrt{t-s}} \,\d s + C' \|  u_0 \|_p t^{-3/2p}
\]
for $t\in (0,(C (1-3/p) / \| u_0 \|_p )^{2p/(p-3)}]$, and (iii) follows.

The last claim follows from \eqref{char_irreg_fund_rel_infty1} and the fact that the constant function $\phi (t) = C \| u_0 \|_\infty$ satisfies the integral inequality
\begin{equation}\label{smallness_condition_temp}
\phi(t)  > C'' \int_0^t \min \left(  \frac{ \phi (s)^2  }{\sqrt{t-s}} ,  \frac{ \| u_0 \|^2 }{(t-s)^2}\right) \,\d s +  \| u_0 \|_\infty ,
\end{equation}
for all $t>0$ if and only if
\[
\| u_0 \|_\infty > C \int_0^\infty \min \left( \frac{\| u_0 \|_\infty^2 }{\sqrt{s}}, \frac{\| u_0 \|^2 }{s^2} \right) \,\d s .
\]
One can check that this last condition is equivalent to the smallness condition $\| u_0 \|^2 \, \| u_0 \|_\infty < \varepsilon$ for some $\varepsilon >0$. Therefore, the integral inequalities \eqref{smallness_condition_temp} and \eqref{char_irreg_fund_rel_infty1} show that, given the smallness condition, $\| u(t) \|_\infty \leq \phi (t) $ for all $t\geq 0$ (where we apply another fact from the theory of integral inequalities, see Corollary \ref{integral_ineq_min_cor}).\end{proof}
From the lemma we immediately obtain lower bounds on the maximal time of existence $T_0$, the rates of blow-up of norms $\| u(t) \|_\infty$, $\| \nabla u (t) \|$ and $\| u(t) \|_p$, where $p>3$, as $t \to T_0^-$, as well as global existence result for small data, which we formulate in the following three corollaries.
\begin{corollary}[Lower bounds on the existence time $T_0$]\label{existence_time_corollary}If $T_0$ is the maximal time of existence of the strong solution $u$ with initial data $u_0 \in V \cap L^\infty$ then
\begin{enumerate}[\rm(i)]
\item $T_0 > C / \| u_0 \|_\infty^2$,
\item $T_0> C/ \| \nabla u_0 \|^{4}$,
\item $T_0 > \left( C \left( 1-\frac{3}{p} \right)/ \| u_0 \|_p\right)^{2p/(p-3)}$ for all $p>3$.
\end{enumerate}
\end{corollary}
\begin{proof} It is a direct consequence of Lemma \ref{bounds_on_infty_norm_lemma} (i)-(iii).
\end{proof}
\begin{corollary}[Blow-up rates]\label{blowup_rates_corollary}
If $u$ is a strong solution of the Navier--Stokes equations on the time interval $(T,T_0)$, where $T_0 <\infty$ is the maximal existence time, then for $t\in (T,T_0)$,
\[
\| u(t) \|_\infty \geq \frac{C}{\sqrt{T_0-t}},\quad \|\nabla u(t) \| \geq \frac{C}{(T_0-t)^{1/4}},
\]
and, for $p>3$,
\[\| u(t) \|_p \geq \frac{C^{(1-3/p)/2}(1-3/p)}{(T_0-t)^{(1-3/p)/2}}.
\]
\end{corollary}
\begin{proof}
Let $t\in (T,T_0)$. Since $u(t)\in H \cap L^\infty$, the local existence and uniqueness theorem (Theorem \ref{local_exist_strong_thm}) gives that $(T_0-t) \geq C/\| u (t) \|^2_\infty$, which gives the first bound. The other two follow in a similar way using Corollary \ref{existence_time_corollary}, (ii) and (iii).
\end{proof}
\begin{corollary}[Global existence for small initial data]
There exists $\varepsilon >0$ such that if either $ \| u_0 \|^2 \, \|  u_0 \|_\infty < \varepsilon$, $\| u_0 \| \, \| \nabla u_0 \| <\varepsilon$ or
\eqnb\label{smallness_lp} \left( C\| u_0 \| \right)^{2(p-3)} \| u_0 \|_p^p < C\left(1- \frac{3}{p} \right)\varepsilon^{p-3}\qquad \text{ for any } p>3
\eqne
then $T_0 = \infty $, that is the strong solution with initial data $u_0$ exists for all times.
\end{corollary}
\begin{proof}The first claim follows directly from the last claim of Lemma \ref{bounds_on_infty_norm_lemma}. As for the smallness condition on $\| u_0 \| \, \| \nabla u_0 \| $, let $t_0 \coloneqq C / \| \nabla u_0 \|^4$, the endpoint  time in Lemma \ref{bounds_on_infty_norm_lemma} (ii). Then
\[
\| u(t_0 ) \|_\infty \leq C \| \nabla u_0 \| t_0^{-1/4} = C \| \nabla u_0 \|^2,
\]
and so, using the energy equality \eqref{EE_half_closed}, we obtain 
\[\| u(t_0 )\|^2 \, \| u(t_0 ) \|_\infty \leq C \| u_0 \|^2 \| \nabla u_0 \|^2 .\]
 In other words, the condition $\| u(t_0 ) \|^2 \, \| u(t_0 ) \|_\infty  < \varepsilon$ holds if $\| u_0 \|\, \| \nabla u_0 \|$ is sufficiently small, as required.

As for the smallness condition on $(C\| u_0 \|)^{2(p-3)} \| u_0 \|_p^p$, take $t_0 \coloneqq (C (1-3/p) / \| u_0 \|_p )^{2p/(p-3)}$, the endpoint  time in Lemma \ref{bounds_on_infty_norm_lemma} (iii). Then
\[
\| u(t_0 ) \|_\infty \leq C \| u_0 \|_p t_0^{-3/2p} = C \| u_0 \|_p^{p/(p-3)} \left( C \left( 1- \frac{3}{p} \right) \right)^{-3/(p-3)}.
\]
Thus the energy equality \eqref{EE_half_closed} and \eqref{smallness_lp} gives 
\[\| u(t_0 ) \|^2 \| u(t_0 ) \|_\infty \leq \| u_0 \|^2 \| u(t_0 ) \|_\infty < \varepsilon ,\]
as required.
\end{proof}
Finally, we deduce the following result\footnote{This corollary is a consequence of Leray's {\it(3.19)}.}, which we will only use later in analysing the structure of a weak solution (Theorem \ref{thmSupInfo}).
\begin{corollary}\label{2nd_char_cor} 
Let $u$ be a strong solution of the Navier--Stokes equations on the time interval $(T,T_0)$, let $t_1\in (T,T_0)$ and $t_2>t_1$. If $t_2-t_1 \leq C \| \nabla u (t_1) \|^{-4}$ then
\[
\| u (t_2) \|_\infty \leq C \frac{ \| \nabla u(t_1) \| }{(t_2-t_1)^{1/4}},\quad \text{ and }\quad \| \nabla u (t_2) \| \leq C \| \nabla u(t_1) \| .
\]
\end{corollary}
\begin{proof}
The first inequality follows directly from Lemma \ref{bounds_on_infty_norm_lemma} (ii). For the second one, note that from the smoothness of strong solutions (see Corollary \ref{smoothness_reg_sol_corollary}) we have $(u\cdot \nabla )u\in C([t_1,T_0) ; L^2)$, so by the representation \eqref{sol_NS_form} and Lemma \ref{prop_of_up_lemma} (ii), we obtain\footnote{This is {\it(3.6)} in \cite{Leray_1934}.}
\[\begin{split}
\| \nabla u (t) \| &\leq C \int_{t_1}^{t} \frac{\| \nabla u (s) \| \| u(s) \|_\infty }{\sqrt{t-s}} \,\d s+ \| \nabla u (t_1) \| \\
&\leq C''' \| \nabla u(t_1) \| \int_{t_1}^{t} \frac{\| \nabla u (s) \|  }{\sqrt{t-s} (s-t_1)^{1/4}} \,\d s + \| \nabla u (t_1) \|
\end{split}
\] 
for all $t\in [t_1, T_0)$. Now, a direct calculation shows that the constant function $\phi (t) \coloneqq C \| \nabla u(t_1 )\|$ satisfies the integral inequality
\[
\phi(t) \geq C'''\| \nabla u(t_1) \| \int_{t_1}^{t} \frac{\phi (s)  }{\sqrt{t-s} (s-t_1)^{1/4}} \,\d s + \| \nabla u (t_1) \|
\]
for $t\in [t_1, t_1+C/\| \nabla u(t_1 ) \|^4 ]$. Therefore, 
\[ \| \nabla u(t) \| \leq \phi (t) \quad \text{ for } t\in [t_1, t_1+C/\| \nabla u(t_1 ) \|^4 ], \]
where we also used a fact from the theory of integral inequalities, see Corollary \ref{integral_ineq_power1_cor}. Thus we obtain the second of the required inequalities.
\end{proof}
\subsection{Semi-strong solutions}\label{sec_semi_reg}
In this section we focus on the regularity required from $u_0$ in order to generate a unique strong solution. In Section \ref{sec_local_exist_uni_strong} we have shown that $u_0 \in H \cap L^\infty$ generates such a solution that is strong on $[0,T)$ (see Definition \ref{def_reg_sol_halfclosed}) for some $T>0$ (see Theorem \ref{local_exist_strong_thm}). We also observed that the high regularity of $u_0$ guarantees some further properties of such solutions; in particular the representation formula \eqref{repr_half_closed}.

It turns out that relaxing the regularity of $u_0$ still gives a unique strong solution for (sufficiently small) positive times. This motivates the following definition.\footnote{Definition \ref{def_semi_reg} and the uniqueness result in Lemma \ref{EE_semi_reg_lemma} are stated in Section 23 in Leray's paper.}
\begin{definition}\label{def_semi_reg} A function $u$ is a \emph{semi-strong solution of the Navier--Stokes equations} \eqref{eqNSE}, \eqref{eqNSE_incomp} on the time interval $[0,T)$ if it is a strong solution on the open time interval $(0,T)$ (see Definition \ref{def_strong_open}) such that $u\in C([0,T);L^2)$ and
\eqnb\label{integral_condition_u_infty}
\int_0^t \| u(s) \|_\infty^2 \,\d s < \infty \quad \text{ for all } t\in (0,T).
\eqne
\end{definition}
Note this definition is less restrictive than the definition of strong solutions on the time interval $[0,T)$ (Definition \ref{def_reg_sol_halfclosed}). Namely, we replace the weak form of the equations \eqref{eqNSE_distr_incl_zero} by the weak form \eqref{eqNSE_distr}, which does not include the initial data $u_0$, and we replace the boundedness of $\| u(t) \|_\infty$ as $t\to 0^+$ by the integral condition \eqref{integral_condition_u_infty}. Note that the initial condition $u(0)=u_0$ is now incorporated in the assumption $u\in C([0,T);L^2)$.
\begin{lemma}\label{EE_semi_reg_lemma}
Semi-strong solutions to the Navier--Stokes equations on $[0,T)$ (that satisfy a given initial condition) are unique and satisfy the energy equality
\[
\| u(t) \|^2 + 2\int_0^t \| \nabla u (s) \|^2 \,\d s = \| u(0 ) \|^2
\]
for all $t\in [0,T)$.
\end{lemma}
\begin{proof} As in Lemma \ref{uniqueness_and_EE_strong_sols_lem}, the claim follows by taking the limit $t_1 \to 0^+$ in \eqref{EE_strong_open_interval} and \eqref{comparison_two_sols_inequality} and applying Monotone Convergence Theorem (note that the integral condition \eqref{integral_condition_u_infty} guarantees that the exponent in \eqref{comparison_two_sols_inequality} remains finite as $t_1\to 0^+$).
\end{proof}
We now use the notion of semi-strong solutions to obtain a local-in-time well-posedness for initial data $u_0 \in V$ (rather than $u_0\in H\cap L^\infty$ as in Theorem \ref{local_exist_strong_thm}).\footnote{This is Section 24 in \cite{Leray_1934}.}
\begin{theorem}\label{semi_reg_sols_existence_thm}
If $u_0 \in V$ (that is $u_0 \in H^1$ is divergence free) then there exists a unique semi-strong solution $u$ on time interval $[0,T)$, where $T\geq C / \| \nabla u_0 \|^4$, such that $u(0)=u_0$.
\end{theorem}
\begin{proof}
Note that $J_\varepsilon u_0 \in  H  \cap L^\infty $ and $\mathrm{div} (J_\varepsilon u_0 ) =0$ (see Section \ref{sec_prelims}). Hence Theorem \ref{local_exist_strong_thm} implies the existence of a unique strong solution $u_\varepsilon (t)$ to the Navier--Stokes equations on some time interval $[0,T_\varepsilon )$ such that $u_\varepsilon (0)= J_\varepsilon u_0$. The energy equality \eqref{EE_half_closed} and the properties of mollification (see Lemma \ref{prop_molli}) give
\begin{equation}\label{4.1}
\| u_\varepsilon (t) \| \leq \|u_\varepsilon (0) \| = \| J_\varepsilon u_0 \| \leq \| u_0 \|,
\end{equation}
\begin{equation}\label{4.1a}
\| \nabla u_\varepsilon (0) \| = \| \nabla ( J_\varepsilon u_0 ) \| = \| J_\varepsilon (\nabla u_0 ) \| \leq \| \nabla u_0 \|.
\end{equation}
The last bound and Corollary \ref{existence_time_corollary} let us bound the existence time $T_\varepsilon$ from below independently of $\varepsilon$, 
\[T_\varepsilon \geq  C/\| \nabla u_0 \|^4  =: T.\] Moreover, Lemma \ref{bounds_on_infty_norm_lemma} (ii) gives
\begin{equation}\label{4.2}
\| u_\varepsilon (t) \|_{\infty} \leq C \| \nabla u_\varepsilon (0) \| t^{-1/4} \leq C \| \nabla u_0 \| t^{-1/4} 
\end{equation}
for $t\in [0,T)$. By \eqref{4.1} and \eqref{4.2} we may apply the convergence lemma (Lemma \ref{lem9_from_Leray}) to extract a sequence $\{ \varepsilon_k \}$ such that $u_{\varepsilon_k } \to u$ almost everywhere, where $u$ is a strong solution of the Navier--Stokes equations on $\RR^3 \times (0,T)$ with
\begin{equation}\label{star1}
\|u(t) \| \leq \| u_0 \|
\end{equation}
and
\begin{equation}\label{starstar11}
\|u(t) \|_{\infty } \leq C \|\nabla  u_0 \| t^{-1/4}
\end{equation}
for $t\in (0,T)$. It follows from the last inequality that $\int_0^t \| u(s) \|_\infty^2 \,\d s $ is finite for all $t\in (0,T)$. It remains to verify that $u(t) \to u_0 $ in $L^2$ as $t\to 0$. Since $u_{\varepsilon_k} $ is a strong solution to the Navier--Stokes equations on $[0,T )$ \eqref{eqNSE_distr_incl_zero_more_flex} gives
\[
0=\int u_{\varepsilon_k} (t) \cdot \phi - \int (J_{\varepsilon_k} u_0) \cdot \phi - \int_0^t \int u_{\varepsilon_k} \cdot \Delta \phi - \int_0^t \int u_{\varepsilon_k} \cdot ( u_{\varepsilon_k} \cdot  \nabla )\phi 
\]
for $t\in [0,T)$ and $\phi \in C_0^\infty $. By the fact that $J_\varepsilon u_0 \to u_0$ in $L^2$ as $\varepsilon \to 0$, \eqref{4.1} and the Dominated Convergence Theorem (applied to the time integrals) we can pass to the limit in the above equation to obtain
\[
0=\int u (t) \cdot \phi - \int u_0 \cdot \phi - \int_0^t \int u \cdot \Delta \phi - \int_0^t \int u \cdot ( u \cdot  \nabla )\phi ,
\]
which gives that
\[
\int u(t) \cdot \phi \to \int u_0 \cdot \phi\qquad \text{ as } t \to 0^+
\]
for all $\phi \in C^\infty_0 $, which, by $L^2$ boundedness \eqref{star1}, gives that $u(t) \rightharpoonup u_0$ weakly in $L^2  $ as $t\to 0^+$. In order to show that $u(t) \to u_0$ strongly in $L^2$ it is enough to show the convergence of the norms, $\| u(t) \| \to \| u_0 \|$ as $t\to 0^+$. This last claim follows from properties of weak limits and \eqref{star1} by writing
\[
\| u_0 \| \leq \liminf_{t\to 0} \| u(t) \| \leq \limsup_{t\to 0 } \| u(t) \| \leq \| u_0 \|. \qedhere
\]
\end{proof}
Similarly, we obtain that the notion of semi-strong solutions gives local-in-time well-posedness for $u_0\in H\cap L^p$, where $p>3$.\footnote{This is Section 25 of \cite{Leray_1934}. Note that the case $u_0\in H\cap L^\infty$, for which Leray states well-posedness of semi-strong solutions, was covered in this article in the well-posedness result for strong solutions, see Theorem \ref{local_exist_strong_thm}.}
\begin{corollary}\label{exist_semireg_Linfty_or_Lp}
Given $u_0 \in H \cap L^p $ with $p\in (3,\infty )$ there exists a semi-strong solution $u$ of the Navier--Stokes equations on $[0,T)$ with $u(0)=u_0$, where $T\geq  \left( C \left( 1-{3}/{p} \right)/ \| u_0 \|_p\right)^{2p/(p-3)}$.
\end{corollary}
\begin{proof}
Copy the proof above making the following replacements. Replace \eqref{4.1a} by $\| u_\varepsilon (0) \|_p \leq \| u_0 \|_p$, \eqref{4.2} by $\| u_\varepsilon (t) \|_\infty \leq C\| u_0 \|_p t^{-3/2p}$, \eqref{starstar11} by $\| u (t) \|_\infty \leq C\| u_0 \|_p t^{-3/2p}$, and $T$ by $(C(1-3/p)/\| u_0 \|_p )^{2p/(p-3)}$.
\end{proof}
\subsection*{Notes}
This section corresponds to Chapters III and IV of \cite{Leray_1934}. 

In Section 20 \cite{Leray_1934} considers the issue of the existence of solutions that blow up. He points out that such a solution exists if 
\[
\begin{cases}
\Delta U(x) - \alpha U(x) - \alpha ( x\cdot \nabla ) U(x) - \nabla P(x)  = (U(x) \cdot \nabla )U(x),\\
\mathrm{div}\, U(x) =0, 
\end{cases}
\]
has a nontrivial solution in $\RR^3$ for some $\alpha >0$. In that case there would exist a strong solution of the Navier--Stokes equations on the time interval $(-\infty , T)$ of the following self-similar form
\[
u(x,t) \coloneqq \frac{1}{\sqrt{2\alpha (T-t)}} U \left( \frac{x}{\sqrt{2\alpha (T-t)}} \right),
\] 
which would blow up at time $T$. However, Ne\v{c}as, R\r{u}\v{z}i\v{c}ka \& \v{S}ver\'{a}k (1996)\nocite{necas_ruzicka_sverak} have shown that this system of equations has no nontrivial $L^3$ solutions. 

In fact, it is rather remarkable that the issue of the existence of solutions that blow up is one of the most important open problems in mathematics to this day, one of seven Millennium Problems (see \cite{Fefferman_Clay}).

Leray showed the smoothness of strong solutions via similar bounds as in the analysis of the Stokes equations (pp. 218-219). Since in our presentation the properties of the representation formulae \eqref{repr_formula}, \eqref{repr_formula_p} and \eqref{repr_formula1}, \eqref{repr_formula1_p} are organised in Lemmas \ref{prop_of_up_lemma} and \ref{additional_prop_up_lemma}, we were able to prove the smoothness of strong solutions by induction. 

Furthermore, Section \ref{char_of_sing_section} shows that the analysis of the maximal time of existence and the blow-up rates can be done without the use of Leray's {\it(3.6)},
\[
\| \nabla u(t) \| \leq C \int_0^t \frac{\| \nabla u (s) \| \, \| u (s) \|_\infty }{\sqrt{t-s}} \d s + \| \nabla u_0 \|,
\]
which we use only in the proof of Corollary \ref{2nd_char_cor}.

\section{Weak solutions of the Navier--Stokes equations}\label{sec_weak}
In this section we show the global existence of a weak solution (which Leray termed a \emph{turbulent solution}) of the Navier-Stokes equations. His approach is characterised by considering the following modified system, for $\varepsilon>0$: 
\begin{equation}\label{regularised_NSE}
\begin{split}
&\p_t u - \Delta u + ((J_\varepsilon u) \cdot \nabla ) u + \nabla p = 0, \\
& \mathrm{div}\, u(t) =0,
\end{split}
\end{equation}
which is often called the \emph{Leray regularisation}. We will see that this regularisation of the nonlinear term gives for each $\varepsilon>0$ a unique, global in time, strong solution. We then study the limit $\varepsilon \to 0$ of the solutions of the above equations. 
\subsection{Well-posedness for the regularised equations}
%Here we focus solely on the regularised equations \eqref{regularised_NSE}. 
\begin{definition}\label{def_of_sol_regularised_eqs}
A function $u_\varepsilon$ is a \emph{strong solution of the regularised equations \eqref{regularised_NSE} on the interval $[0,T)$} if for some $p_\varepsilon \in L_{\loc}^1 (\RR^3 \times [0,T))$ 
\begin{equation}\label{regularised_eqs_distr_form}
\int u_\varepsilon(0) \cdot f(0)  +\int_0^T  \int (u_\varepsilon \cdot (f_t + \Lap f) + p_\varepsilon \, \mathrm{div}\, f) = \int_0^T \int (J_\varepsilon u_\varepsilon ) \cdot (u_\varepsilon \cdot \nabla) f
\end{equation}
for all $f\in C_0^\infty (\RR^3 \times [0,T);\RR^3)$, $u_\varepsilon (t)$ is weakly divergence free for $t\in (0,T)$, and 
\[ u_\varepsilon \in C([0,T); L^2)\cap C((0,T);L^\infty )\]
with $\| u_\varepsilon (t) \|_\infty$ bounded as $t\to 0^+$.
\end{definition}
Note this definition follows the lines of the definition of a strong solution of the Navier--Stokes equations on time interval $[0,T)$ (Definition \ref{def_reg_sol_halfclosed}), the difference appearing only in the form of the distributional equations \eqref{regularised_eqs_distr_form}. Moreover, we see that a solution $u_\varepsilon$ and the corresponding pressure $p_\varepsilon$ are given by
\begin{equation}\label{sol_regularisedNSE_form}
\begin{split}
u_\varepsilon (t)&=  \Phi (t) \ast u_\varepsilon (t_1 )+\int_{t_1}^{t} \int \nabla T (t-s) \ast \left[ (J_\varepsilon u_\varepsilon ) (s)  \, u_\varepsilon (s) \right]   \,\d s ,\\
p_\varepsilon(t)&=  \p_i \p_k (-\Delta )^{-1} \left( ( J_\varepsilon u_{\varepsilon , i} )(t) \, u_{\varepsilon , k} (t) \right)   
\end{split}
\end{equation}
for all $0\leq t_1 < t <T$, cf. \eqref{sol_NS_form} and \eqref{repr_half_closed}, see also Corollary \ref{cor_can_have_yz}. 

Now let $u_0 \in H^1  \cap L^\infty $ be divergence free; we will show existence and uniqueness of global-in-time strong solution of the regularised equations with initial data $u_0$.
\begin{theorem}[Global well-posedness of the regularised equations\footnote{This is Section 26 of \cite{Leray_1934}.}]\label{thm5.1}
For each $\varepsilon >0$ there exists a unique strong solution $u_\varepsilon $ of the regularised equations \eqref{regularised_NSE} on the time interval $[0,\infty )$ such that $u_\varepsilon (0)=u_0$, $u_\varepsilon $ is smooth on the time interval $(0,\infty )$ (in the sense of Corollary \ref{smoothness_reg_sol_corollary}) and the energy equality 
\begin{equation}\label{EE_regularised_eqs}
\| u_\varepsilon (t) \|^2 + 2 \int_0^t \| \nabla u_\varepsilon (s) \|^2 \,\d s = \| u_0 \|^2
\end{equation}
holds for all $t\geq 0$.
\end{theorem}
\begin{proof}
We note that the analysis from Section \ref{sec_strong} can be applied to the regularised equations \eqref{regularised_NSE}. In particular, by noting that $\| J_\varepsilon v \| \leq \| v \|$ and $\| J_\varepsilon v \|_\infty \leq \| v \|_\infty$ for any $v\in  L^2\cap L^\infty $ (see Lemma \ref{prop_molli}, (i)), we can prove a local existence and uniqueness theorem following Theorem \ref{local_exist_strong_thm}, to obtain a unique strong solution $u_\varepsilon$ of the system \eqref{regularised_NSE} on the time interval $[0,T)$ for some $T \geq C / \| u_0 \|^2_\infty$. Now following the arguments in Section \ref{smoothness_section} we note that (using the representation formulae \eqref{sol_regularisedNSE_form} instead of \eqref{sol_NS_form}) $u_\varepsilon$ is smooth on the time interval $(0,T)$, and so following Theorem \ref{EE_open_int_thm} and Lemma \ref{uniqueness_and_EE_strong_sols_lem} we obtain the energy equality 
\[
\| u_\varepsilon (t) \|^2 + 2 \int_0^t \| \nabla u_\varepsilon (s) \|^2 \,\d s = \| u_0 \|^2
\]
for $t\in [0,T)$. % (note that, similarly as in \eqref{temp_cancellation_EE}, $\int u_\varepsilon \cdot ((J_\varepsilon u_\varepsilon ) \cdot \nabla )u_\varepsilon=0$, that is the nonlinear term gives no contribution to the energy).
 It remains to show that $T$, the maximal time of existence, is infinite. 

As in Section \ref{char_of_sing_section} we see that $\| u_\varepsilon (t) \|_\infty $ must blow-up as $t\to T^-$ if $T<\infty$. We also obtain
\[
\| u(t) \|_\infty \leq C \int_0^t \frac{ \| (J_\varepsilon u_\varepsilon ) (s) \|_\infty  \| u(s) \|_\infty }{\sqrt{t-s}}\,\d s + \| u_0 \|_\infty 
\]
for $t\in [0,T)$, in a similar way to the derivation of \eqref{char_irreg_fund_rel_infty}. This inequality, however, is fundamentally different from \eqref{char_irreg_fund_rel_infty} in the sense that we can now apply the bound 
\[
\| (J_\varepsilon u_\varepsilon ) (s) \|_\infty \leq C \varepsilon^{3/2} \| u_\varepsilon (s) \|  \leq C \varepsilon^{3/2} \| u_0 \|, 
\]
(see Lemma \ref{prop_molli} (iii) and the energy equality \eqref{EE_regularised_eqs}). Moving this bound outside of the integral, we obtain a linear integral inequality,
\[
\|u_\varepsilon (t) \|_{\infty} \leq  C' \varepsilon^{-3/2} \|u_0\| \int_0^t \frac{ \|u_\varepsilon (s)\|_{\infty }}{\sqrt{t-s}} \,\d s + \| u_0 \|_\infty.
\]
Letting $\phi_\varepsilon \in C([0,\infty ))$ be the unique solution to the corresponding linear integral equation\footnote{\re{eqVolterra} is an example of the Volterra equation; see Appendix \ref{volterra_eq_section} for the proof of existence and uniqueness of solutions.},
\begin{equation}\label{eqVolterra}
\phi_\varepsilon (t) = C' \varepsilon^{-3/2} \|u_0 \| \int_0^t \frac{\phi_\varepsilon (s)}{\sqrt{t-s}} \,\d s + \| u_0 \|_{\infty },\quad t\geq 0,
\end{equation}
we see that $\| u_\varepsilon (t) \|_\infty \leq \phi_\varepsilon (t)$ for all $t\geq 0$ (see Corollary \ref{integral_ineq_power1_cor}). Therefore $\|  u_\varepsilon (t) \|_\infty $ remains bounded on every finite interval and hence $T=\infty$, as required.
\end{proof}
In order to study the limit as $\varepsilon \to 0$ of the solutions $u_\varepsilon$ to the system \eqref{regularised_NSE}, we first show that the kinetic energy of $u_\varepsilon (t) $ outside of a ball can be bounded independently of $\varepsilon$.
\begin{lemma}[Separation of energy\footnote{This is Section 27 of \cite{Leray_1934}.}]\label{tail_estimate}
Let $\varepsilon >0$, $0<R_1 < R_2$ and let $u_\varepsilon$ be the solution of \eqref{regularised_NSE} with initial condition $u_0$. Then for $t\geq 0$
\[
 \int_{|x|>R_2} |u_\varepsilon (t)|^2 \,\d x \leq  \int_{|x|>R_1} |u_0 |^2 \,\d x + \frac{C(u_0,t)}{R_2-R_1},
\]
where $C(u_0,t):=C \|u_0\|^2 \sqrt{t} + C \|u_0 \|^3 t^{1/4}$. 
\end{lemma}
\begin{proof}
For brevity we will write $u$ in place of $u_\varepsilon$. Let 
\[
f(x) := \begin{cases}
0\qquad & |x| <R_1,\\
\frac{|x| - R_1 }{R_2-R_1} & R_1 \leq |x| \leq R_2,\\
1 & |x|> R_2.
\end{cases}
\]
Taking the scalar product of the regularised equations \eqref{regularised_NSE} against $-2f(x)u_i(x,t) $, integrating in space and time and using $\mathrm{div}\, u=0$ yields
\begin{multline}
2\int_0^t \int f |\nabla u |^2 + \int f |u|^2 \\=  \int f |u_0|^2 - \int_0^t \int \left( 2\p_k f\, u_i\, \p_k u_i +2 \p_i f \, p\, u_i +  \p_k f  (J_\varepsilon u_k) |u|^2 \right)
\end{multline}
Bounding below the second term on the left-hand side by $ \int_{|x|>R_2} |u(t)|^2$ and using the nonnegativity of the first term yields
\begin{align*}
\MoveEqLeft[0]
 \int_{|x|>R_2} |u(t)|^2 \\
 &\leq   \int f |u_0|^2 - \int_0^t \int \left( 2\p_k f u_i \p_k u_i  +  2\p_i f \, p\, u_i +  \p_k f  (J_\varepsilon u_k) |u|^2 \right) \\
 &\leq   \int_{|x|>R_1} |u_0|^2 + \frac{1}{R_2-R_1} \int_0^t \left(2 \|u \| \| \nabla u\| +2 \| u  \| \| p \| +  \| J_\varepsilon u \| \|u\|_{4}^{2} \right)&\\
&&\mathllap{\leq \int_{|x|>R_1} |u_0|^2 +\frac{\|u_0\|}{R_2-R_1}\left(  2\int_0^t  \| \nabla u\| +2\int_0^t \| p \| +  \int_0^t \|u\|_{4}^{2} \right).}
\end{align*}
Let us denote the last three integrals on the right hand side by $I_1$, $I_2$ and $I_3$ respectively. We have
\begin{equation}\label{i1}
I_1 \leq \sqrt{t} \left( \int_0^t \| \nabla u (s) \|^2 \,\d s \right)^{1/2} \leq \frac{\sqrt{t}}{2} \|u_0 \|
\end{equation} 
by the energy equality \eqref{EE_regularised_eqs}. As for $I_2$, $I_3$ note that since $|u|^2$ solves (trivially) the Poisson equation $\Delta |u|^2 = \Delta |u|^2$ we can integrate by parts to obtain
\[
|u|^2 = \frac{-1}{4 \pi} \int \frac{1}{|x-y|} \Delta |u (y)|^2 \,\d y = \frac{1}{4\pi } \int \frac{x - y}{|x-y|^3}\cdot  \nabla |u (y) |^2 \,\d y .
\]
Thus
\begin{align*}
\MoveEqLeft
\|u \|_{4}^4 = \frac{1}{4 \pi } \int \int | u(x)|^2  \frac{x - y}{|x-y|^3} \cdot  \nabla |u(y)|^2  \,\d x\, \d y\\
&\leq \frac{1}{2 \pi } \int \int   \frac{| u(x)|^2}{|x-y|^2}\, |u(y) | \,|\nabla u (y)|  \,\d x\, \d y\\
&\leq C \| \nabla u \|^2 \int |u(y)| \, |\nabla u (y) | \,\d y \leq C \| \nabla u \|^3 \| u \|,
\end{align*}
where we used Lemma \ref{temp_singular_int_bound} and the Cauchy-Schwarz inequality. Note that this estimate is independent of $t$ (which we omitted in the notation). Moreover, the representation formula \eqref{sol_regularisedNSE_form} for $p$ together with the Plancherel Lemma (Lemma \ref{CZ_lemma}) and the bound $\| J_\varepsilon u \|_4 \leq \| u \|_4 $ (see Lemma \ref{prop_molli} (i)) give
\begin{equation}\label{bound_on_p_magic}
\| p \| \leq C \|  | J_\varepsilon u | \, | u | \|\leq C  \| u \|_4^2 \leq C \| \nabla u \|^{3/2} \| u \|^{1/2} .
\end{equation}
Therefore $\| p \|$ and $\| u \|_4^2 $ enjoy the same bound $C \| \nabla u \|^{3/2} \| u \|^{1/2}$. Thus, by the energy equality \eqref{EE_regularised_eqs} and H\"older's inequality, we obtain
\begin{align*}
\MoveEqLeft
I_2,I_3 &\leq C \int_0^t  \| \nabla u (s) \|^{3/2} \|  u (s)  \|^{1/2} \,\d s \leq C \|  u_0  \|^{1/2} \int_0^t \| \nabla u (s) \|^{3/2} \,\d s \\
&&\mathllap{\leq C \|  u_0  \|^{1/2} t^{1/4} \left( \int_0^t \| \nabla u (s) \|^2 \,\d s \right)^{3/4}  \leq C \| u_0 \|^2 t^{1/4}.}
\end{align*}
Hence, finally
\[
\begin{split}
 \int_{|x|>R_2} |u(t)|^2 &\leq   \int_{|x|>R_1} |u_0|^2 +\frac{\|u_0\|}{R_2-R_1}\left( 2I_1 + 2I_2 + I_3 \right) \\
&\leq   \int_{|x|>R_1} |u_0|^2 +\frac{C}{R_2-R_1}\left(  t^{1/2} \|u_0\|^2 + t^{1/4} \|u_0\|^3 \right). \qedhere
\end{split} 
\]
\end{proof}
\begin{remark}
It is interesting to note that \cite{Leray_1934} presented a way of deriving the bound \eqref{bound_on_p_magic} that does not use the Plancherel Lemma (in fact Leray does not mention the use of Fourier transform). See Appendix \ref{another_proof_of_magic_bound_sec} for details.
\end{remark}
\subsection{Global existence of a weak solution}
Here we study the limit as $\varepsilon \to 0$ of the solutions $u_\varepsilon$ of the regularised equations \eqref{regularised_NSE} to obtain a global-in-time weak solution of the Navier--Stokes equations, as defined below.
\begin{definition}\label{def_weak_sol}\footnote{This is the definition in Section 31 of \cite{Leray_1934}.}
A function $u$ is a \emph{weak solution of the Navier--Stokes equations} if there exists a set $S\subset (0,\infty)$ of measure zero such that $u(t)$ is weakly divergence free for all $t\in (0,\infty ) \setminus S$, 
\begin{equation}\label{distr_form_weak_sol}
\int u (t) \cdot  f(t) = \int u_0 \cdot f(0) + \int_0^t \int u \cdot \left( \p_t f +  \Delta f   \right)+\int_0^t \int  u\cdot (u\cdot \nabla ) f
\end{equation}
for all $t>0$ and all divergence-free test functions $f$ such that $\p_t^m \nabla^k f \in C([0,\infty); L^2) \cap C([0,\infty ); L^\infty )$ for all $k,m\geq 0$, and
\begin{equation}\label{EI_weak_sol}
\| u(t) \|^2 + 2 \int_s^t \| \nabla u (r) \|^2 \,\d r \leq \| u (s)\|^2 
\end{equation}
for every $s\in [0,\infty ) \setminus S$ and every $t\geq s$. 
\end{definition}
Such a solution is often called a \emph{Leray-Hopf weak solution} (Eberhard \cite{Hopf_1951} considered weak solutions in a similar sense on a bounded domain), while \emph{weak solutions} are functions (belonging to $L^\infty(0,T;H)\cap L^2(0,T;V)$) satisfying the first part but not necessarily the energy inequality, see, for example, Definitions 4.9 and 3.3 in \cite{JCR_R_S_NSE_book} (see also Lemma 6.6 therein for the equivalence of the spaces of test functions). The set $S$ is often called the \emph{set of singular times}.
\begin{corollary}\label{continuity_in_t_of_weak_sol_cor}
A weak solution $u$ satisfies
\[
u\in L^\infty ((0,\infty );L^2) \cap L^2 ((0,\infty ); H^1).
\]
Moreover, it is $L^2$ weakly continuous in time and for $s\in [0,\infty ) \setminus S$, $u(t)\to u(s)$ in $L^2$ as $t\to s^+$.
\end{corollary}
\begin{proof}
The first property is a consequence of the energy inequality \eqref{EI_weak_sol}, the $L^2$ weakly continuity is a consequence of \eqref{distr_form_weak_sol}, and the last property is a consequence of the weak continuity and the convergence of the norms $\| u(t) \| \to \| u(s) \|$ for $s\in [0,\infty ) \setminus S$ (which follows from \eqref{EI_weak_sol}).
\end{proof}
\begin{theorem}[Global existence of a weak solution\footnote{This theorem corresponds to Sections 28-31 of \cite{Leray_1934}.}]\label{existence_weak}
If $u_0\in H$ then there exists a weak solution $u$ of the Navier--Stokes equations such that $\| u(t) - u_0 \|\to 0$ as $t\to 0^+$.
\end{theorem}
\begin{proof} For each $\varepsilon >0$ let $u_\varepsilon$ be the unique strong solution of \eqref{regularised_NSE} with $u_\varepsilon (0) = J_\varepsilon u_0$. The existence of such $u_\varepsilon$ is guaranteed by Theorem \ref{thm5.1}, which also gives
\begin{equation}\label{2)}
 \| u_\varepsilon (t) \|^2 + 2\int_0^t \| \nabla u_\varepsilon (s) \|^2 \,\d s =  \|J_\varepsilon u_0 \|^2 \leq  \| u_0 \|^2 
 \end{equation}
 for $t\geq 0$, $\varepsilon >0$ (see \eqref{EE_regularised_eqs}; the last inequality is a property of the mollification operator, see Lemma \ref{prop_molli} (i)). The weak solution is constructed in the following four steps.
 
 \vspace{10pt}
\noindent\emph{Step 1. Construct the sequence $\{ \varepsilon_n \}$.}

% Multiplying the equation in \eqref{regularised_NSE},
 %\[
 %\p_t u_\varepsilon - \Delta u_\varepsilon + ((J_\varepsilon u_\varepsilon) \cdot \nabla ) u_\varepsilon + \nabla p_\varepsilon = 0,
% \]
Multiplying the equation in \eqref{regularised_NSE} by $f$ and integrating by parts we obtain, for $t\geq 0$,
 \begin{equation}\label{1)}
 \int u_\varepsilon (t) \cdot  f(t) = \int J_{\varepsilon }u_0 \cdot  f(0) + \int_0^t \int u_\varepsilon \cdot \left( \p_t f + \Delta f  \right)+\int_0^t \int u_\varepsilon \cdot ( J_\varepsilon u_\varepsilon \cdot \nabla ) f  .
 \end{equation}
Inequality \eqref{2)} implies that for each $t$ the numbers $\|u_{\varepsilon} (t) \| $ are bounded independently of $\varepsilon$, and so using a diagonal argument we extract a subsequence $\{ \varepsilon_n \} $ such that for all $t\in \QQ^+$, $\|u_{\varepsilon_n} (t)\| \to W(t)$ as $k\to \infty$, for some function $W \colon \QQ^+ \to [0,\infty )$. Extend $W$ to the whole of $\RR_+$ by letting $W(t) \coloneqq \liminf_{s\to t^-} W(s)$ for $t\in \RR_+ \setminus \QQ^+$. Using the following fact, we see that $\|u_{\varepsilon_n} (t)\| \to W(t)$ as $n\to \infty$ for times $t$ at which $W$ is continuous.
 \begin{fact}[Helly's theorem\footnote{This result is due to \cite{helly}, see also Lemma 13.15 in \cite{carothers}.}]
 If $g_n \in C([0,1])$ is nonincreasing for each $n$ and $g_n(t) \to g(t)$ for all $t\in \QQ\cap [0,1]$ then $g_n (t) \to g(t)$ at each continuity point $t$ of $g$. 
 \end{fact}
Since $\|u_{\varepsilon_n} (t)\|$ is a nonincreasing function of $t$ for each $n$ (see the energy equality \eqref{EE_regularised_eqs}), the same is true of the limit function $W(t)$. Therefore, since any non-increasing non-negative function has at most countably many points of discontinuity, we can apply the diagonal argument to account for such points and obtain
\begin{equation}\label{norms_converge_temp}
\|u_{\varepsilon_n} (t)\| \to W(t) \qquad \text{ for } t\geq 0,
\end{equation}
where we have also relabelled the sequence $\{\varepsilon_n \}$ and redefined $W$ at its points of discontinuity. Note that since 
\begin{equation}\label{weak_thm_temp1}
u_{\varepsilon_n} (0) =J_{\varepsilon_n } u_0 \to u_0\quad \text{ in } L^2
\end{equation}
(see Lemma \ref{prop_molli} (vi)) we obtain
\[
W(0) = \lim_{n\to \infty }  \|u_{\varepsilon_n} (0)\| = \lim_{ n \to \infty }  \|J_{\varepsilon_n } u_0 \| = \| u_0 \| .
\]
We now want to take the limit of the functions themselves (rather than the norms). Since 
\begin{equation}\label{weak_sol_temp2}
\| J_{\varepsilon_n } u_{\varepsilon_n } (t) \|\leq \| u_{\varepsilon_n } (t) \| \leq \| u_0 \|, \qquad \varepsilon_n>0,t \geq 0
\end{equation}
(see Lemma \ref{prop_molli} (i) and \eqref{2)}) we can apply the diagonal argument once more (and relabel the sequence $\{ \varepsilon_n \}$) to deduce that the numbers
\eqnb\label{numbers_converge}
\int_{t_1}^{t_2} \int_\Omega \left( u_{\varepsilon_n} \right)_k  \quad \text{and}\quad\int_{t_1}^{t_2} \int_\Omega \left( u_{\varepsilon_n } \right)_k \left( J_{\varepsilon_n} u_{\varepsilon_n } \right)_l  \quad \text{ converge as }n\to \infty 
\eqne
for all $t_1, t_2 \in \QQ^+$, $k,l=1,2,3$ and all cubes $\Omega\subset\RR^3$ with rational vertices, that is vertices whose coordinates are rational numbers. Moreover, from the timewise uniform continuity of the above integrals (that is from the bound $C(T)|t_2-t_1|$ of the integrals whenever $t_1,t_2 \in (0,T)$ for some $T$) we obtain that in fact they converge as $n \to \infty$ for every pair $t_1,t_2 \geq 0$, every $k,l=1,2,3$ and every cube $\Omega$ with rational vertices. 

From the convergence in \eqref{numbers_converge} we see that
\eqnb\label{number_converge_better}
\int_{t_1}^{t_2} \int u_{\varepsilon_n} \cdot ( \p_t f + \Delta f ) \quad \text{ and} \quad \int_{t_1}^{t_2} \int  u_{\varepsilon_n } \cdot \left(  J_{\varepsilon_n} u_{\varepsilon_n } \cdot \nabla \right) f 
\eqne
converge as $n\to \infty$ for all $t_1,t_2 \geq 0$ and all test functions $f$ (see Definition \ref{def_weak_sol}). Indeed, fix $t_1,t_2 \geq 0$, a test function $f$, and $\epsilon >0$. Without loss of generality we can assume $t_1<t_2$. There exists $N\in \NN$, $\{g_k\}_{k=1}^N \subset \RR^3$, $\{\Omega_k\}_{k=1}^N$ (cubes with rational coordinates), and a family of intervals
\[\{(p_k,q_k)\colon 0\leq p_k<q_k\leq\infty,\ k=1,\ldots, N\}
\] such that the (vector-valued) function
\[
G_N (x,t) \coloneqq \sum_{k=1}^N g_k I_{\Omega_k} (x) I_{(p_k,q_k)}(t),
\]
satisfies
\eqnb\label{choice_of_GN}
\| (\p_t f + \Delta f ) - G_N \|_{L^\infty((t_1,t_2);L^2)} \leq \frac{\epsilon }{4 \| u_0 \|(t_2-t_1)}.
\eqne
Moreover, the first part of \eqref{numbers_converge} gives that for sufficiently large $n,m$
\[
\left| \int_{t_1}^{t_2} \int (u_{\varepsilon_n}-u_{\varepsilon_m}) \cdot G_N \right| \leq \frac{\epsilon }{2}
\]
Thus 
\begin{align*}
\MoveEqLeft[0]
\left| \int_{t_1}^{t_2} \int (u_{\varepsilon_n}-u_{\varepsilon_m}) \cdot ( \p_t f + \Delta f )\right|\\& \leq \left| \int_{t_1}^{t_2} \int (u_{\varepsilon_n}-u_{\varepsilon_m}) \cdot G_N \right|
+\int_{t_1}^{t_2} \int (|u_{\varepsilon_n}|+|u_{\varepsilon_m}|) \left| ( \p_t f + \Delta f )- G_N  \right|\\
&\leq \frac{\epsilon }{2} +  \frac{\epsilon }{4 \| u_0 \|(t_2-t_1)} \int_{t_1}^{t_2} \left( \|u_{\varepsilon_n} (s) \| + \|u_{\varepsilon_m} (s) \|  \right) \d s \\
&&\mathllap{\leq \frac{\epsilon }{2} +  \frac{\epsilon }{2 \| u_0 \|(t_2-t_1)} \int_{t_1}^{t_2} \| u_0 \| \, \d s = \epsilon,\hspace{10pt}}
\end{align*}
where we also used the energy inequality \eqref{2)}. Thus we obtain the first part of \eqref{number_converge_better}. The second part follows in a similar way, by choosing a simple (matrix) function $G_N$ such that 
\[
\| \nabla f - G_N \|_{L^\infty ((t_1,t_2);L^\infty )} \leq \frac{\epsilon }{2 \| u_0 \|^2(t_2-t_1)},
\]
rather than \eqref{choice_of_GN}. Thus we obtain \eqref{number_converge_better}.

On the other hand, the convergence $J_{\varepsilon_n } u_0 \to u_0$ in $L^2$ (see \eqref{weak_thm_temp1}) implies, in particular, the convergence of numbers
\[
\int J_{\varepsilon_n }u_0 \cdot  f(0) \to \int u_0 \cdot  f(0)\qquad \text{ for all test functions }f.
\]
Combining this with \eqref{number_converge_better} we can use the weak form of the regularised equations \eqref{1)} to obtain that the numbers
\[
\int u_{\varepsilon_n} (t)\cdot  f(t) \quad \text{ converge for all test functions }f \text{ and all } t\geq 0.
\] 
Thus letting $f(x,t)\coloneqq \phi (x)$ for some $\phi \in L^2$ such that $\phi \in H^m$ for all $m\geq 1$ and $\mathrm{div}\, \phi =0$, we obtain that the numbers
\eqnb\label{numbers_converge_phi}
\int u_{\varepsilon_n} (t)\cdot  \phi  \quad \text{ converge for all } t\geq 0, \phi .
\eqne
This together with the fact that $\| u_{\varepsilon_n} (t) \| \leq \|u_0\|$ (see \eqref{2)}) implies that\footnote{This functional analytical fact is one of Leray's remarkable contributions, which he discusses on page 209. To see it, note that if \eqref{weak_conv_turb} does not hold then (using the boundedness in $L^2$) one could extract subsequences $\{ n_k \}$, $\{ m_k \}$ such that $u_{\varepsilon_{n_k}} (t) \rightharpoonup v(t) $, $u_{\varepsilon_{m_k}} (t) \rightharpoonup w(t) $ for some $v(t), w(t)\in L^2$, $v(t)\ne w(t)$. In that case \eqref{numbers_converge_phi} gives $\int (u_{\varepsilon_{n_k}} (t) -u_{\varepsilon_{m_k}} (t))\phi \to 0$ for all $\phi$, and taking $\phi\coloneqq J_\delta (v(t)-w(t))$ gives $\int(v(t)-w(t)) J_\delta(v(t)-w(t))=0$, which in the limit $\delta \to 0^+$ gives $\| v(t)-w(t) \|=0$, a contradiction.}
\begin{equation}\label{weak_conv_turb}
u_{\varepsilon_n} (t) \rightharpoonup u(t) \qquad \text{ as }n\to \infty \text{ in }L^2, \quad t\geq 0,
\end{equation}
for some $u(t) \in L^2$ (note $u(0)=u_0$ by \eqref{weak_thm_temp1} and $u(t)$ is weakly divergence free for each $t$). We have thus constructed the sequence $\{ \varepsilon_n \}$ and we obtained $u$ as the weak limit of $u_{\varepsilon_n}$'s. In order to show that $u$ is the required weak solution (Step 4), we first show that in fact $u_{\varepsilon_n} (t) \to u(t)$ strongly in $L^2$ for almost all $t$ (Step 3).

 \vspace{10pt}
\noindent\emph{Step 2. Define the set of singular times $S$.}

Fatou's lemma and \eqref{2)} give
\[
\int_0^\infty \liminf_{\varepsilon_n \to 0} \|\nabla u_{\varepsilon_n} (t) \|^2 \,\d t \leq \liminf_{\varepsilon_n \to 0} \int_0^\infty  \|\nabla u_{\varepsilon_n}(t) \|^2\,\d t  \leq  \frac{1}{2} \|u_0\|^2.
\]
Hence $\liminf_{\varepsilon_n \to 0} \|\nabla u_{\varepsilon_n} (t) \|^2 < \infty$ for almost every $t>0$, that is $|S|=0$, where  
\begin{equation}\label{def_of_S}
S\coloneqq \{  t> 0 \, : \, \| \nabla u_{\varepsilon_n } (t) \| \to \infty  \text{ as } n\to \infty \} 
\end{equation}
is the \emph{set of singular times}.

 \vspace{10pt}
\noindent\emph{Step 3. Show that $u_{\varepsilon_n} (t) \to u(t)$ strongly in $L^2$ for $t\in (0,\infty ) \setminus S$.}

Recalling that $W(t) = \lim_{\varepsilon_n\to 0 } \| u_{\varepsilon_n }(t) \|$ (see \eqref{norms_converge_temp}) we see that it is enough to show that 
\eqnb\label{conv_of_norms_wt}
\| u(t ) \|=W(t ), \quad t\in (0,\infty ) \setminus S,
\eqne
since weak convergence \eqref{weak_conv_turb} together with the convergence of the norms is equivalent to strong convergence. The estimate $\| u(t ) \|\leq W(t )$ holds for all $t\geq 0$ by the property of weak limits,
\eqnb\label{square}
\|u(t)\| \leq \liminf_{\varepsilon_n \to 0} \|u_{\varepsilon_n} (t) \| = W(t).
\eqne
In order to prove the converse inequality, fix $t\in (0,\infty )\setminus S$. For such $t $ let $\{ \varepsilon_{n_k}\} $ be a subsequence along which $\liminf_{\varepsilon_n \to 0} \|\nabla u_{\varepsilon_n} (t ) \|^2$ is attained, that is
\eqnb\label{liminf_is_attained}
\lim_{k\to \infty} \| \nabla u_{\varepsilon_{n_k}} (t  ) \| = \liminf_{\varepsilon_n \to 0} \|\nabla u_{\varepsilon_n} (t ) \| .
\eqne
It follows that the sequence $\{ \nabla u_{\varepsilon_{n_k}} (t ) \}$ is bounded in $L^2$ and therefore 
\begin{equation}\label{weak_conv_gradients_outside_S}
\nabla  u_{\varepsilon_{n_k}} (t ) \rightharpoonup \nabla u (t )  \text{ in }L^2
\end{equation}
on a subsequence (which we relabel back to $\varepsilon_{n_k}$; note on such a subsequence \eqref{liminf_is_attained} still holds). The limit function is $\nabla u(t )$ by the definition of weak derivatives, and 
\begin{equation}\label{ae_fatou_for_gradients}
\| \nabla u (t ) \| \leq \liminf_{\varepsilon_n \to 0} \|\nabla u_{\varepsilon_n} (t ) \| \qquad \mbox{ for }t  \in (0,\infty )\setminus S.
\end{equation}
At this point we need to apply the separation of energy result (Lemma \ref{tail_estimate}). We fix $\eta >0$, and we let $R_1 (\eta )>0 $ be large enough that 
\[ \int_{|x|>R_1(\eta )} |u_0|^2 \leq \frac{\eta }{2}, \]
and set
\[ R_2 (\eta , t) \coloneqq  R_1 (\eta ) + \frac{4}{\eta } C(u_0,t),\]
where $C(u_0,t)$ is the constant from the lemma. The lemma now implies that
 \begin{equation}\label{5.11}
 \int_{|x| > R_2 (\eta , t) } |u_\varepsilon (t) |^2 \leq \eta,  \qquad  \varepsilon >0.
 \end{equation} 
Since $u_{\varepsilon_{n_k}}(t ) \rightharpoonup u(t )$ in $H^1$ (see \eqref{weak_conv_turb} and \eqref{weak_conv_gradients_outside_S}) we have in particular 
\[ u_{\varepsilon_{n_k}}(t ) \rightharpoonup u(t)\quad \text{ in }H^1 (B(R_2(\eta , t ))),
\]
and so the compact embedding\footnote{Leray's Lemma 2 provides a similar compactness result.} $H^1 (B(R_2(\eta , t ))) \subset \subset L^2 (B(R_2(\eta , t )))$ implies
\[
u_{\varepsilon_{n_k}} (t )\to u (t ) \quad \text{ in } L^2 (B(R_2(\eta , t ))).
\]
Therefore, from \eqref{5.11}, we obtain
\begin{align*}
\MoveEqLeft
 \limsup_{{\varepsilon_{n_k}} \to 0} \| u_{\varepsilon_{n_k}} (t ) \|^2 \\
 &\leq  \limsup_{{\varepsilon_{n_k}} \to 0} \int_{|x| \leq  R_2 (\eta , t ) } | u_{\varepsilon_{n_k}} (t ) |^2 +\limsup_{{\varepsilon_{n_k}} \to 0} \int_{|x| > R_2 (\eta , t ) } | u_{\varepsilon_{n_k}} (t ) |^2 \\
&& \mathllap{\leq  \int_{|x| \leq  R_2 (\eta , t ) } | u (t ) |^2  +\eta 
\leq \| u(t  ) \|^2 + \eta \qquad \eta >0,}
\end{align*}
And hence, taking $\eta \to 0$,
\[
W(t ) =\limsup_{{\varepsilon_{n_k}} \to 0} \| u_{\varepsilon_{n_k}} (t ) \| \leq \| u(t  ) \|.
\]
Thus we obtained the $\geq $ inequality in \eqref{conv_of_norms_wt}, as required.

 \vspace{10pt}
\noindent\emph{Step 4. Verify that $u$ is a weak solution.}

We already established (after \eqref{weak_conv_turb}) that $u(t)$ is weakly divergence free. Step 3 gives
\[\begin{split}
\int u_{\varepsilon_n } (t) \cdot \left( \p_t f(t) + \Delta f (t) \right) &\to \int u (t) \cdot \left( \p_t f(t) + \Delta f (t) \right),\\
 \int u_{\varepsilon_n }(t) \cdot ( J_{\varepsilon_n } u_{\varepsilon_n }(t) \cdot \nabla ) f (t) &\to \int u(t) \cdot (  u(t) \cdot \nabla ) f (t) 
 \end{split}
\]
for almost every $t> 0$. Thus, using \eqref{weak_sol_temp2}, we obtain \eqref{distr_form_weak_sol} by taking the limit $\varepsilon_n\to 0^+$ in \eqref{1)} and applying the Dominated Convergence Theorem to the time integrals. As for the energy inequality \eqref{EI_weak_sol}, let $s \in (0,\infty )\setminus S$ and $t>s$. Since $ \| \nabla u (\tau  ) \| \leq \liminf_{\varepsilon_n \to 0} \|\nabla u_{\varepsilon_n} (\tau ) \|$ for almost every $\tau \geq 0$ (see \eqref{ae_fatou_for_gradients}), we can use \eqref{square}, Fatou's lemma, the identity 
\[
 \| u_\varepsilon (t) \|^2 + 2\int_s^t \| \nabla u_\varepsilon (s) \|^2 \,\d s =  \| u_\varepsilon (s) \|^2, \qquad \varepsilon >0
\]
(see \eqref{2)}) and the fact that $\lim_{\varepsilon_n \to 0} \| u_{\varepsilon_n} (s) \|=\| u(s) \|$ (see \eqref{conv_of_norms_wt}) to obtain
\begin{align*}
\MoveEqLeft
 \| u(t)\|^2 +2\int_s^t \| \nabla u (\tau ) \|^2\,\d \tau\\
 &\leq   \liminf_{\varepsilon_n \to 0 } \| u_{\varepsilon_n} (t)\|^2+2 \int_s^t \liminf_{\varepsilon_n \to 0} \| \nabla u_{\varepsilon_n} (\tau ) \|^2 \,\d \tau\\
 &\leq   \liminf_{\varepsilon_n \to 0} \left(     \| u_{\varepsilon_n} (t)\|^2 +2\int_0^t \| \nabla u_{\varepsilon_n} (\tau ) \|^2 \,\d \tau  \right) \\
&&\mathllap{=  \liminf_{\varepsilon_n \to 0}  \| u_{\varepsilon_n } (s) \|^2= \| u (s) \|^2 . }
\end{align*}
The case $s=0$ follows similarly by using \eqref{weak_thm_temp1} (rather than \eqref{conv_of_norms_wt}) in the last equality (recall also $u(0)=u_0$).
Finally $\| u(t) - u_0 \| \to 0$ as $t\to 0$ follows from the weak convergence $u(t) \rightharpoonup u_0$ in $L^2$ (a consequence of \eqref{distr_form_weak_sol}) and the convergence of norms $\| u(t) \| \to \| u_0 \|$ (a consequence of Fatou's lemma and energy inequality with $s=0$).
\end{proof}
\subsection{Structure of the weak solution}
Let $u_0 \in H$ and consider the weak solution $u$ given by Theorem \ref{existence_weak}. We first show that such solution enjoys the following \emph{weak-strong uniqueness} property\footnote{Leray calls this \emph{Comparison of a regular solution and a turbulent solution}, see pp. 242-244.}.
\begin{lemma}[Weak-strong uniqueness]\label{weak_strong_uniqueness_lem}
If $\|\nabla u(t_0)\| <\infty$ for some $t_0 \geq 0$ then $u=v$ on the time interval $[t_0,T)$ for some $T>t_0$, where $v$ is the semi-strong solution corresponding to the initial data $u(t_0)$ (recall Definition \ref{def_semi_reg}).
\end{lemma}
\begin{proof}
The assumption gives that $u(t_0)\in V$ (that is $u(t_0) \in H^1$ and $u(t_0)$ is divergence free, recall \eqref{def_of_HV}) and so Theorem \ref{semi_reg_sols_existence_thm} gives the existence of a unique semi-strong solution $v(t)$ on $[t_0,T)$, for some $T>t_0$. We will show that\footnote{This is essentially the so-called \emph{weak-strong uniqueness} property (see, for instance, Section 6.3 in \cite{JCR_R_S_NSE_book}).} $u=v$ on time interval $[t_0,T)$. 

From the energy equality for $v$ (see Lemma  \ref{EE_semi_reg_lemma}) and the energy inequality for $u$ (see \eqref{EI_weak_sol}), we obtain for a.e. $t_1\in (t_0,T)$ and every $t\in (t_1,T)$
\[\begin{split}
& \|v(t)\|^2 + 2 \int_{t_1}^t \| \nabla v (s) \|^2 \,\d s =  \|v(t_1)\|^2, \\
&  \|u(t)\|^2 + 2 \int_{t_1}^t \| \nabla u (s) \|^2 \,\d s  \leq  \|u(t_1)\|^2.
\end{split}\]
Adding these together and letting $w:=u-v$ gives
\begin{multline*}
 \|u(t_1)\|^2 +  \| v(t_1) \|^2 \geq  \| w(t) \|^2 + 2\int_{t_1}^t \| \nabla w (s) \|^2\,\d s + 2\int u (t) \cdot v (t)  \\+   4 \int_{t_1}^t \int  \nabla u :\nabla v  ,
\end{multline*}
where $A:B \coloneqq A_{ij}B_{ij}$ denotes the inner product of matrices. Since $v$ is strong on $(t_0,T)$ (see Definition \ref{def_semi_reg}), it satisfies $\p_t v - \Delta v + (v\cdot \nabla ) v + \nabla p =0$ in $(t_0,T)\times \RR^3$ and it can be used as a test function for $u$ on time interval $(t_1,t)$ (see \eqref{distr_form_weak_sol}) to write
\[\begin{split}
\int u (t_1)\cdot  v (t_1) &= \int u (t) \cdot v(t) - \int_{t_1}^t \int u \cdot \left( \p_t v + \Delta v +(u \cdot \nabla ) v   \right)  \\
&= \int u (t) \cdot v(t) - \int_{t_1}^t \int u \cdot \left( 2 \Delta v + (w\cdot \nabla ) v -\nabla p \right)\\
&=  \int u (t) \cdot v(t) + \int_{t_1}^t \int v \cdot \left( w \cdot \nabla \right)w  +2 \int_{t_1}^t \int \nabla u : \nabla v , \end{split}
\]
where in the last step we integrated by parts all three terms under the last integral and used the facts that $u(t)$ and $v(t)$ are weakly divergence free and that $\int v \cdot (w \cdot \nabla ) v=0$ (cf. \eqref{temp_cancellation_EE}). Hence, using the inequality above and the Cauchy-Schwarz inequality we obtain
\begin{align*}
\MoveEqLeft
 \|w(t_1)\|^2 =  \|u(t_1)\|^2 +  \| v(t_1) \|^2 - 2 \int u(t_1) \cdot v(t_1) \\
  &\geq  \| w(t) \|^2 + 2\int_{t_0}^t \| \nabla w (s) \|^2 \,\d s -2  \int_{t_1}^t \int v \cdot \left( w\cdot \nabla \right)w  \\
&\geq  \| w(t) \|^2 + 2 \int_{t_1}^t \| \nabla w(s) \|^2\,\d s  - 2 \int_{t_1}^t \| v(s) \|_{\infty} \| w(s) \|\, \| \nabla w (s) \|\,\d s \\
&&\mathllap{\geq   \| w(t) \|^2 -\frac{1}{2} \int_{t_1}^t \| v(s) \|_{\infty}^2 \| w (s) \|^2 ,\hspace{6pt}}
\end{align*}
where in the last step we used Young's inequality $ab\leq a^2 +b^2/4$. Hence Gronwall's inequality implies
\[
\| w(t) \|^2 \leq \| w(t_1 ) \|^2 \mathrm{exp} \left( \frac{1}{2} \int_{t_1}^t \| v (s) \|_{\infty}^2 \,\d s \right) .
\]
Since $\int_{t_0}^t \| v(s) \|_\infty < \infty $ (see Definition \ref{def_semi_reg}) and since $t_0 \not \in S$ both $u$ and $v$ are continuous in $L^2$ as $t_1 \to  t_0^-$ (see Corollary \ref{continuity_in_t_of_weak_sol_cor}) and we can take the limit $t_1 \to t_0^-$ to obtain
\[
\| w(t) \|^2 \leq \| w(t_0 ) \|^2 \exp \left( \frac{1}{2} \int_{t_1}^t \| v (s) \|_{\infty}^2 \,\d s \right) =0 \qquad t\in (t_0,T).
\]
Thus $u(t) = v(t)$ for $t\in [t_0,T)$, as required.
\end{proof}
 We can use the weak-strong uniqueness to obtain that $u$, the weak solution given by Theorem \ref{existence_weak}, is regular in certain time intervals.
\begin{definition}\label{def_max_interval}
We call an open interval $(a,b)\subset (0,\infty )$ \emph{a maximal interval of regularity} if $u$ is a strong solution on $(a,b)$ (see Definition \ref{def_strong_open}) and $u$ is not a strong solution on any open interval $I$ strictly containing $(a,b)$.
\end{definition}
\begin{theorem}[The structure of the weak solution\footnote{This theorem corresponds to Sections 32 and 33 in \cite{Leray_1934}.}]\label{thmEpochs}
If $u$ is a weak solution given by Theorem \ref{existence_weak} then there exists a family of pairwise disjoint maximal intervals of regularity $(a_i,b_i)\subset (0,\infty )$ of $u$ such that the set
\[
\Sigma \coloneqq (0,\infty ) \setminus \bigcup_{i} (a_i , b_i) 
\]
has measure zero.
\end{theorem}
Clearly $S\subset (0,\infty ) \setminus \bigcup_{i} (a_i , b_i)$ (where $S$ is the set of singular times of $u$, see Definition \ref{def_weak_sol}), since a strong solution is divergence free for all times and satisfies the energy equality \eqref{EE_strong_open_interval}. Therefore this theorems asserts that any $t_0 \in (0,\infty ) \setminus S$ is an initial point of an interval in the interior of which $u$ coincides with a strong solution, and the coincidence continues as long as the strong solution exists. Moreover, the energy inequality \eqref{EI_weak_sol} shows that $\| u(t) \|$ is strictly decreasing on every interval of regularity $(a_i,b_i)$. Note also that the family $\{(a_i,b_i)\}_i$ is at most countable (as a family of pairwise disjoint open intervals on $\RR$).
\begin{proof}
Since for the weak solution $u$ constructed in the last theorem the set $S$ is given by \eqref{def_of_S}, we see that if $t_0\notin S$ then $\| \nabla u (t_0 ) \| <\infty$ and so Lemma \ref{weak_strong_uniqueness_lem} gives that $t_0$ either belongs to a maximal interval of regularity (see Definition \ref{def_max_interval}) or is a left endpoint of one such interval. That is $t_0 \in \bigcup_i [a_i,b_i)$, and so 
\[\Sigma \subset S\cup \bigcup_i \{a_i\}.\] 
 Since there are at most countably many maximal intervals of regularity we obtain $\left| \Sigma \right| =0$, as required. 
\end{proof}
Furthermore, it turns out that one of the intervals of regularity $(a_i,b_i)$ of the weak solution $u$ contains $(C\| u_0 \|^4 , \infty )$ and on this interval $u$ enjoys a certain decay, which we make precise in the following theorem.
\begin{theorem}[Supplementary information\footnote{This is Section 34 in \cite{Leray_1934}.}]\label{thmSupInfo}
Under the assumptions of the last theorem the set of singular times $\Sigma $ is bounded above by $C\|u_0\|^4$, and for $t > C \| u_0 \|^4$
\begin{eqnarray}
\| \nabla u(t) \| &<& C \| u_0 \| t^{-1/2}, \label{additional1} \\
\| u(t) \|_{\infty } &<& C\| u_0 \| t^{-3/4}. \label{additional2}
\end{eqnarray}
Moreover
\begin{equation}\label{thmSupInfo:est3}
\sum_{i:\, b_i<\infty} \sqrt{b_i - a_i} \leq C \| u_0 \|^2 .
\end{equation}
\end{theorem}
\begin{proof}
If the union $\bigcup_i (a_i , b_i) $ has only one element $(0,\infty )$ then the first claim follows trivially. If not then let $i$ be such that $b_i<\infty$ and let $t\in (a_j,b_j)$, where $j$ is such that $b_j \leq  b_i$ (that is $(a_i,b_i)$ does not preceed $(a_j,b_j)$ on the line). Since $u$ becomes singular at $b_j$ Corollary \ref{blowup_rates_corollary} gives 
\begin{equation}\label{suppl_info_temp}
\| \nabla u (t) \| \geq C (b_j-t)^{-1/4} \geq C(b_i-t)^{-1/4}.
\end{equation}
Applying this lower bound in the energy inequality \eqref{EI_weak_sol} gives
\[
\| u_0 \|^2 \geq 2 \int_0^{b_i} \| \nabla u(t) \|^2 \,\d t \geq C\int_0^{b_i} (b_i - t )^{-1/2} \,\d s = C (b_i)^{1/2} .
\]
Thus all finite $b_i$'s are bounded by $C\| u_0 \|^4$ and the first claim follows. As for the last claim of the lemma apply the first inequality of \eqref{suppl_info_temp} to the energy inequality \eqref{EI_weak_sol} to obtain
\begin{multline*}
\| u_0 \|^2 \geq 2 \sum_{j:\, b_j<\infty} \int_{a_j}^{b_j} \| \nabla u(t) \|^2 \,\d t \\ \geq C \sum_{j:\, b_j<\infty} \int_{a_j}^{b_j} (b_j -t)^{-1/2} \,\d t  = C \sum_{j:\, b_j<\infty}  \sqrt{b_j-a_j},
\end{multline*}
as required. It remains to show the decay estimates \eqref{additional1}, \eqref{additional2}. Let $s\in \bigcup_i (a_i,b_i)$ and $t>s$. Since $u$ is strong on an interval containing $s$, we can use Corollary \ref{2nd_char_cor} to write that if $t-s \leq C \| \nabla u (s) \|^{-4}$ then
\begin{eqnarray}
&& \| \nabla u (t) \| \leq C \| \nabla u (s) \| , \label{circle} \\
&& \| u (t) \|_{\infty} \leq C \| \nabla u(s) \| (t-s)^{-1/4}. \label{triangle}
\end{eqnarray}
From the first of these two inequalities we see that either $t-s\geq C \| \nabla u (s) \|^{-4}$ or $\| \nabla u(t) \| \leq C \| \nabla u (s) \|$, and so 
\[
\| \nabla u(s) \| \geq C \min \left( (t-s)^{-1/4} , \| \nabla u (t) \|  \right) .
\]
Using this lower bound in the energy inequality \eqref{EI_weak_sol} we obtain
\[
\begin{split}
 \| u_0 \|^2 &\geq  2 C \int_0^t  \min \left( (t-s)^{-1/2} , \| \nabla u (t) \|^2  \right) \,\d s \\
& \geq  C \int_0^t  \min \left( t^{-1/2} , \| \nabla u (t) \|^2  \right) \,\d s\\
&= C   \min \left( t^{1/2} , t \| \nabla u (t) \|^2  \right) .
\end{split}
\]
Therefore, since the first argument of the minimum tends to $\infty$ as $t\to \infty$, we see that for $t^{1/2} > \| u_0 \|^2/C'$ we must have $\min \left( t^{1/2} , t \| \nabla u (t) \|^2  \right)=  t \| \nabla u (t) \|^2 $ and hence $\| u_0 \|^2 \geq C t \| \nabla u (t) \|^2$, from which \eqref{additional1} follows. 

To obtain \eqref{additional2}, let $s \geq  \| u_0 \|^4/2(C')^2$ and $t\coloneqq 2s$. Then $t^{1/2}\geq \| u_0 \|^2/C'$ and so using the above result gives
\[
t-s=t/2=\frac{1}{2} (t^{-1/4} t^{1/2})^4 \leq C(\| u_0 \|^{-1} t^{1/2})^4 \leq C \| \nabla u (t) \|^{-4} .
\]
Therefore  we may apply \eqref{triangle} to obtain $\| u(t) \|_\infty \leq C \| \nabla u (s) \| t^{-1/4}$. Now \eqref{additional2} follows by an application of \eqref{additional1},
\[
\| u(t) \|_{\infty} \leq C \| \nabla u (s ) \| t^{-1/4} \leq C \| u_0 \| t^{-3/4} . \qedhere
\]
\end{proof}
Finally, we remark on a consequence of \re{thmSupInfo:est3}. 
\begin{corollary}\label{corBoxDim}
The set of singular times $\Sigma $ satisfies
\[d_H(\Sigma ) \leq d_B( \Sigma ) \leq 1/2.\]
\end{corollary}
Here $d_H$ denotes the Hausdorff dimension and $d_B$ denotes the upper box-counting dimension (see Sections 2.1 and 3.2 in \cite{Falconer_Intro} for the definitions). This corollary does not appear in Leray's paper. Here we show it following the proof of a more general result by \cite{Lapidus_Zeta_book} (see Theorem 1.10 therein); see also \cite{Besicovitch_Taylor_1954}. An alternative approach can be found in the proof of Theorem 8.13 of \cite{JCR_R_S_NSE_book}.
\begin{proof}
Note that due to the general inequality $d_H(K) \leq d_B(K)$ for any compact $K\subset \RR$ (see Lemma 3.7 in \cite{Falconer_Intro}) it is enough to show the bound $d_B (\Sigma ) \leq 1/2$. In this context the upper box-counting dimension is equivalent to the Minkowski-Bouligand dimension, which (for subsets of $\RR$) is given by
\begin{equation}\label{eqMinkowskiDim}
d_B (\Sigma ) = 1- \liminf_{\delta \to 0^+} \frac{\log |\Sigma_\delta |}{\log \delta },
\end{equation}
where $\Sigma_\delta $ denotes the $\delta$-neighbourhood of $\Sigma$ (see, for instance, Proposition 2.4 in \cite{Falconer_Intro}). Since $\Sigma $ is bounded there exists a unique index $i_0$ such that $b_{i_0}=\infty$, and so the set $\bigcup_{i\ne i_0} (a_i, b_i)$ is bounded by $a_{i_0}$. Thus
\[
\Sigma = (0,a_{i_0}] \setminus \bigcup_i (a_i ,b_i), 
\]
where now each $b_i$ is finite. Thus we can renumber the intervals $(a_i,b_i)$ such that the length of $(a_i,b_i)$ does not increase with $i$, that is $b_i-a_i \geq b_{i+1}-a_{i+1}$ for all $i$.

For $\delta >0$ let $N_\delta$ be such that
\[
\begin{cases}
b_i-a_i < 2 \delta \qquad & i>N_\delta ,\\
b_i-a_i \geq 2 \delta & i \leq N_\delta .
\end{cases}
\]
Observe that
\[
\Sigma_\delta = \Sigma \cup \left( \bigcup_{i=1}^{N_\delta } (a_i , a_i + \delta ) \cup (b_i - \delta , b_i ) \right) \cup \bigcup_{i>N_\delta } (a_i, b_i ).
\]
Thus, since $|\Sigma |=0$,
\eqnb\label{measure_of_sigma_delta}
|\Sigma_\delta | \leq 2\delta N_\delta + \sum_{i>N_\delta } (b_i-a_i) \leq 2\delta N_\delta + \sqrt{2\delta } \sum_{i>N_\delta }  \sqrt{b_i-a_i} \leq 2\delta N_\delta + C \sqrt{\delta } ,
\eqne
where we used the assumption $\sum_i \sqrt{b_i-a_i} \leq C$ (see the last inequality in Theorem \ref{thmSupInfo}). Now since for each $k$
\[
k \sqrt{b_k-a_k } = \sum_{j=1}^k \sqrt{b_k-a_k} \leq \sum_{j=1}^k \sqrt{b_j-a_j} \leq C,
\]
we see that $b_k - a_k< 2\delta $ for $k > C/ \sqrt{2\delta }$. In other words $N_\delta \leq C / \sqrt{\delta }$. Hence \eqref{measure_of_sigma_delta} gives $|\Sigma_\delta | \leq C\sqrt{\delta }$ and so (since $\log \delta <0$ for small $\delta >0$)
\[
\frac{\log |\Sigma_\delta |}{\log \delta } \geq \frac{\log C + \frac{1}{2}\log \delta }{\log \delta } \to \frac{1}{2}
\]
as $\delta \to 0^+$. Therefore $d_B (\Sigma ) \leq 1/2$, as required.
\end{proof}
It is interesting to note that apart from the bound $d_H(\Sigma ) \leq 1/2$ a stronger property $\mathcal{H}^{1/2} (\Sigma )=0$ holds. This can be shown by an (earlier) argument due to \cite{scheffer_1976}, which is similar to the above and is also presented by \cite{robinson_sema}, see Proposition 8 therein.
%#####################################
\subsection*{Notes}
This section corresponds to Chapters V and VI of \cite{Leray_1934}, except that Leray did not discuss the dimension of the set of singular times. In fact the notion of the box-counting dimension was not properly unified in the 1930's, although some variants were already being studied at the time (see \cite{Bouligand_1928} and \cite{Ponttrjagin_Schnirelman_1932} for instance). On the other hand, the notion of the Hausdorff dimension was already  quite well developed (see \cite{Hausdorff_1918}, \cite{Besicovitch_1935}), but it is not apparent whether or not Leray was aware of these developments. In any case, it was \cite{scheffer_1976} who was the first to study the dimension of the set of singular times for the Navier--Stokes equations.
%#####################################
\section*{Acknowledgements}
\comment{This publication arose from series of lectures presented by the authors for a fluid mechanics reading group organised at the University of Warwick by James Robinson and Jose Rodrigo.
We are particularly grateful to James Robinson for his encouragement and interest in this work.}
We would like to thank Robert Terrell for his English translation of Leray's original work, and Simon Baker, Tobias Barker, Antoine Choffrut and Kenneth Falconer for their helpful comments.
We are particularly grateful to James Robinson for his enthusiasm, encouragement and interest in this work.

WSO is supported by EPSRC as part of the MASDOC DTC at the University of Warwick, Grant No. EP/HO23364/1.

BCP was partially supported by an EPSRC Doctoral Training Award and partially by postdoctoral funding from ERC 616797.

%#######################################################
%#######################################################
\appendix
\renewcommand{\thesection}{A}
\numberwithin{equation}{section}
\section{Appendix}
%'''''''''''''''''''''''''''''''''''''''''''''''''''''''''''''''''''''''''''''''''''''''''''''''''''''''''''''''''''
\subsection{The heat equation and the heat kernel}\label{heat_section_appendix}

Let $v_0 \in L^2$ and $v(t)\coloneqq  \Phi (t)  \ast v_0 $, where $\Phi =(4\pi t)^{-3/2} \mathrm{e}^{-|x|^2/4t}$ is the heat kernel.

If $v_0$ is a vector-valued function which is weakly divergence free then $\mathrm{div}\, v(t) =0$ for each $t>0$ (which can be shown directly by approximating $\nabla \Phi (x-y,t)$ in $L^2$ by the gradient of a smooth and compactly supported $g$ at each $t$). Furthermore for $m\geq 1$, $v$ satisfies
\begin{enumerate}[\rm(i)]
\item $v\in C([0,T);L^2)$ with $v(0) = v_0$, and $\| v(t) \| \leq \| v_0 \|$,
\item $\nabla^m v\in C((0,T);L^2)$ and 
\[\| \nabla^m v(t) \| \leq C_m  \| v_0 \| t^{-m/2}. \] Moreover if $\nabla^m v_0 \in L^2$ then $\nabla^m v\in C([0,T);L^2)$ with $  \| \nabla^m v(t) \| \leq C_m \min_{k\leq m} \lewy  \| \nabla^k v_0 \| t^{-(m-k)/2}  \prawy$.
\item $v\in C((0,T);L^\infty )$ and $ \| v(t) \|_{\infty } \leq C \min \{ \| v_0 \|t^{-3/4}  ,\| v_0 \|_\infty \}$,
\item $\nabla^m v\in C((0,T);L^\infty )$ and 
\[\| \nabla^m v \|_\infty \leq C_m \min \{ \| v_0 \| t^{-m/2-3/4}  ,\| v_0 \|_\infty  t^{-m/2} \},\]
\item $ v, \nabla^m v \in \mathcal{H}^{1/2} ((0,T))$ with the corresponding constants $C_0 (t) \leq c_0 \| v_0 \| /t$, $C_m (t) \leq c_m \| v_0 \| t^{-1-m/2}$.
\end{enumerate}
The inequalities in (i)-(iv) follow directly from Young's inequality \eqref{young_for_convolutions} and the bounds $\| \nabla^m \Phi (t) \|_1 \leq C_m t^{-m/2}$, $\| \nabla^m \Phi (t) \| \leq C_m t^{-m/2-3/4}$, $m\geq 0$ (which can be verified by a direct calculation). The claims regarding continuity in $(0,T)$ follow from the fact that $\Phi$ and all its spacial derivatives $\nabla^m \Phi$, $m\geq 1 $, belong to $C((0,T);L^1)\cap C((0,T);L^2)$ (a consequence of the Dominated Convergence Theorem). The claims regarding continuity at $t=0$ in (i) and (ii) are known as the \emph{approximation of identity}, see e.g. Theorem 1.18 in Chapter 1 of \cite{stein_weiss_introduction} for a proof. Finally, property (v) follows from Morrey's inequality (see, for example, Section 4.5.3 in \cite{evans_gariepy}) and Young's inequality for convolutions \eqref{young_for_convolutions} by writing
\eqnb\label{tempekek}
\| \Phi (t) \ast v_0 \|_{C^{0,1/2}} \leq C \| \nabla \Phi (t) \ast v_0 \|_{6} \leq C \| \nabla \Phi (t) \|_{3/2} \| v_0 \| = C \|v_0 \|/ t. 
\eqne
The claim for the derivatives follows similarly by noting that 
\[\| \nabla^{m+1} \Phi (t) \|_{3/2} = C_m t^{-1-m/2}.\]

As for the pointwise convergence as $t\to 0^+$ we have
\begin{lemma}[pointwise convergence as $t\to 0^+$ of the heat flow]\label{Linfty_conv_heat_eq}
If $x_0\in \RR^3 $ is a continuity point of $v_0$ then
\[
v(x_0,t) \to v_0 (x_0) \qquad \text{ as } t\to 0^+.
\]
If $v_0$ is bounded and uniformly continuous then
\[
\| v(t) - v_0 \|_\infty \to 0 \qquad \text{ as } t\to 0^+.
\]
\end{lemma}
See Section 4.2 in \cite{giga_saal} for a proof.

One can verify that $v(t)$ is a solution of the heat equation $v_t = \Delta v$ in $\RR^3 \times (0,\infty )$  by a direct calculation. It is a unique such solution in the class of functions $C([0,\infty ); L^2)$ satisfying $v(0)=v_0$ for some $v_0\in L^2$, which we make precise in the following lemma.
\begin{lemma}\label{uniqueness_heat_eq_lemma}
Let $w\in C([0,T);L^2)$ be a weak solution of the heat equation with $w(0)=0$, that is
\begin{equation}\label{A1:weakheat}
\int_0^T \int (\phi_t+\Delta \phi ) w  \, \d x \,\d t =0
\end{equation}
for all $\phi\in C_0^\infty (\RR^3 \times [0,T))$. Then $w\equiv 0$.
\end{lemma}
\begin{proof}
We modify the argument from Section 4.4.2 of \cite{giga_saal}. We focus on the case $T<\infty$ (the case $T=\infty$ follows trivially by applying the result for all $T>0$). 

We first show that assumption $w\in C([0,T);L^2)$ implies that \eqref{A1:weakheat} holds also for all $\phi \in C^\infty (\RR^3 \times [0,T) )$ such that
\[
\begin{cases}
\sup_{t\in [0,T) } \left(  \| \p_t \phi (t) \| + \| D^\alpha \phi (t) \| \right) \leq C &\text{ for some } C>0, \\
\mathrm{supp }\, \phi \subset \RR^3 \times [0,T' ) &\text{ for some } T' <T .
\end{cases} \quad ( \ast )
\]
Indeed, let $\theta \in C^\infty (\RR ; [0,1])$ be such that $\theta (\tau ) = 1$ for $\tau \leq 1$ and $\theta (\tau ) = 0 $ for $\tau \geq 2$ (take for instance $\theta (\tau ) \coloneqq q(2-\tau ) / (q(2-\tau )+ q(1-\tau ))$, where $q(s) \coloneqq \e^{-1/s}$ for $s>0$ and $q(s)\coloneqq 0$ otherwise). For $x\in \RR^3$ Let $\theta_j (x) \coloneqq \theta (|x|/j)$. Then $\theta_j \in C_0^\infty $, for some $M>0$ $\| \nabla \theta_j \|_\infty \leq M/j$ and $\| D^2 \theta_j \|_\infty \leq M/{j^2}$, and
\[
\theta_j (x) = \begin{cases}
1 \qquad &|x|<j,\\
0& |x| > 2j .
\end{cases}
\]
Now for $\phi \in C^\infty (\RR^3 \times [0,T) )$ satisfying $(\ast )$ let $T'<T$ be such that $\phi (t) \equiv 0$ for $t\geq T'$ and let
\[\phi_j (x,t)  \coloneqq \theta_j (x) \phi(x,t).\]
Then $\phi_j \in C_0^\infty (\RR^3 \times [0,T))$ and so \eqref{A1:weakheat} gives
\[
\int_0^T \int (\p_t \phi_j+\Delta \phi_j ) w  \, \d x \, \d t =0 ,
\]
or equivalently
\begin{equation}\label{A1:weakheat1_phi_j}
\int_0^T \int \theta_j (\p_t \phi +\Delta \phi ) w  \, \d x \,\d t +\int_0^T \int w \left( 2 \nabla \theta_j \cdot \nabla \phi +  \phi \Delta \theta_j \right)  \, \d x \, \d t  =0 ,
\end{equation}
Since $\theta_j$ converges pointwise to $1$ and
\begin{multline*}
\int_0^T \int |(\p_t \phi+\Delta \phi ) w|  \, \d x \, \d t \leq \int_0^T \| \p_t \phi (t)+ \Delta \phi (t) \| \| \| w (t) \| \, \d t \\\leq T C \sup_{t\in [0,T']}  \| w(t) \| < \infty
\end{multline*}
the Dominated Convergence Theorem gives that the first integral on the left-hand side of \eqref{A1:weakheat1_phi_j} converges to $\int_0^T \int (\p_t \phi+\Delta \phi ) w  \, \d x\, \d t$. The second integral is bounded by
\begin{multline*}
\int_0^T \int |w| \left| 2 \nabla \theta_j \cdot \nabla \phi +  \phi \Delta \theta_j \right|  \, \d x \, \d t \\
\leq T \sup_{t\in [0,T']}  \| w(t) \| (2CM/j + CM/{j^2} ) \stackrel{j\to \infty }{\longrightarrow} 0.
\end{multline*}
Therefore taking the limit $j\to \infty$ in \eqref{A1:weakheat1_phi_j} gives
\[
\int_0^T \int (\p_t \phi+\Delta \phi ) w  \, \d x \, \d t =0
\]
as required.

We will show that
\[
\int_0^T \int \Psi w \, \d x \, \d t = 0
\]
for all $\Psi \in C_0^\infty (\RR^3 \times (0,T) )$. This finishes the proof as $C_0^\infty$ is dense in $L^2$ and so $\int_0^T \int v w\, \d x \,\d t = 0$ for all $v\in L^2 (\RR^3 \times (0,T))$. Hence $w\equiv 0$.

Given $\Psi \in C_0^\infty (\RR^3 \times (0,T) )$ let $T'<T$ be such that $\Psi(t) \equiv 0 $ for $t\geq T'$. Extend $\Psi$ by zero for $t\leq 0 $ and $t\geq T$. There exists a classical solution $\phi$ to the problem 
\begin{equation}\label{phi_sol_of_backwards_heat}
\begin{cases}
\phi_t + \Delta \phi = \Psi \qquad & \text{ in } \RR^3 \times (-\infty,T),\\
\phi(x,T) =0 &\text{ for } x\in \RR^3 .
\end{cases}
\end{equation}
Indeed, denoting $\widetilde{\phi } (x,t) = \phi (x,T-t)$ and $\widetilde{\Psi} (x,t) = \Psi (x,T-t)$ the above problem becomes
\[
\begin{cases}
\widetilde{\phi}_t - \Delta \widetilde{\phi} =- \widetilde{\Psi}\qquad & \text{ in } \RR^3 \times (0,\infty),\\
\widetilde{\phi}(x,0) =0 &\text{ for } x\in \RR^3 .
\end{cases}
\]
A classical solution $\widetilde{\phi }$ is obtained by an application of the Duhamel principle,
\[
\widetilde{\phi }(t) \coloneqq -\int_0^t \Phi (t-s) \ast \widetilde{\Psi } (s)\, \d s,
\]
and such $\widetilde{\phi }$ is smooth (see for example Section 4.3.2 of \cite{giga_saal}). Observe that, since $\Psi (s) \equiv 0$ for $s\geq T'$, we have that $\widetilde{\phi }(t) \equiv 0$ for $t\in [0,T-T']$. Moreover, since $\Psi$ is smooth and compactly supported in $\RR^3 \times \RR$ we have for $|\alpha |\leq 2$
\[
D^\alpha \widetilde{\phi } (t) = -\int_0^t \Phi (t-s) \ast D^\alpha \widetilde{\Psi } (s)\, \d s,
\]
and it follows from Lemma \ref{young_spacetime_continuity} that $D^\alpha \widetilde{\phi } \in C ([0,T],L^2 )$ (note that $\Phi \in L^1_{\loc} ([0,\infty ); L^1)$ due to the bound $\| \Phi (t) \|_1 \leq C_0$). Therefore, we also obtain $\widetilde{\phi }_t = \Delta \widetilde{\phi } - \widetilde{\Psi } \in C([0,T];L^2)$. Hence $\phi (x,t)=\widetilde{\phi}(x,T-t)$ is smooth, $\phi(t)\equiv 0$ for $t \in [T',T]$, and
\[
\sup_{t\in [0,T) } \left(  \| \p_t \phi (t) \| + \| D^\alpha \phi (t) \| \right) = \sup_{t\in [0,T]} \left(  \| \p_t \widetilde{\phi }(t) \| + \| D^\alpha \widetilde{ \phi } (t) \| \right),
\]
that is $\phi$ satisties $(\ast )$. We can therefore use the first part and \eqref{phi_sol_of_backwards_heat} to write
\[
0 = \int_0^T \int (\phi_t + \Delta \phi ) w \, \d x \, \d t = \int_0^T \int \Psi w \, \d x \, \d t . \qedhere 
\]
\end{proof}
\subsection{The extension of Young's inequality for convolutions}\label{young_spacetime_continuity_section}
We will prove the following lemma (Lemma \ref{young_spacetime_continuity}).
\begin{lemma}\label{young_spacetime_continuity_appendix}
If $p,q,r \geq 1$ are such that $1/q=1/p+1/r-1$, $A\in L^1_{\loc}([0,T);L^p)$ and $B\in C((0,T);L^r)$ with $\| B(t) \|_r$ bounded as $t\to 0^+$ then $u$ defined by
\[
u(t)\coloneqq \int_0^t  A(t-s) \ast B(s) \, \d s
\]
belongs to $C([0,T);L^q)$ and
\begin{equation}\label{integral_convolution_bound}
\| u (t) \|_q \leq \int_0^t \| A(t-s) \|_p \| B(s) \|_r \,\d s.
\end{equation}
\end{lemma}
\begin{proof}
The bound \eqref{integral_convolution_bound} is clear from the integral version of Minkowski inequality \eqref{minkowski_integral_version} and from Young's inequality \eqref{young_for_convolutions}. This bound also implies the continuity $u(t) \to 0$ in $L^q$ as $t\to 0^+$. It remains to show that $u \in C((0,T);L^q)$. For this reason let $T'<T$ and $\varepsilon >0$. Let $M>0$ be such that 
\[
\max_{s\in [0,T']} \| B(s) \|_r \leq M, \quad \int_0^{T'} \| A(s) \|_p \,\d s  \leq M.
\] 
From the assumption on $A$ we see that there exist $\eta >0$ such that
\[
\int_0^{\eta} \| A(s) \|_p \, \d s \leq \varepsilon .
\]
Since $B$ is uniformly continuous into $L^r$ on $[\eta/2 , T']$ there exists $\delta \in (0, \eta/2)$ such that
\[
\| B(s)-B(t) \|_{r} < \varepsilon 
\]
whenever $s,t \in [\eta/2 , T']$, $t>s$ and $t-s < \delta $. Letting $t_1,t_2 \in (0,T']$ be such that $t_1<t_2$, $t_2-t_1 <\delta $ we obtain
\begin{equation}\label{temp_decomp}
\begin{split}
\MoveEqLeft
u(t_2)-u(t_1)=\int_0^{t_2}  A(t_2-s) \ast B(s) \, \d s - \int_0^{t_1} A(t_1-s)\ast B(s)  \, \d s \\
&= \int_0^{t_2}  A(t_2-s) \ast B(s) \, \d s - \int_{t_2-t_1}^{t_2} A(t_2-s)\ast B(s-(t_2-t_1))  \, \d s\\
&=  \int_0^{t_2-t_1}  A(t_2-s) \ast B(s) \, \d s\\
&\quad+ \int_{t_2-t_1}^{t_2} A(t_2-s)\ast \left( B(s) -B(s-(t_2-t_1)) \right)  \, \d s.
\end{split}
\end{equation}
Thus, using the integral Minkowski inequality \eqref{minkowski_integral_version} and Young's inequality for convolutions \eqref{young_for_convolutions} we obtain
\[
\begin{split}
\| u(t_2)-u(t_1) \|_{q} &\leq \int_0^{t_2-t_1} \| A(t_2-s)\|_p \| B(s) \|_{r} \,\d s\\
& + \int_{t_2-t_1}^{t_2} \| A(t_2-s) \|_p \| B(s) -B(s-(t_2-t_1))  \|_{r} \,\d s .
\end{split}
\]
Since $t_2-t_1<\delta < \eta $ the first term on the right-hand side is bounded by $M \varepsilon $. As for the second term, if $t_2 < \eta$ it can be similarly bounded by $2M \varepsilon $. If not, then we decompose it into $\int_{t_2-t_1}^{\eta } + \int_\eta^{t_2}$, we bound the first of the resulting integrals by $2M\varepsilon$ and, as for the second one, observe that $s,s-(t_2-t_2) \in [\eta/2, T']$ whenever $s\in [\eta, T']$ to write
\begin{multline*}
\int_{\eta}^{t_2 } \| A(t_2-s) \|_p \| B(s) -B(s-(t_2-t_1))  \|_{r}\, \d s \\\leq \varepsilon \int_\eta^{t_2} \| A(t_2-s ) \|_p \, \d s  \leq \varepsilon M.
\end{multline*}
Therefore $\| u(t_2)-u(t_1 ) \|_q \leq 4M \varepsilon$. 
\end{proof}

\subsection{Decay estimates of $P(x,t)$}\label{decay_p_appendix}
Let $P(x,t):=\frac{1}{|x|} \int_0^{|x|} \frac{\mathrm{e} ^{-\xi^2/4t}}{t^{1/2}} \,\d \xi$ (see \eqref{def_of_p}). We have the following decay estimates.
\begin{theorem}\label{decay_p_theorem}
\begin{equation}\label{decay_p_inequality}
| \nabla^m P(x,t) | \leq \frac{C_m}{(|x|^2+t)^{(m+1)/2}}
\end{equation}
for all $x\in \RR^3$, $t>0$.
\end{theorem} 
\begin{proof}
Let us first assume that $t=1$ (the general case will follow from the rescaling $P(x,t)=\frac{1}{\sqrt{t}}P(x/\sqrt{t},1)$). We claim that any partial derivative $D^\alpha P(x,1)$ of order $|\alpha |=m$ is of the form
\begin{equation}\label{repr_der_of_p}
D^\alpha P(x,1)= \frac{Q_\alpha(x)}{|x|^{2m}} P(x,1)+E_\alpha (x)
\end{equation}
for $m\geq 0$ and $|x|>1$, where $Q_\alpha(x)$ denotes a polynomial of degree less than or equal to $|\alpha |=m$ and $E_\alpha(x)$ denotes a function that is smooth in $|x|>1$ and whose derivatives of all orders decay exponentially when $|x|\to \infty$. The case $m=0$ follows trivially with $E_0\equiv 0$. For the inductive step, assume \eqref{repr_der_of_p} holds for all partial derivatives $D^\alpha P(x,t)$ with $|\alpha |=m\geq 0$ and write
\[
\begin{split}
\partial_{x_i} D^\alpha P(x,1) &= \partial_{x_i} \left( \frac{Q_\alpha(x)}{|x|^{2m}} P(x,1)+E_\alpha (x) \right) \\
&= \frac{\tilde Q_{m-1} (x)}{|x|^{2m}}P(x,1) - \frac{2mQ_\alpha (x)\, x_i }{|x|^{2m+2}}P(x,1)\\
&+\frac{Q_\alpha(x)}{|x|^{2m}} \left( \frac{x_i}{|x|^2} P(x,1)+ \frac{x_i}{|x|^2} \e^{-|x|^2/4} \right)+ \partial_{x_i} E_\alpha (x)\\
&= \frac{P(x,1)}{|x|^{2m+2}} \left( |x|^2 \tilde Q_{m-1} (x) -(2m-1)Q_\alpha (x) x_i  \right)\\
&+ \frac{Q_\alpha(x)x_i}{|x|^{2m+2}}  \e^{-|x|^2/4} +  \partial_{x_i} E_\alpha (x) 
\end{split}
\]
for $i=1,2,3$, where $\tilde Q_{m-1} (x)$ denotes some polynomial of degree less than or equal to $m-1$. Clearly, the last bracket is some polynomial of degree less than or equal to $m+1$ and the remaining two terms are smooth in $|x|>1$ and decay exponentially as $|x|\to \infty$. Hence the induction follows. Because $P(x,1)$ decays like $|x|^{-1}$ as $|x|\to \infty$ we see from \eqref{repr_der_of_p} that 
\[
|\nabla^m P(x,1) | \leq \frac{C_m}{|x|^{m+1}} \leq \frac{C_m}{(|x|^2+1)^{(m+1)/2}}
\]
holds for $|x|\geq 2$. As $P(\cdot ,1)$ is a smooth function (see Section \ref{oseen_kernel_section}) we have $|\nabla^m P(x,1)| \leq C_m \leq {C_m}{(|x|^2+1)^{-(m+1)/2}}$ for all $x\in B(0,2)$. Hence \eqref{decay_p_inequality} follows in the case $t=1$.  Finally, the rescaling $P(x,t)=\frac{1}{\sqrt{t}}P(x/\sqrt{t},1)$ yields 
\[\begin{split}
|\nabla^m P(x,t) | &= \left| t^{-(m+1)/2} \left[ \nabla^m P(y,1) \right]_{y=x/\sqrt{t}} \right| \\
&\leq  t^{-(m+1)/2} \frac{C_m}{ \left( \left|x/\sqrt{t} \right|^2+1 \right)^{(m+1)/2}} = \frac{C_m}{\left( |x|^2+t\right)^{(m+1)/2}}. \qedhere
\end{split}
\]
\end{proof}
\subsection{Properties of the Stokes equations}\label{appendix_stokes_eq_X}
Here we present proofs of some results from Section \ref{stokes_eq_section}. Namely we complete the proof of Theorem \ref{classical_sol_if_X_smooth_theorem}, show property (iii) of the representation \eqref{repr_formula} and the continuity $\nabla u \in C((0,T);L^\infty )$ given representation \eqref{repr_formula1} with $Y\in \mathcal{H}^{1/2} ((0,T))$, and we show uniqueness of distributional solutions of the Stokes equations (Theorem \ref{uniqueness_stokes}).
\subsubsection{Proof of Theorem \ref{classical_sol_if_X_smooth_theorem}}\label{app_stokes_pf_of_thm}
Recall that it remains to verify that if for some $R>0$
\[
F\in C^\infty (\RR^3 \times [0,T); \RR^3 )\quad \text{ and } \quad \supp \, F(t) \subset B(0,R)\text{ for } t\in [0,T) 
\]
then the pair of functions $u_2$, $p$ given by \eqref{repr_formula}, \eqref{repr_formula_p} is a classical solution of the problem
\[
\left\{
\begin{array}{rl}
\p_t u_2 -\Delta u_2 +\nabla p  &=F,\\
\mathrm{div}\, u_2 &=0,\\
 u_2 (0)&=0,
\end{array}  
\right.
\]
and $u_2 \in C([0,T); L^2 )$. Indeed, as remarked after the statement of Theorem \ref{classical_sol_if_X_smooth_theorem}, then $u=u_1+u_2$ is a classical solution of the Stokes equations \eqref{eqLinSys1}, \eqref{eqLinSys1_incomp} with $u(0)=u_0$ and $u\in C([0,T); L^2 )$.

Note that, since $F$ is smooth and compactly supported in space and since $\mathcal{T}\in L^1_{\loc} ([0,T);L^2)$ (see \eqref{integral_est_T}) we can use Lemma \ref{young_spacetime_continuity} to obtain $u_2 \in C([0,T);L^2 )$ and $\| u_2 (t) \| \to 0$ as $t \to 0^+$ (which means that $u_2$ satisfies the initial condition $u_2(0)=0$). Moreover, since $\mathcal{T}\in C((0,\infty );L^2)$ (see \eqref{fcn_spaces_for_T}), we deduce that the functions $\nabla u_2$, $\Lap u_2$, $\p_tu_2$ are continuous (an application of the Dominated Convergence Theorem) and belong to $C((0,T);L^2)$ (a consequence of Lemma \ref{young_spacetime_continuity}). Similarly $\nabla p$ is continuous and $\nabla p \in C([0,T); L^2)$ (by an application of the Plancherel Lemma (Lemma \ref{CZ_lemma})). Therefore, since the Fourier transform is an isometry from $L^2$ into $L^2$, we see that $u_2$, $p$ satisfy the claim above if and only if
\begin{equation}\label{eqU2Fourier}
\p_t \widehat{ u_2}(\xi,t) + 4\pi^2|\xi|^2 \widehat{u_2}(\xi,t)+2\pi\irm \, \xi \, \widehat{ p}(\xi,t)=\widehat{F}(\xi,t)
\end{equation}
and
\begin{equation}\label{eqU2Fourier_b}
\xi\cdot \widehat{ u_2}(\xi,t)=0
\end{equation}
hold for $t>0$ and almost every $\xi \in \RR^3$. Here $\widehat{u_2}$, $\widehat{p}$, $\widehat{F}$ denote the Fourier transform $\mathcal{F}$ of $u_2$, $p$, $F$, respectively. Since $p$ satisfies $\Delta p = \mathrm{div} \, F$ we obtain
\[
2\pi\irm \,\xi \, \widehat{ p}(\xi,t) = \frac{\xi\otimes\xi}{|\xi|^2}\widehat F(\xi,t),
\]
where $\xi \otimes \xi $ denotes the $3\times 3$ matrix with components $\xi_i \xi_j$. Therefore \eqref{eqU2Fourier} is equivalent to 
\begin{equation}\label{eqU2Fourier_equiv}
\left( \p_t+4\pi^2|\xi|^2 \right) \widehat{ u_2}(\xi,t)=\left(I-\frac{\xi\otimes\xi}{|\xi|^2}\right) \widehat{F}(\xi,t).
\end{equation}
Now since
\begin{equation}\label{ft_heat_kernel}
\mathcal{F} \left[ \Phi (\cdot ,t) \right] = \mathrm{e}^{-4\pi^2 t |\xi |^2}
\end{equation}
(see e.g. Theorem 1.13 in Chapter 1 of \cite{stein_weiss_introduction} for a proof of this fact) and $-\Delta P = \Phi$ (see \eqref{laplace_relation}), we obtain 
\[
\mathcal{F} \left[ P(\cdot ,t) \right] = \frac{1 }{4\pi^2 |\xi |^2} \mathrm{e}^{4\pi^2 t |\xi |^2},
\]
and so
\[
\mathcal{F} \left[ \p_i \p_j P(\cdot ,t) \right] = \frac{-\xi_i \xi_j }{|\xi |^2} \mathrm{e}^{-4\pi^2 t |\xi |^2}.
\]
Hence for all $t>0$
\[
\mathcal{F} [\mathcal{T}(t)]=\left(I- \frac{\xi\otimes\xi}{|\xi|^2} \right) \mathrm{e}^{-4\pi^2 t|\xi|^2}.
\]
Since $u_2(t) = \int_0^t \mathcal{T}(t-s)\ast F(s)\,\d s$, 
\begin{equation}\label{FT_of_u2}
\widehat{u_2} (\xi, t) = \int_0^t \left(I- \frac{\xi\otimes\xi}{|\xi|^2} \right) \widehat{F} (\xi, s) \mathrm{e}^{-4\pi^2 (t-s)|\xi|^2}\,   \d s,\qquad  \xi \ne 0,
\end{equation}
and so \eqref{eqU2Fourier_b} holds for $\xi \ne 0$ and
\begin{multline*}
\left( \p_t+4\pi^2|\xi|^2 \right) \widehat{ u_2}(\xi , t)\\
=\left( \p_t+4\pi^2|\xi|^2\right) \int_0^t\left(I-\frac{\xi\otimes\xi}{|\xi|^2}\right) \mathrm{e}^{-4\pi^2 (t-s)|\xi|^2}\widehat{F}(\xi,s)\,\d s\\
=\left(I -\frac{\xi\otimes\xi}{|\xi|^2} \right)\widehat{F}(\xi,t)
\end{multline*}
for $\xi \ne 0$, that is \eqref{eqU2Fourier}, as claimed. 

\subsubsection{Property (iii) of the representation \eqref{repr_formula}}\label{app_prop_iii}
We need to show that if $F\in C([0,T);L^2)$, $u_0 \in L^2$ and $u$ is given by \eqref{repr_formula},
\[u(t)= u_1 (t)+u_2 (t) = \Phi (t)\ast u_0  +\int_0^t  \mathcal{T}(t-s) \ast F(s)\, \d s ,\]
then $u\in C([0,T); L^2 )$ with
\begin{equation}\label{l2_bound_appendix}
\| u (t) \| \leq \int_0^t \| F(s) \|\,  \d s + \| u_0 \| \qquad \text{ for }t\in (0,T).
\end{equation}
Moreover $u$ satisfies the energy dissipation equality
\begin{equation}\label{en_dissipation_appendix}
\| u(t) \|^2 - \| u_0 \|^2 + 2 \int_0^t \| \nabla u(s) \|^2\,  \d s = 2 \int_0^t \int u \cdot F\,\, \d x \, \d s 
\end{equation}
for $t\in (0,T)$. 

We prove these claim by considering two cases.
 
Case 1. $F$ is regular, that is satisfies \eqref{regularity_of_X_requirement} for some $R>0$. For such $F$, due to Theorem \ref{classical_sol_if_X_smooth_theorem} (or rather to the proof above), $u=u_1+u_2$ and $p$ satisfy $u(0)=u_0$,
\begin{equation}\label{eq_for_u2_pointwise}
\p_t u_1 - \Delta u_1=0,\quad  \p_t u_2 -  \Delta u_2 +  \nabla p =  F \qquad \text{ in } \RR^3 \times (0,T)
\end{equation}
and $u_1,u_2 \in C([0,T);L^2)$, where $p$ is given by \eqref{repr_formula_p}.

Note that for $y\in B(0,R)$ and $|x|>2R$
\[ |x|/2<|x|-R<|x-y|<|x|+R \]
and using \eqref{eqOseenKerEst2} we can write for such $x$ and for all $t$
\[
\begin{split}
|u_2 (x,t)|&\leq \int_0^t \int_{B(0,R)} \frac{C_0}{\left( |x-y|^2 + t-s \right)^{3/2}} |F(y,s)|\,  \d y\,  \d s \\
&\leq \frac{C |B(0,R)|^{1/2}t}{|x|^{3}} \max_{s\in[0,t]} \| F(s) \|,
\end{split}
\]
which gives a decay $\sim | x|^{-3}$ of $|u_2(x,t)|$ as $|x| \to \infty$ that is uniform on any compact time interval $[0,T']$, where $T'<T$. Similarly one can derive a decay $\sim | x|^{-5}$ of $|\Delta u_2(x,t)|$ and decay $\sim | x |^{-3}$ of $|\nabla p(x,t)|$. Thus \eqref{eq_for_u2_pointwise} gives the decay $\sim |x|^{-3}$ of $\p_t u_2 (x,t)$. Letting 
$\delta,t \in (0,T)$, $t>\delta$ and using this decay as well as the fact that $u_1$ and all its derivatives belong to $C((0,T);L^2)$ (see (ii) in Appendix \ref{heat_section_appendix}) we can integrate the equality
\[
u \cdot \p_t u - u \cdot \Delta u + u \cdot \nabla p = u \cdot F
\]
on $ \RR^3 \times (\delta , t)$ to obtain
\begin{equation}\label{en_dissipation_Xcompact}
 \|u (t) \|^2 -  \| u(\delta ) \|^2 + 2\int_\delta^t \| \nabla u (s) \|^2 \,\d s = 2\int_\delta^t \int u \cdot F \, \d x \,\d s .
\end{equation}
Since $u\in C([0,T);L^2)$ we can take the limit $\delta \to 0^+$ (and apply the Monotone Convergence Theorem) to obtain \eqref{en_dissipation_appendix}. As for the bound \eqref{l2_bound_appendix} note that the above equality and the Cauchy-Schwarz inequality give
\[
\frac{\d }{\d t} \| u(t) \|^2 \leq  2 \int u (t) \cdot F(t) \, dx \leq 2 \| u(t) \| \, \| F(t) \|,  \qquad t\in (0,T) ,
\]
that is
\[
\frac{\d }{\d t} \| u(t) \| \leq   \| F(t) \|,  \qquad t\in (0,T).
\]
Integrating this inequality in $t$ gives \eqref{l2_bound_appendix}.\\

Case 2. $F\in C([0,T);L^2)$. Let $u$ be the velocity field corresponding to $F$ and the initial velocity field $u_0$, and consider a compact time interval $[0,T']\subset [0,T)$. For such $F$ consider a sequence $\{F_R \} \subset C([0,T'];L^2)$ such that $F_R \in C^\infty (\RR^3 \times \RR$), $\supp \, F_R (t) \subset B(0,R)$ for all $t$ and 
\[
F_R \to F \qquad \text{ in } C([0,T'];L^2) \text{ as } R\to \infty .
\] 
(See Lemma \ref{approx_of_X_lemma} for a proof of the existence of such a sequence.) Let $u_R$ denote the velocity field corresponding to $F_R$ and to the initial velocity $u_0$. Since the representation formula \eqref{repr_formula} is linear we can apply \eqref{l2_bound_appendix} to the difference $u_{R_1}-u_{R_2}$ for $R_1,R_2 >0$ to see that $\{ u_R \}$ is Cauchy in $C([0 ,T'];L^2)$ as $R\to \infty$. Thus $u_R \to u'$ in $C([0 ,T'];L^2 )$ as $R\to \infty$ for some $u' \in C([0,T'];L^2)$. However, Lemma \ref{prop_of_up_lemma} (i) applied to the difference $u_R-u$ gives $\| u(t) - u_R (t)\|_\infty \to 0$ for every $t\in (0,T']$, and thus $u=u'$ (since convergence in $L^2$ gives convergence almost everywhere on a subsequence). Hence $u\in C([0,T'];L^2)$ and \eqref{l2_bound_appendix} follows for $t\in [0,T']$ by taking $R\to \infty$ in the corresponding inequality for $u_R$. As for the energy dissipation equality \eqref{en_dissipation_appendix}, let $\delta \in (0,T')$ and note that property (ii) of the representation (see Section \ref{section_forcing_X}) applied to the difference $u_R-u$ gives
\[
\nabla u_R \to \nabla u \qquad \text{ in } C([\delta , T'];L^2) \text{ as } R\to \infty .
\]
(Note that the convergence is not in $C([0,T'];L^2)$ since each $\nabla u_R$ need not belong to this space; see (ii).) Since Case 1 gives for each $R>0$, $t\in [0,T']$ 
\[
\frac{1}{2} \|u_R (t) \|^2 - \frac{1}{2} \| u_R(\delta ) \|^2 + \int_\delta^t \| \nabla u_R (s) \|^2 \,\d s = \int_\delta^t \int u_R \cdot F_R \, \d x \,\d s, 
\]
(see \eqref{en_dissipation_Xcompact}) we can take the limit $R\to \infty $ and then take the limit $\delta  \to 0^+$ (and apply the Monotone Convergence Theorem again) to obtain \eqref{en_dissipation_appendix} for all $t\in [0,T']$.
\subsubsection{The continuity $\nabla u \in C((0,T);L^\infty )$ for $u$ given by formula \eqref{repr_formula1}}\label{app_stokes_continuity}
We need to show that if $Y \in C((0,T),L^\infty )$ is weakly divergence free, $ \| Y(t) \|_\infty $ remains bounded as $t \to 0^+$ and $Y\in \mathcal{H}^{1/2} ((0,T))$ and if $u_2$ is given by
\[
u_{2} (t) \coloneqq    \int_0^t  \nabla  \mathcal{T}(t-s) \ast \left[ Y(s) Y (s) \right] \,  \d s 
\]
(recall the notation \eqref{notation_oseen_ker_conv}) then $\nabla u \in C((0,T);L^\infty )$.

In order to prove it, fix $T'<T$ and let $M>0$ be such that
\[
\| Y(t) \|_\infty , C_0(t) \leq M \qquad \text{ for } t\in [0,T'].
\]
Fix $\varepsilon >0$. Let $\eta >0$ be such that
\[
\int_0^t \frac{1}{(t-s)^{3/4}} \,\d s \leq \frac{\varepsilon}{M^2} \quad \text{ for } t\leq 2\eta .
\]
Since for each pair $i,k$ we have $Y_i Y_k \in C((0,T);L^\infty )$, $Y_iY_k$ is uniformly continuous into $L^\infty $ on time interval $[\eta/2,T']$. Thus there exists $\delta \in (0,\eta/2)$ such that for all $i,k$
\[
\| Y_i (t) Y_k(t) - Y_i (s) Y_k (s) \|_\infty \leq \frac{\varepsilon \,\eta}{T'} \]
whenever $s,t\in [\eta/2 , T']$ and $|t-s|<\delta$. Now let $t_1,t_2\in [0,T']$ be such that $t_1<t_2$ and $t_2-t_1<\delta $, and calculate
\[
\begin{split}
\MoveEqLeft
\p_l u_{2,j} (t_1) - \p_l u_{2,j} (t_2) = \int_0^{t_2-t_1} \p_{li} \mathcal{T}_{jk} (t_2-s) \ast \left[ Y_i (s) Y_k (s) \right] \,\d s\\
&+ \int_{t_2-t_1}^{t_2} \p_{li} \mathcal{T}_{jk} (t_2-s) \ast  \left[ Y_i (s) Y_k (s)\right. \\
& \hspace{3cm}- \left. Y_i (s-(t_2-t_1))Y_k (s-(t_2-t_1)) \right] \,\d s
\end{split}
\]
(see calculation \eqref{temp_decomp}) and denote the two integrals on the right-hand side by $I_1$, $I_2$ respectively. As for $I_1$, using the same trick of employing the H\"older continuity of $Y_k(s)$ as indicated in \eqref{temp_prop_up}, we obtain
\begin{align*}
\left| I_1 \right| &\leq \int_0^{t_2-t_1} \int \left| \p_{li} \mathcal{T}_{jk} (x-y,t_2-s) Y_i (y,s) \left[Y_k (y,s)-Y_k (x,s) \right]\right| \,\d y \, \d s \\
&\leq M^2 \int_0^{t_2-t_1} \int \frac{C |x-y |^{1/2}}{\left( |x-y|^2 + (t_2-s) \right)^{5/2}} \,\d y \, \d s
\\&&\mathllap{=M^2 \int_0^{t_2-t_1} \frac{C}{(t_2-s)^{3/4}} \,\d s\leq C \varepsilon .\hspace{5pt}}
\end{align*}
If $t_2 \leq 2\eta $ then one can bound $|I_2 |$ in a similar way   to obtain 
\[
|I_2 |\leq 2C \varepsilon .
\]
Otherwise, we write $\int_{t_2-t_1}^{t_2} = \int_{t_2-t_1}^{\eta} + \int_{\eta}^{t_2-\eta} + \int_{t_2-\eta}^{t_2} $ and denote the resulting three integrals by $I_{2,1}$, $I_{2,2}$, $I_{2,3}$, respectively. Since the length of the intervals of integration in $I_{2,1}$, $I_{2,3}$ is less than $2\eta$ we obtain, as above,
\[
\left| I_{2,1} \right|, \left| I_{2,3} \right| \leq 2C\varepsilon .
\]
For $I_{2,2}$ note that $s,s-(t_2-t_1)\in [\eta/2,T']$ for each $s$ from the interval of integration to write

\begin{align*}
\left| I_{2,2} \right| &\leq \int_\eta^{t_2-\eta } \| \p_{li} \mathcal{T}_{jk} (t_2-s) \|_1 \| Y_i(s)Y_k(s) \\
&\hspace{3cm}- Y_i (s-(t_2-t_1))Y_k(s-(t_2-t_1)) \|_\infty \,\d s\\
&&\mathllap{\leq \frac{\varepsilon \, \eta }{T'} \int_\eta^{t_2-\eta} \frac{C}{t_2-s } \,\d s \leq C \varepsilon ,}
\end{align*}
where we used the Minkowski inequality \eqref{minkowski_integral_version}, Young's inequality \eqref{young_for_convolutions} and the bound $\| \nabla^2 \mathcal{T}(t) \|_1 \leq C/t$ (see \eqref{eqOseenKerEst2}). Thus altogether
\[
\left| \p_l u_{2,j} (t_1) - \p_l u_{2,j} (t_2) \right| \leq 5 C \varepsilon \qquad l,j=1,2,3
\]
and the continuity $\nabla u_2 \in C((0,T);L^\infty )$ follows.
\subsubsection{Uniqueness of distributional solutions of the Stokes equations}\label{app_stokes_uniqueness}
We will show that if $u,p$ are such that $u\in C([0,T); L^2)$ is weakly divergence free, $p\in L^1_{\loc} (\RR^3 \times [0,T))$, and 
\begin{equation}\label{weak_homogeneous_solution_app}
\int_0^T \int \left( (\phi_t + \Delta \phi ) \cdot u + p \, \mathrm{div}\, \phi \right) \,\d x \, \d t =0
\end{equation}
for all $\phi \in C_0^\infty (\RR^3 \times [0,T) ; \RR^3)$, then $u\equiv 0$.

\begin{proof} Fix $\psi \in C_0^\infty (\RR^n \times [0,T))$ let $\phi \coloneqq \nabla \psi$. Then $\mathrm{div}\,  \phi = \Delta \psi$ and so \eqref{weak_homogeneous_solution_app} gives
\begin{equation}\label{p_weakly_harmonic}
\int_0^T \int  p \, \Delta \psi \, \d x \, \d t = \int_0^T \int  \nabla (\psi_t + \Delta \psi ) \cdot u  \, \d x \, \d t =0
\end{equation}
since $u$ is weakly divergence free.

For $\varepsilon >0$ let
\[
v (x,t) \coloneqq \int_0^t (J_\varepsilon u)(x,s)\, \d s, \qquad {q} (x,t) \coloneqq \int_0^t (J_\varepsilon p)(x,s)\, \d s.
\]
We first show that $v$, $q$ also satisfy \eqref{weak_homogeneous_solution_app}, 
\begin{equation}\label{weak_homogeneous_solution_vq}
\int_0^T \int \left( (\phi_t + \Delta \phi ) \cdot v + q \, \mathrm{div}\, \phi \right) \,\d x \,\d t =0
\end{equation}
for all $\phi \in C_0^\infty (\RR^3 \times [0,T) ; \RR^3)$, or equivalently
\[
\begin{split}
0 &= \int_0^T \int \int_0^t \int \left( \eta_\varepsilon (y) u(x-y,s) \cdot (\phi_t (x,t)+ \Delta \phi (x,t) ) \right.\\
&\hspace{70pt}+ \left.\eta_\varepsilon (y)  p(x-y,s) \, \mathrm{div}\, \phi (x,t) \right) \,\d y \,\d s \,\d x \,\d t \\
&=\int  \eta_\varepsilon (y)  \int \int_0^T  \left( \int_0^t u(x,s) \,\d s \cdot (\phi_t(x+y,t) + \Delta \phi(x+y,t) )  \right. \\
&\hspace{90pt}+ \left. \int_0^t p (x,s) \,\d s \,\, \mathrm{div} \, \phi (x+y,t) \right)\,\d t \,\d x \,\d y .
\end{split}
\]
We will show that the expression under the $y$ integral vanishes. In fact, for fixed $y\in \RR^3$ let 
\[
\Psi (x,t) \coloneqq  -\int_{t}^{T'} \phi(x+y,s) \,\d s,
\]
where $T'<T$ is such that $\phi (t) \equiv 0$ for $t\geq T'$. Clearly $\Psi \in C_0^\infty (\RR^3 \times [0,T);\RR^3)$ and so \eqref{weak_homogeneous_solution_app} gives
\begin{align*}
0&=  \int \int_0^T u(x,t) \cdot (\Psi_t (x,t) + \Delta \Psi(x,t) ) \, \d t\,\d x \\
&\quad+  \int \int_0^T  q  (x,t) \, \mathrm{div} \, \Psi (x,t)  \, \d t\,\d x.
\end{align*}
Integration by parts in $t$ and the identity
\[\p_t ( \Psi_{t} (x,t)+ \,\Delta \Psi (x,t) )=\phi_t (x+y,t)+\Delta \phi (x+y,t)
\]
give
\begin{align*}
\int \int_0^T  &\left( \int_0^t u(x,s) \,\d s \cdot (\phi_t(x+y,t) + \Delta \phi(x+y,t) ) \right.\\
&\quad\left.+ \int_0^t p (x,s) \,\d s \,\, \mathrm{div} \, \phi (x+y,t) \right)\,\d t \,\d x =0
\end{align*}
as required. Therefore $v$, $q$ indeed satisfy \eqref{weak_homogeneous_solution_vq}. Moreover, letting 
\[\Psi (x,t) \coloneqq  -\int_{t}^{T'} \psi(x+y,s) \,\d s,\]
this time for a scalar test function $\psi$ we obtain from \eqref{p_weakly_harmonic} that for each $y\in \RR^3$
\begin{align*} 0&=\int_0^T  \int p \Delta \Psi  \,\d x \,\d t= - \int_0^T  \int p(x,t) \int_t^{T'} \Delta \psi (x+y,s) \,\d s \, \d x\, \d t \\
&&\mathllap{= - \int_0^T  \int \left( \int_0^t p(x,s)\, \d s \right) \Delta \psi (x+y,t) \, \d x \,\d t.}
\end{align*}
Hence Fubini's theorem gives

\begin{multline*}
\int_0^T  \int q \Delta \psi \, \d x \,\d t = \int_0^T \int \int_0^t \int \eta_\varepsilon (y) p(x-y,s ) \Delta \psi (x,t) \, \d y \,\d s \,\d x \,\d t \\
= \int \eta_\varepsilon (y) \int_0^T \int \left( \int_0^t p(x,s)\, \d s \right) \Delta \psi (x+y, t) \,\d x \,\d t \,\d y =0,
\end{multline*}

that is, like $p$, $q$ satisfies \eqref{p_weakly_harmonic}. Therefore \eqref{weak_homogeneous_solution_vq} applied with $\Delta \phi$ in place of $\phi$ gives 
\[
0 = \int_0^T \int \Delta (\phi_t + \Delta \phi ) \cdot v \, \d x \,\d t = \int_0^T \int  (\phi_t + \Delta \phi ) \cdot \Delta v \, \d x\, \d t .
\]
The uniqueness of weak solutions to the heat equation (see Lemma \ref{uniqueness_heat_eq_lemma}) now implies $\Delta v\equiv 0$ almost everywhere in $\RR^3\times [0,T)$, hence $\Delta v \equiv 0$ everywhere by continuity of $\Delta v$. Hence Liouville's theorem implies that $v(t)$ is constant for each $t$. Therefore $v(t)\equiv 0$ at each $t$ due to the fact $v\in C([0,T);L^2)$. Hence, differentiating the definition of $v$ (in $t$), we see that $J_\varepsilon u (t)\equiv 0$ for each $t$ and $\varepsilon$. The almost everywhere convergence of the mollification (see Lemma \ref{prop_molli}, (v)) gives $u(t) \equiv 0$ for each $t$, as required. 
\end{proof}

\subsection{Integral inequalities}\label{integral_ineq}
%Let $T>0$ or $T=\infty$.
\begin{lemma}\label{lem_integral_ineq_app}
Suppose $g>0$ is a continuous function on $(0,T)$ that is locally integrable $[0,T )$, that functions $f,\phi: (0,T)\to \RR^+$ satisfy
\begin{eqnarray}
f(t) &\leq & \int_0^t g(t-s) f(s)^2 \,\d s + a(t) , \label{f_ineq} \\
\phi (t) &\geq & \int_0^t g(t-s) \phi (s)^2 \,\d s + b(t) \label{phi_ineq}
\end{eqnarray}
for all $t\in (0,T)$, where $a$, $b$ are continuous functions satisfying $a\leq b$, $\phi$ is continuous, and that $f^2$ and $\phi^2$ are integrable near $0$.
Then $f\leq \phi $ on $ (0,T)$.
\end{lemma}
\begin{proof}
The proof proceeds in three steps.

\vspace{10pt}
\noindent\emph{Step 1. The case $ a(t) \leq b(t)- \delta $ for $t \in (0,\tau )$ for some $\delta, \tau >0$}.

Let 
\[I\coloneqq \lewy t' \, : \, \int_0^{t} g(t-s)f(s)^2 \,\d s + a(t) < \phi(t) \text{ for all } t\in (0,t']  \prawy.\]
Note that \eqref{f_ineq} gives $f(t)< \phi (t)$ for $t\in I$. Let $t_0\in (0,\tau )$ be such that
\[
\left| \int_0^t g (t-s) f(s)^2 \,\d s \right| < \frac{ \delta }{2} , \quad \left| \int_0^t g (t-s) \phi (s)^2 \,\d s \right| <\frac{ \delta }{2}  
\]
for $t\in (0,t_0]$. Then for $t\in (0,t_0]$ \eqref{phi_ineq} gives
\begin{multline*}
\int_0^t g(t-s) f(s)^2 \,\d s + a(t) < \delta/2 +b(t )-\delta  \\\leq \phi(t) - \int_0^t g(t-s) \phi (s)^2 \,\d s - \delta/2 < \phi (t),
\end{multline*}
that is $t_0 \in I$. Now let $T':=\sup I$. We need to show that $T'=T$. Suppose otherwise that $T'<T$. Then
\[
\int_0^{T'} g(T'-s)f(s)^2 \,\d s + a(T')  < \int_0^{T'} g(T'-s)\phi (s)^2 \,\d s + b(T') \leq \phi(T') 
\]
by \eqref{phi_ineq}. By continuity we obtain $\int_0^{t} g(t-s)f(s)^2 \,\d s + a(t) <  \phi(t)$ for $t\in [T',T'']$ for some $T''>T'$. Hence $T''\in I$, which contradicts the definition of $T'$. Therefore indeed $T'=T$ and the lemma follows in this case.

\vspace{10pt}
\noindent\emph{Step 2. The case $\liminf_{t\to 0^+} (b(t)-a(t))=0$: there exists $t_0>0$ such that $f(t)\leq \phi(t)$ for $t\in (0,t_0)$}.

Let $t_0>0$ be small enough such that 
\[
\left| \int_0^t g(t-s) \,\d s \right| <\frac{1}{4}, \quad \left| \int_0^t g(t-s) f(s) \,\d s \right| <\frac{1}{4}\qquad \text{ for } t\in [0,t_0].
\]
Let $\varepsilon \in (0,1/3)$ and consider $f_\varepsilon (t):= f(t) -\varepsilon$. Then $ \varepsilon/2+3\varepsilon^2/4 - \varepsilon \leq - \varepsilon/4$ and so $f_\varepsilon$ satisfies the inequality
\begin{align*}
\MoveEqLeft[0]
f_\varepsilon (t) \leq  \int_0^t g(t-s) (f_\varepsilon (s)+\varepsilon )^2 \,\d s + a(t) -\varepsilon \\
&\leq \int_0^t g(t-s) f_\varepsilon (s)^2 \,\d s +2\varepsilon \int_0^t g(t-s) (f(s)-\varepsilon ) \,\d s\\
&\quad + \varepsilon^2 \int_0^t g(t-s) \,\d s + a(t) - \varepsilon \\
&\leq \int_0^t g(t-s) f_\varepsilon (s)^2 \,\d s +\varepsilon/2 + \varepsilon^2/2 + \varepsilon^2/4  + a(t) - \varepsilon \\
&&\mathllap{\leq \int_0^t g(t-s) f_\varepsilon (s)^2 \,\d s + a(t) - \varepsilon/4}
\end{align*}

Because $a(t)- \varepsilon/4 \leq b(t) - \varepsilon/4$ on $(0,t_0)$, similarly as in Step 1 we obtain $f_\varepsilon \leq \phi $ on $(0,t_0)$. The claim follows by taking the limit $\varepsilon \to 0^+$.

\vspace{10pt}
\noindent\emph{Step 3. The case $\liminf_{t \to 0^+} (b(t)-a(t))=0$: $f(t) \leq \phi (t) $ for all $t\in (0,T)$}.

Let
\[
I_1 := \{ t\in (0,T) \, : \, f(s) \leq \phi (s) \text{ for } s\in (0,t) \}.
\]
Let $t_1 \coloneqq \sup I_1$. Note that Step 2 gives $t_1\geq t_0>0$. Suppose that $t_1 <T$. Let $F(t)\coloneqq f(t_1+t)$, $\Phi (t)\coloneqq \phi (t_1+t)$. Then $F$, $\Phi$ satisfy
\[ \begin{split}
F(t) &\leq  \int_0^t g(t-s) F(s)^2 \,\d s+A(t) ,\\
\Phi (t) &\geq  \int_0^t g(t-s) \Phi (s)^2 \,\d s +B(t),
\end{split} \]
where 
\[A(t)\coloneqq a(t_1+t)+ \int_0^{t_1} g(t_1+t-s) f(s)^2\,\d s,\]
\begin{multline*}
B(t)\coloneqq b(t_1+t)+\int_0^{t_1} g(t_1+t-s) \phi (s)^2\,\d s \\
\geq A(t) + b(t_1+t)-a(t_1+t).
\end{multline*}
Noting that $A$, $B$ are continuous (by the Dominated Convergence Theorem) and that $A(t)\leq B(t)$ for all $t\in (0,T-t_1)$, we can apply Step 1 and Step 2 to the functions $F$, $\Phi$, to conclude that $F(t)\leq \Phi (t)$ for all $t\in (0,t_2]$ for some $t_2>0$. Thus $f(t)\leq \phi (t)$ for all $t\in [0, t_1+t_2)$, which contradicts the definition of $t_1$.\qedhere

\end{proof}
The above lemma can be modified to fit several other settings.
\begin{corollary}\label{integral_ineq_power1_cor}
Let $g,a,b$ be as in Lemma \ref{lem_integral_ineq} and let $h$ satisfy the same conditions as $g$. Let functions $f,\phi \colon (0,T)\to \RR$ be such that $f$ and $\phi$ are bounded near as $t \to 0^+$ and $\phi$ is continuous and 
\[\begin{cases}
f(t) &\leq \int_0^t g(t-s)h(s)f(s) \,\d s + a(t),\\
\phi (t) &\geq \int_0^t g(t-s)h(s) \phi (s) \,\d s + b(t)
\end{cases}\]
for all $t\in (0,T)$. Then $f\leq \phi$ on $(0,T)$.
\end{corollary}
\begin{proof}
Similar to the proof above.
\end{proof}
\begin{corollary}\label{integral_ineq_min_cor}
Let $a,b>0 $ be such that $a\leq b$ and let $h,g \colon (0,\infty ) \to \RR^+$ be continuous functions such that $g$ is locally integrable on $[0,\infty)$ and $h$ is integrable on $(1,\infty )$. Suppose also that there exist $\tau>0$, $C'>0$ such that
\[\begin{split}
h(s) &\geq C'^2 g (s) \quad \text{ for } s\in (0,\tau ), \\
h(s) &\leq C'^2 g(s) \quad \text{ for } s\in [\tau,\infty ),
\end{split}
\]
and for all $t>0$
\[
C'> \int_0^t \min \left( C'^2 g(t-s) , h(t-s) \right) \,\d s + b.
\]
If $T>0$ and $f$ is a positive function on $(0,T)$ that is bounded near $0$ and
\[
f(t) \leq \int_0^t \min \left( g(t-s) f(s)^2 , h(t-s) \right) \,\d s + a
\]
for $t\in (0,T)$, then $f \leq C'$ on $(0,T)$.
\end{corollary}
\begin{proof} If $t< \tau $ then the minimum under the second last integral is $(C')^2 g(t-s)$ and so Lemma \ref{lem_integral_ineq_app} gives $f(t) \leq C'$ for such $t$'s. Thus letting
\[ t_0 \coloneqq \sup \{ t'>0 \colon f(t) \leq C' \text{ for } t<t'\} \]
we see that $t_0 \geq \tau >0$. If $t_0 <T $ we obtain
\begin{multline*}
\int_0^{t_0 }  \min \left( g(t_0-s) f(s)^2 , h(t_0-s) \right) \,\d s +  a \\
\leq \int_0^{t_0 }  \min \left( g(t_0-s) (C')^2 , h(t_0-s) \right) \,\d s +  b < C' .
\end{multline*}
Thus, by continuity
\[
\int_0^{t }  \min \left( g(t-s) f(s)^2 , h(t-s) \right) \,\d s +  a  < C'
\]
for $t\in [t_0, t_0+\delta )$ for some $\delta >0$, which contradicts the definition of $t_0$. Thus $t_0=T$, as required.
\end{proof}

\subsection{The Volterra equation}\label{volterra_eq_section}
In this section we show that the equation
\begin{equation}\label{volterra}
\phi (x) = C \int_0^x \frac{\phi (y)}{\sqrt{x-y}} \,\d y +D,
\end{equation}
where $C,D>0$, has a unique solution $\phi \in C[0,\infty )$. This is equivalent to showing that \eqref{volterra} has a unique solution $\phi \in C[0,T]$ for every $T>0$. We can rewrite \eqref{volterra} in the form
\begin{equation}\label{volterra1}
\phi - A\phi = D,
\end{equation}
where
\eqnb\label{def_of_A_Volterra}
A\phi(x) := \int_0^T K(x,y) \phi(y) \, \d y
\eqne
with $K(x,y) := C\,\chi_{\{ y<x\}} (x-y)^{-1/2}$. This is an example of the Volterra integral equation of the 2nd kind with with a weakly singular kernel $K$, that is any $K$ such that for all $x,y\in [0,T]$ with $x\ne y$ $K$ is continuous at $(x,y)$ and $|K(x,y)|\leq M|x-y|^{\alpha-1 }$ for some $\alpha \in (0,1]$, $M>0$. In what follows we apply the theory of such equations to \eqref{volterra1}. We consider only the case $\alpha =1/2$; other cases follow similarly. We follow the arguments from \cite{kress}.

First note that the set of compact operators is closed in the operator norm.
\begin{lemma}\label{conv_operators_lemma}
Let $X$ be a Banach space and $A_n \in L(X) $ be a sequence of compact operators such that $\| A_n - A \| \to 0$ as $n\to \infty$ for some $A\in L(X)$. Then $A$ is compact. 
\end{lemma}
This is elementary (see e.g. \cite{kress}, p. 26, for the proof). 
We now show that $A$ (defined by \eqref{def_of_A_Volterra}) is a compact operator on $X:= C[0,T]$ by cutting off the singularity and using the above lemma.
\begin{lemma}
The operator $A:X\to X$ is compact.
\end{lemma}
\begin{proof}
We see that $A$ is continuous by writing
\begin{multline*}
|A\phi (t)| = C \left| \int_0^t  \frac{\phi (s) }{\sqrt{t-s}} \,\d s \right| \\\leq C \| \phi \|_{\sup} \int_0^t (t-s)^{-1/2} \,\d s = 2C \| \phi \|_{\sup} \, t^{1/2}  
\leq 2 C T^{1/2} \| \phi \|_{\sup} .
\end{multline*}
For $n\in \NN$ we define a cut-off $K_n$ of the kernel $K$ by 
\[
K_n(t,s) := \begin{cases}
h (n|t-s|) K(t,s) \quad &{t\ne s} \\
0 & t=s,
\end{cases}
\]
where $h: [0,\infty ) \to [0,1]$ is a continuous function such that $h(t) =0$ for $t \in [0,1/2]$ and $h(t)=1$ for $t\geq 1$. Because $K_n \in C([0,T]^2)$ for every $n$, the corresponding integral operators $A_n$ are compact by the Arzel\`{a}-Ascoli theorem. Moreover
\[
\begin{split}
|A\phi (t) - A_n \phi (t) | &= \left| \int_{t-1/n}^t (1-h(n|t-s|)) K(t,s) \phi(s) \,\d s \right| \\
&\leq \| \phi \|_{\sup} \int_{t-1/n}^t  \frac{C}{\sqrt{t-s}}  \,\d s = C  \| \phi \|_{\sup} \,\,n^{-1/2} \to 0
\end{split}
\]
as $n\to \infty$ uniformly in $t$. Hence $\| A_n-A\|\to 0$ and Lemma \ref{conv_operators_lemma} gives compactness of $A$.
\end{proof}
We now show the claim by applying Fredholm Alternative.
\begin{theorem}\label{thm_volterra}
The equation \eqref{volterra1} has a unique solution $\phi \in C[0,T]$.
\end{theorem}
\begin{proof}
Because $X$ is a Banach space and $A:X\to X$ is compact we can apply Fredholm Alternative to conclude that \eqref{volterra1} has a unique solution if the equation 
\begin{equation}\label{volterra2}
\phi-A\phi =0
\end{equation}
has no non-zero solution. We will use induction to show that a solution $\phi$ to this homogeneous problem satisfies 
\begin{equation}\label{volterra_induction}
|\phi (t) | \leq  \| \phi \|_{\sup } \frac{M^k t^k }{k!}
\end{equation}
for some $M>0$ and all $t\in [0,T]$, $k=0,1,\ldots$. The base case $k=0$ is trivial. For the inductive step we first note that that for any $t,s\in [0,T]$ with $0\leq s <t\leq T$ we have
\[
\int_s^t \frac{\d \tau }{\sqrt{t-\tau}\sqrt{\tau -s}} = \int_0^1 \frac{\d z}{\sqrt{z(1-z)}} =:I
\]
by the change of variable $z:= \frac{\tau -s}{t-s}$. Now assume that a solution $\phi$ to \eqref{volterra2} satisfies $|\phi (t) | \leq M \| \phi \|_{\sup } \frac{C^k t^k }{k!}$ for some $k$. Then, because $\phi = A\phi = A^2 \phi$, we have for all $t\in (0,T]$
\begin{align*}
\MoveEqLeft[0]
|\phi(t) | = |A^2 \phi (t)| = C^2 \left| \int_0^t \int_0^\tau \frac{\phi (s) }{\sqrt{\tau - s}\sqrt{t-\tau }} \,\d s \,\d \tau \right| \\
&\leq C^2  \int_0^t \int_0^\tau \frac{|\phi (s)| }{\sqrt{\tau - s}\sqrt{t-\tau }} \,\d s \,\d \tau =C^2  \int_0^t \int_s^t \frac{|\phi (s)| }{\sqrt{\tau - s}\sqrt{t-\tau }}  \,\d \tau \,\d s  \\
&= I C^2  \int_0^t |\phi(s) | \,\d s \leq  \| \phi \|_{\sup } \frac{I C^2 M^k} {k!} \int_0^t s^k \,\d s=\| \phi \|_{\sup }  \frac{I C^{2} M^k  t^{k+1}}{(k+1)!} ,
\end{align*}
where we used Fubini's theorem. The bound \eqref{volterra_induction} now follows with $M:= I C^2$. Taking the limit $k\to \infty $ in \eqref{volterra_induction} gives $\phi \equiv 0$. 
\end{proof}
\subsection{A proof of \eqref{bound_on_p_magic} without the use of the Plancherel Lemma}\label{another_proof_of_magic_bound_sec}
Here we give an elementary proof of \eqref{bound_on_p_magic},
\[
\| p \|^2 \leq C \| \nabla u \|^3 \| u \|.
\]
First note that the representation formula for $p$ (see \eqref{sol_regularisedNSE_form}) can be rewritten in the form
\[
p= \frac{1}{4\pi } \int \nabla \left( \frac{1}{|x-y|}\right) \cdot g(y) \, \d y,
\]
where $g \coloneqq (J_\varepsilon u \cdot \nabla )u$. Therefore
\[
\begin{split}
\| p \|^2 &= \frac{1}{(4\pi )^2}\iiint  \left[ \nabla \frac{1}{|x-y|} \cdot g(y) \right]  \left[ \nabla \frac{1}{|x-z|} \cdot g(z) \right] \,\d x\, \d y\, \d z \\
&=  \frac{-1}{(4\pi )^2}\iiint  \left[  \frac{1}{|x-y|} \mathrm{div}\, g(y) \right]  \left[ \nabla \frac{1}{|x-z|} \cdot g(z) \right] \,\d x \,\d y\, \d z 
\end{split}
\]
Since $\nabla_x |x-y|^{-1} =\nabla_y |x-y|^{-1} $, integration by parts in $x$ and then in $y$ gives
\[
\| p \|^2 = \frac{-1}{(4\pi )^2}\iiint  \left[  \frac{1}{|x-y|} \nabla ( \mathrm{div}\, g(y)) \right] \cdot  \frac{g(z)}{|x-z|}   \,\d x\, \d y \, \d z .
\]
Now the calculus identity
\[
\nabla (\mathrm{div} \, g ) = \Delta g + \mathrm{curl} (\mathrm{curl} \, g )
\]
gives
\begin{equation}\label{p_magic_temp}
\begin{split}
\| p \|^2 &= \frac{1}{4\pi}\iint   \frac{g(x) \cdot g(z)}{|x-z|}    \,\d x \, \d z \\
 &-\frac{1}{(4\pi )^2}\iiint  \left[  \frac{1}{|x-y|} \mathrm{curl} ( \mathrm{curl}\, g(y)) \right] \cdot  \frac{g(z)}{|x-z|}   \,\d x\, \d y \, \d z .
\end{split}
\end{equation}
Since the $i$th component of $\mathrm{curl}\, g $ can be expressed in the form $(\mathrm{curl}\, g )_i = \epsilon_{ijk} \p_j g_k$, where the coefficients 
\[
\epsilon_{ijk} \coloneqq \begin{cases} 1 \qquad &\text{ if } ijk \text{ is an even permutation of } 123,\\
-1 &\text{ if } ijk \text{ is an odd permutation of }123,\\
0 &\text{ otherwise }
\end{cases}
\]
satisfy $\epsilon_{ijk}=-\epsilon_{kji}$, we can write the last triple integral as
\[
\iiint  \frac{ \epsilon_{ijk}\p_j  \left(\mathrm{curl}\, g(y)\right)_k  g_i (z)}{|x-y|\, |x-z|} \,\d x \, \d y \, \d z,
 \]
which integrated by parts three times (first in $y$ then in $x$ and $z$) gives
\begin{align*}
\MoveEqLeft
-\iiint  \frac{ \epsilon_{ijk} \left( \mathrm{curl}\, g(y)\right)_k \p_j g_i (z)}{|x-y|\, |x-z|} \,\d x \, \d y \, \d z\\
&=\iiint  \frac{  \left( \mathrm{curl}\, g(y)\right)_k \epsilon_{kji} \p_j g_i (z)}{|x-y|\, |x-z|} \,\d x \, \d y \, \d z\\
&= \iiint  \frac{  \left( \mathrm{curl}\, g(y)\right) \cdot \left( \mathrm{curl}\, g(z)\right)}{|x-y|\, |x-z|} \,\d x \, \d y \, \d z= \int |F(x)|^2 \,\d x \geq 0,
\end{align*}
where $F(x) \coloneqq \int \mathrm{curl}\, g(y) /|x-y| \,\d y $. Therefore \eqref{p_magic_temp} gives
\begin{align*}
\| p \|^2 &\leq  \frac{1}{4\pi}\iint   \frac{g(x) \cdot g(z)}{|x-z|}    \,\d x \, \d z  \\
&=  \frac{1}{4\pi}\iint  \frac{(J_\varepsilon u_k (x))  \p_k u_i (x)  (J_\varepsilon u_j(z))  \p_j u_i(z)}{|x-z|}    \,\d x \, \d z 
\end{align*}
(this inequality appears in \cite{Leray_1934} on p. 233 with ``$\leq$'' wrongly replaced by ``$=$''),
from where the Cauchy-Schwarz inequality and Lemma \ref{temp_singular_int_bound} give
\[\begin{split}
\| p \|^2 &\leq  \frac{1}{4\pi}\int \| \nabla u \| \left(   \int \frac{|J_\varepsilon u_k (x)|^2  }{|x-z|^2}    \,\d x \right)^{1/2} \, |J_\varepsilon u_j(z)| \, |  \nabla u (z) | \,\d z \\
& \leq C \| \nabla u \|^2 \int |J_\varepsilon u_j(z)| \, |  \nabla u (z) | \,\d z \leq C \| \nabla u \|^3 \| u \| ,
\end{split}
\]
where we also used $\| J_\varepsilon (\nabla u ) \|\leq \|\nabla u \|$, $\| J_\varepsilon u \|\leq \|u \|$ (see Lemma \ref{prop_molli} (i)).
\subsection{Smooth approximation of the forcing}\label{appendix_approx_of_X}
\begin{lemma}[Dini's lemma]\label{dini_lemma_fact}
Let $I$ be a compact interval and let $f_n,f \in C(I;\RR)$ be continuous functions such that $f_n (t) \to f (t)$ as $n\to \infty$ and $f_{n+1}(t) \leq f_n (t)$ for each $t\in I$. Then $\| f_n - f \|_{C(I)} \to 0$. 
\end{lemma}
\begin{proof}
This is elementary.
\end{proof}
\begin{lemma}%[Approximation of $F\in C( {[} 0,T' {]} , L^2)$ by a smooth $\widetilde{F}$ that is compactly supported in space] 
\label{approx_of_X_lemma}
Let $p\in[1,\infty)$ and  $F\in C([0,T), L^p)$. For any $T'\in(0,T)$ and any $\varepsilon >0$ there exists $\widetilde{F} \in C^\infty ( \RR^3  \times \RR )$ and $R>0$ such that $\mathrm{supp} \, \widetilde{F} (t) \subset B(0,R)$ for all $t\in \RR$ and
\begin{equation}\label{approx_of_X_lemma_ineq}
\| F - \widetilde{F} \|_{C([0,T'],L^2)} < \varepsilon.
\end{equation}
Moreover $\max_{t\in[0,T']} \| \widetilde{F} (t)\|_\infty \leq \max_{t\in[0,T']} \| {F} (t)\|_\infty $.% whenever the right-hand side is finite.
\end{lemma}
\begin{proof}
It suffices to consider $T<\infty$. First extend $F$ in time from $[0,T']$ to the whole line by taking $F(\cdot ,t) \coloneqq F(\cdot , 0)$ for $t<0$ and $F(\cdot , t) \coloneqq F(\cdot ,T')$ for $t>T'$.
For $R>0$ let 
\[
F_R (x,t) \coloneqq \chi_{B(0,R)} (x) F(x,t).
\]
Clearly $F_R\in C(\RR ; L^p)$ (as a product of two such functions) and so $\| F_R (t) - F(t) \|_{p} $ is continuous in $t$ for each $R$. Hence, noting that $\| F_R (t) - F(t) \|_{p} $ is a nonincreasing function of $R$ converging to zero as $R\to \infty$ for each $t\in [0,T']$, we can use Dini's Lemma (Lemma \ref{dini_lemma_fact}) to fix $R>0$ such that
\begin{equation}\label{dini_gives_this}
\| F_R - F \|_{C([0,T'];L^p)} < \varepsilon/3.
\end{equation}
We will now mollify $F_R$ to obtain $\widetilde{F}$. Let $\eta_\delta$, $\xi_\delta$ be mollifiers in $\RR^3$ and $\RR$ respectively, that is let $\xi (t) \coloneqq C \exp (1/(|t|^2-1))$ for $t\in (0,1)$ and $\xi (t) \coloneqq 0$ if $t\not \in (0,1)$, where $C$ is such that $\int_\RR \xi =1$, and let $\xi_\delta (t) \coloneqq \xi(t/\delta )/\delta$, $\eta (x) \coloneqq \xi (|x|)$, $\eta_\delta (x) \coloneqq \eta (x/\delta )/ \delta^3$. Define the mollification $F_R^\delta $ of $F_R$ by
\begin{multline*}
F_R^\delta (x,t) \coloneqq \int_\RR \xi_\delta (s) \int \eta_\delta (y) F_R (x-y, t-s) \,\d y \,\d s\\
=\int_\RR \xi (s) \int \eta (y) F_R (x-\delta y,t-\delta s) \,\d y \,\d s.
\end{multline*}
Clearly $\widetilde{F} \coloneqq F_R^\delta (x,t)$ has the required regularity for each $\delta >0$ (in particular $\max_{t\in[0,T']} \| \widetilde{F} (t)\|_\infty \leq \max_{t\in[0,T']} \| {F} (t)\|_\infty $ holds by the property of mollifiers (see (i) in Lemma \ref{prop_molli})). Therefore the proof will be complete if we show the approximation property \eqref{approx_of_X_lemma_ineq} for some $\delta >0$. 
Because $F_R (x,t)- F_R^\delta (x,t) = \int_\RR \xi (s) \int \eta (y) \left( F_R (x,t) -F_R (x-\delta y,t-\delta s)\right) \,\d y \,\d s$, we can use the Minkowski inequality (see \ref{minkowski_integral_version_more_general}) and the triangle inequality 
\begin{multline*}
\| F_R(x-\delta y ,t-\delta s) - F_R (x,t)\|_{p} \leq \| F_R(x-\delta y ,t-\delta s) - F_R (x-\delta y,t)\|_{p}\\
 + \| F_R(x-\delta y ,t) - F_R (x,t)\|_{p}  
\end{multline*}
to write
\begin{equation}\label{approx_of_X_temp1}
\begin{split}
\| F_R (t) - F_R^\delta (t) \|_{p} &\leq \int_\RR \xi (s) \int \eta (y) \| F_R (\cdot -\delta y,t-\delta s) - F_R (\cdot ,t) \|_{p} \,\d y \,\d s\\
&\leq \int_\RR \xi (s) \| F_R(t-\delta s) - F_R (t)\|_{p} \,\d s \\
&\quad+ \int \eta(y) \| F_R(\cdot -\delta y ,t) - F_R (\cdot ,t)\|_{p}   \,\d y.
\end{split}
\end{equation}
Since $F_R$ is uniformly continuous on $[-1,T'+1]$ into $L^p$ (recall we extended $F$ in time to the whole line) there exists sufficiently small $\delta_1>0$ such that for $\delta \in (0,\delta_1)$
\[
\| F_R(t-\delta s) - F_R (t)\|_{p} < \varepsilon /3
\]
for all $t\in [0,T']$, and so the first term on the right-hand side of \eqref{approx_of_X_temp1} is less than $\varepsilon /3$ for all $t\in [0,T']$. As for the second term there exists $\delta>0$ such that $\delta < \delta_1$ and
\[
\| F_R (\cdot - z ,t)-F_R (\cdot , t) \|_{p} < \varepsilon /3
\]
whenever $|z| < \delta $ and $t\in [0,T']$. Indeed, by the continuity of translation in space of $L^p$ functions for each $t_0\in [0,T']$ there exists a $\delta_{t_0}$ such that $\| F_R (\cdot - z ,t_0)-F_R (\cdot , t_0) \|_{p} < \varepsilon /3$ whenever $|z|<\delta_{t_0}$. Moreover, the continuity of $F_R$ in time into $L^p$ and triangle inequality gives that $\| F_R (\cdot - z ,t)-X_R (\cdot , t) \|_p < \varepsilon /3$ whenever $|z|<\delta_{t_0}$ and $t$ belongs to some open set $J_{t_0}$ containing $t_0$. By compactness of $[0,T']$ we obtain a finite cover $\{ J_{t_i} \}_{i=1,\ldots , m}$ of $[0,T']$ consisting of such open sets and $\delta$ is obtained by taking the minimum of the corresponding $\delta_{t_i}$, $i=1,\ldots ,m$. This means in particular that
\[
\| F_R (\cdot - \delta y, t) - F_R (\cdot , t) \|_{p} < \varepsilon /3
\]
whenever $|y|<1$, that is for all $y\in \supp \, \eta $. Therefore, for such $\delta$ the second term on the right-hand side of \eqref{approx_of_X_temp1} is bounded by $\varepsilon /3$ uniformly in $t\in[0,T']$. The approximation property \eqref{approx_of_X_lemma_ineq} therefore follows directly from \eqref{dini_gives_this} and \eqref{approx_of_X_temp1}.
\end{proof}

%************************************************************
%Bibliography
%************************************************************
\bibliography{Bib_LerayNotes}{}
\end{document}